\documentclass[leqno,11pt]{amsart}

\usepackage[letterpaper,margin=1in]{geometry}
\usepackage{arydshln}
\usepackage[all]{xy} % diagrams
\usepackage{amsmath, amssymb, amsfonts, latexsym, mdwlist, amsthm, amscd,mathabx,braket}
\usepackage{subfig}
\usepackage{graphicx}
\usepackage{wrapfig}
\usepackage{etex}
\usepackage{tikz-cd}
\usepackage[colorlinks=true, citecolor=blue, urlcolor=blue, linkcolor=blue, pagebackref]{hyperref}
\usepackage{tikz}
\usetikzlibrary{calc,trees,positioning,arrows,chains,shapes.geometric,%
    decorations.pathreplacing,decorations.pathmorphing,shapes,%
    matrix,shapes.symbols}

\tikzset{
>=stealth',
  punktchain/.style={
    rectangle,
    rounded corners,
    draw=black, thick,
    minimum height=3em,
    text centered,
    on chain},
  line/.style={draw, thick, <-},
  eLement/.style={
    tape,
    top color=white,
    bottom color=blue!50!black!60!,
    minimum width=8em,
    draw=blue!40!black!90, very thick,
    text width=10em,
    minimum height=3.5em,
    text centered,
    on chain},
  every join/.style={->, thick,shorten >=1pt},
  decoration={brace},
  tuborg/.style={decorate},
  tubnode/.style={midway, right=2pt},
}
\usepackage{paralist}
\usepackage{enumitem} 
\setdefaultenum{(1)}{(a)}{(i)}{}
\setlist[enumerate,1]{label={\upshape(\arabic*)}}
\setlist[enumerate,2]{label={\upshape(\alph*)},ref=\theenumi\upshape(\alph*)}
\setlist[enumerate,3]{label={\upshape(\roman*)},ref=\theenumi\theenumii\upshape(\roman*)}
\usepackage[capitalize]{cleveref}
\crefname{Prop}{Proposition}{Propositions}
\crefname{Thm}{Theorem}{Theorems}
\crefname{Lem}{Lemma}{Lemmas}
\crefname{enumi}{Case}{Cases}

\def\C{\ensuremath{\mathbb{C}}}

\def\H{\ensuremath{\mathbb{H}}}
\def\N{\ensuremath{\mathbb{N}}}
\def\P{\ensuremath{\mathbb{P}}}
\def\Q{\ensuremath{\mathbb{Q}}}
\def\R{\ensuremath{\mathbb{R}}}

\def\Z{\ensuremath{\mathbb{Z}}}

\def\alg{\mathrm{alg}}
\def\Amp{\mathrm{Amp}}

\def\ch{\mathop{\mathrm{ch}}\nolimits}

\def\Coh{\mathop{\mathrm{Coh}}\nolimits}

\def\dim{\mathop{\mathrm{dim}}\nolimits}

\def\Ext{\mathop{\mathrm{Ext}}\nolimits}
\def\ext{\mathop{\mathrm{ext}}\nolimits}
\def\lExt{\mathop{\mathcal Ext}\nolimits} % means local Ext

\def\GL{\mathop{\mathrm{GL}}\nolimits}
\def\Hal{H^*_{\alg}}
\def\Hilb{\mathop{\mathrm{Hilb}}\nolimits}

\def\Hom{\mathop{\mathrm{Hom}}\nolimits}

\def\id{\mathop{\mathrm{id}}\nolimits}

\def\Ker{\mathop{\mathrm{Ker}}\nolimits}

\def\min{\mathop{\mathrm{min}}\nolimits}

\def\NS{\mathop{\mathrm{NS}}\nolimits}

\def\Pic{\mathop{\mathrm{Pic}}\nolimits}

\def\rk{\mathop{\mathrm{rk}}}

\def\td{\mathop{\mathrm{td}}\nolimits}

\def\Cone{\mathop{\mathrm{Cone}}\nolimits}

\newenvironment{Prf}{\textit{Proof.}\/}{\hfill$\Box$}

\def\MG13{\ensuremath{{\mathcal M}_{\Gamma_1(3)}}}
\def\tildeMG13{\ensuremath{\widetilde{\mathcal M}_{\Gamma_1(3)}}}
\def\Stab{\mathop{\mathrm{Stab}}\nolimits}
\def\Stabd{\mathop{\Stab^{\dagger}}\nolimits}
\def\into{\ensuremath{\hookrightarrow}}
\def\onto{\ensuremath{\twoheadrightarrow}}

\def\blank{\underline{\hphantom{A}}}

\def\Db{\mathrm{D}^{\mathrm{b}}}

\newcommand\TFILTB[3]{%
\xymatrix@=1pc{
{0 = {#1}_0} \ar[rr]&&
{{#1}_1} \ar[rr]\ar[ld] &&
{{#1}_2} \ar[r]\ar[ld] &
{\cdots} \ar[r] & { {#1}_{#3-1}} \ar[rr] &&
{{#1}_{#3} = {#1}} \ar[ld]
\\
& *{{#2}_1} \ar@{.>}[ul] &&
{{#2}_2} \ar@{.>}[ul] & &&&
{{#2}_{{#3}}} \ar@{.>}[ul]
}}

\def\abs#1{\left\lvert#1\right\rvert}

\makeatletter
\newtheorem*{rep@theorem}{\rep@title}
\newcommand{\newreptheorem}[2]{%
\newenvironment{rep#1}[1]{%
 \def\rep@title{#2 \ref{##1}}%
 \begin{rep@theorem}}%
 {\end{rep@theorem}}}
\makeatother

\newtheorem{Thm}{Theorem}[section]
\newreptheorem{Thm}{Theorem}
\newtheorem{Prop}[Thm]{Proposition}

\newtheorem{Lem}[Thm]{Lemma}

\newtheorem{Cor}[Thm]{Corollary}
\newreptheorem{Cor}{Corollary}

\newreptheorem{Con}{Conjecture}

\newtheorem*{theorem*}{Theorem}
\newtheorem*{lemma*}{Lemma}
\newtheorem*{proposition*}{Proposition}
\newtheorem*{conjecture*}{Conjecture}
\newtheorem*{corollary*}{Corollary}
\newtheorem*{problem*}{Problem}

\newtheorem{Thm-int}{Theorem}

\theoremstyle{definition}
\newtheorem{Def-s}[Thm]{Definition}
\newtheorem{Def}[Thm]{Definition}
\newtheorem{Rem}[Thm]{Remark}

\newtheorem{Prob}[Thm]{Problem}
\newtheorem{Ex}[Thm]{Example}

\newenvironment{NB}{
\color{red}{\bf NB}. \footnotesize 
}{}
\def\C{\ensuremath{\mathbb{C}}}

\def\H{\ensuremath{\mathbb{H}}}
\def\N{\ensuremath{\mathbb{N}}}
\def\P{\ensuremath{\mathbb{P}}}
\def\Q{\ensuremath{\mathbb{Q}}}
\def\R{\ensuremath{\mathbb{R}}}
\def\Z{\ensuremath{\mathbb{Z}}}

\def\AA{\ensuremath{\mathcal A}}

\def\CC{\ensuremath{\mathcal C}}

\def\EE{\ensuremath{\mathcal E}}
\def\FF{\ensuremath{\mathcal F}}
\def\GG{\ensuremath{\mathcal G}}
\def\HH{\ensuremath{\mathcal H}}

\def\MM{\ensuremath{\mathcal M}}

\def\OO{\ensuremath{\mathcal O}}

\def\TT{\ensuremath{\mathcal T}}

\def\WW{\ensuremath{\mathcal W}}

\newcommand{\mor}[1][]{\xrightarrow{#1}}
\newcommand{\isomor}{\mor[\sim]}

\def\u{\mathbf{u}}
\def\v{\mathbf{v}}
\def\w{\mathbf{w}}

\def\ii{\mathfrak{i}}

\def\HHH{\mathfrak H}

\def\cal{\mathcal}
\def\Bbb{\mathbb}

\newcommand{\ignore}[1]{}
\begin{document}

\title{The cohomology of the general stable sheaf  on a K3 surface}
\author{Izzet Coskun}
\author{Howard Nuer}
\author{K\={o}ta Yoshioka}

\thanks{During the preparation of this article, I.C. was partially supported by NSF grant DMS 1500031 and
NSF FRG grant DMS 1664296, H.N. was partially supported an NSF Postdoctoral Fellowship DMS 1606283 and the NSF RTG grant 1246844 and K.Y. was partially supported by Grant in Aid for Scientific Research No. 18H01113, 17H06127, 26287007 JSPS}

\begin{abstract}
Let $X$ be a K3 surface with  Picard group $\Pic(X)\cong\Z H$ such that $H^2=2n$. Let $M_{H}(\v)$ be the moduli space of Gieseker semistable sheaves on $X$ with Mukai vector $\v$. We say that $\v$ satisfies weak Brill-Noether if  the general sheaf in  $M_{H}(\v)$ has at most one nonzero cohomology group. We show that given any rank $r \geq 2$, there are only finitely many Mukai vectors of rank $r$ on K3 surfaces of Picard rank one where weak Brill-Noether fails. We give an algorithm for finding the potential counterexamples and classify all such counterexamples up to rank 20 explicitly. Moreover, in each of these cases we calculate the cohomology of the general sheaf.  Given $r$, we give sharp bounds on $n$, $d$, and $a$ that guarantee that $\v$ satisfies weak Brill-Noether.  As a corollary, we obtain another proof of the classification of Ulrich bundles on K3 surfaces of Picard rank one. In addition, we discuss the question of when the general sheaf in $M_H(\v)$ is globally generated. Our techniques make crucial use of Bridgeland stability conditions.    
    
\end{abstract}

\maketitle
\section{Introduction}

Brill-Noether theory for line bundles on curves has played a central role in developing algebraic geometry since the 19th century (see \cite{ACGH}). However much less is known about Brill-Noether theory for higher rank vector bundles on curves. For vector bundles on higher dimensional varieties, even the first step of Brill-theory, computing the cohomology of the generic vector bundle in a moduli space, is extremely challenging.  Indeed, this step has been carried out in full for very few surfaces, such as minimal rational surfaces and certain del Pezzo surfaces \cite{ CoskunHuizenga:WBN, CoskunHuizenga:GloballyGenerated, GottscheHirschowitz:WBN, LS19}. 

The problem of computing the generic cohomology of stable sheaves underlies many of the fundamental problems in the field, ranging from the construction of theta divisors and Ulrich bundles to classifying Chern characters of stable bundles and understanding the birational geometry of the moduli space of sheaves (see \cite{CoskunHuizenga:ICM, CHW14,OG97}). Moreover, vanishing of higher cohomology and global generation play key roles in the S-duality conjecture (see \cite{MarianOprea:StrangeDuality}). In this paper we undertake the problem of computing the  cohomology of a generic sheaf in a moduli space of stable sheaves on a K3 surface of Picard rank one.  The Brill-Noether theory for K3 surfaces has been investigated by numerous authors (see for example \cite{Leyenson:BN,Markman:BN,Yos99b}).

\subsection*{Weak Brill-Noether} Let $X$ be a K3 surface such that $\Pic(X) \cong \Z H$ with $H^2 = 2n$. Let $\v$ be a Mukai vector with $\v^2 \geq -2$ and let $M_{H}(\v)$ denote the moduli space parameterizing $S$-equivalence classes of Gieseker semistable sheaves on $X$ with Mukai vector $\v$. 

We say that  $\v$ satisfies {\em weak Brill-Noether} if the general sheaf $E \in M_{H}(\v)$ has at most one nonzero cohomology group. If $\v$ satisfies weak Brill-Noether, then the Euler characteristic and the slope completely determine the cohomology of the general sheaf in $M_{H}(\v)$. In this paper, we study the problem of characterizing the Mukai vectors $\v$ that satisfy weak Brill-Noether.  Our main qualitative result is the following.

\begin{Thm}\label{thm-intro}
Let $X$ be a K3 surface such that $\Pic(X)\cong \mathbb{Z}H$ with $H^2=2n$. Let $\v=(r,dH,a)$ be a Mukai vector with $\v^2\geq -2$, $r\geq 2$ and $d>0$.
\begin{enumerate}
\item For each $r\geq 2$, there exists a finite set of tuples $(n,r,d,a)$ for which $\v$ fails to satisfy weak Brill-Noether (\cref{thm:finiteness}). 
\item If $n\geq r$, then $\v$ satisfies weak Brill-Noether (\cref{thm:bound n r}).
\item If $a \leq 1$, then $\v$ satisfies weak Brill-Noether (\cref{Prop:a<=0}, \cref{lem:NoTSSWalls a<=1}).
    \item  If $d \geq r \left\lfloor \frac{r}{n}\right\rfloor +2$, then $\v$ satisfies weak Brill-Noether (\cref{thm:WBN sharp}).
    
    \end{enumerate}
\end{Thm}

\begin{Rem}
Let $E$ be a stable locally free sheaf on a K3 surface of Picard rank one. By Serre duality, to compute the cohomology of $E$, we may assume that $d \geq 0$. Moreover, if $d=0$ and $E$ has a section, then $E \cong \OO_X$. Hence, it suffices to consider Mukai vectors with $d>0$. In this case, stability implies that $H^2(X, E)=0$, so we need to compute $H^0(X,E)$ and $H^1(X,E)$. 
\end{Rem}

\begin{Ex}\label{Ex:PullbackTangentBundle}
The linear system $|H|$ defines a morphism $f: X \to \mathbb{P}^{n+1}$. The pullback of the tangent bundle $f^* T \mathbb{P}^{n+1}$ is a spherical stable bundle on $X$ with Mukai vector $$\v_0= (n+1, (n+2)H, n^2+3n+1).$$ The pullback of the Euler sequence $$0\to \OO_X \to \OO_X(H)^{\oplus(n+2)} \to  f^* T \mathbb{P}^{n+1} \to 0$$ shows that $h^0(f^*T \mathbb{P}^{n+1}) = n^2 + 4n +3$ and $h^1(f^* T \mathbb{P}^{n+1})=1$. Since $f^* T \mathbb{P}^{n+1}$ is the unique point in its moduli space, $\v_0$ fails weak Brill-Noether. Consequently, parts (2) and (4) of \cref{thm-intro} are sharp. 
\end{Ex}

In \cref{sec-initialclassification}, we classify the boundary cases in \cref{thm-intro}. In Theorems \ref{Thm:n<r<=2n} and \ref{Thm:2n<r<=3n}, we classify the Mukai vectors $\v=(r, dH, a)$ with $n < r \leq 3n$ such that $\v$ fails weak Brill-Noether and we compute the cohomology of the general sheaf for these $\v$. Our main result in this direction is the following.

\begin{Thm}\label{Thm:intro:n<r<=3n}[Theorems \ref{Thm:n<r<=2n} and \ref{Thm:2n<r<=3n}]
Let $X$ be a K3 surface such that  $\Pic(X) = \mathbb{Z}H$ with $H^2=2n$.
Let $\v=(r,dH,a)$ be a Mukai vector such that $n<r\leq3n$ and $d>0$.  Then  $\v$ fails to satisfy weak Brill-Noether if and only if $\v$ belongs to one of the following three cases:
\begin{enumerate}
\item $\v=(n+r_1^2,((\frac{n+1}{r_1})+r_1)H,(\frac{n+1}{r_1})^2+n)$, where $r_1\mid n+1$ and $1 \leq r_1 \leq\sqrt{2n}$;
    \item $\v=(r,(r+1)H,nr+2n)$ with $2n<r\leq 3n$;
     \item
    $\v=(3n,(3n+2)H,3n^2+4n+1)$ with $n>1$.
\end{enumerate}
\end{Thm}

In  \cref{Lem:NoTSSWalls a=2}, we classify all the Mukai vectors with $a=2$ for which  $\v$ fails weak Brill-Noether. We find that there is a unique pair $(n,\v)$, with $n=1$ and $\v=(5,3H,2)$, which fails to satisfy weak Brill-Noether. Hence, part (3) of \cref{thm-intro} is also sharp.

More importantly, given a rank $r$, we provide an easy-to-implement,  purely-numerical algorithm for enumerating a finite set of Mukai vectors that contain all the Mukai vectors $\v$ of rank $r$ that do not satisfy weak Brill-Noether. We do this by means of \cref{thm:summing up inequalities}, which asserts that if $\v$ fails weak Brill-Noether, then there must exist a spherical character $\v_1$ satisfying certain inequalities with respect to $\v$.  For each rank, a computer can easily list the Mukai vectors $\v$ for which such a $\v_1$ exists. Similarly, for a  given Mukai vector $\v$, it is easy to verify whether the conditions of \cref{thm:summing up inequalities} are satisfied. 

Our algorithm then provides a canonical resolution of the general sheaf in $M_H(\v)$ for each  Mukai vector that fails weak Brill-Noether. In the hundreds of examples we have studied, this resolution allows one to compute the cohomology of the general sheaf in $M_H(\v)$. In  \cref{sec-rank20}, we list all the pairs $(n,\v)$ where $\v$ is a Mukai vector of rank at most 20 failing weak Brill-Noether on Picard rank one K3 surface of degree $2n$, and in each case, we compute the cohomology of the generic sheaf in the corresponding moduli spaces.

\subsection*{Applications to Ulrich bundles} An immediate consequence of \cref{thm-intro} (3) is a classification of stable Ulrich bundles on K3 surfaces of Picard rank 1. The problem of constructing and classifying Ulrich bundles has received a lot of attention in recent years. Aprodu, Farkas and Ortega have constructed Ulrich bundles on K3 surfaces of Picard rank 1  \cite{AFO12}. More generally,  Faenzi \cite{Fae19} has constructed Ulrich bundles on arbitrary K3 surfaces. 

\begin{proposition*}[\ref{prop:UlrichBundles}]
Let $X$ be a K3 surface with $\Pic(X)=\Z H$.  There exists an Ulrich bundle of rank $r$ with respect to $mH$ if and only if $2\mid rm$.  Moreover, when an Ulrich bundle of rank $r$ exists, it has Mukai vector $\v=\left(r,\left(\frac{3rm}{2}\right)H,r(2m^2n-1)\right)$.  In particular, there exists an Ulrich bundle of any rank $r\geq 2$ with respect to $2H$.
\end{proposition*}

\subsection*{Global generation} When $\v$ satisfies weak Brill-Noether, we also study when the general sheaf in $M_H(\v)$ is globally generated. Our qualitative result can be summarized as follows.

\begin{Thm}\label{Thm-introgg}
Let $X$ be a K3 surface such that $\Pic(X)\cong \mathbb{Z}H$ with $H^2=2n$. Let $\v=(r,dH,a)$ be a Mukai vector with $\v^2\geq -2$, $r\geq 2$, $d>0$ and $a\geq 2$.
\begin{enumerate}
\item If $n \geq 2r$, then the generic sheaf in $M_H(\v)$ is globally generated (\cref{prop:uniform gg n}).
\item If $n>1$ and $d \geq r \left\lfloor \frac{2r}{n} \right\rfloor + 2$, then the generic sheaf in $M_H(\v)$ is globally generated (\cref{thm:sufficient condition on d}). 
\end{enumerate}
\end{Thm}

 When $n=1$, $H$ is not very ample, but defines a two-to-one map onto $\mathbb{P}^2$. As a consequence, the twists of the ideal sheaf of a point $I_p(H)$ and $I_p(2H)$ are not globally generated. This complicates the answer when $n=1$. See \cref{thm:sufficient condition on d} for a precise statement in that case.

\begin{Rem}
Being globally generated is not an open condition. However, it is an open condition among sheaves with vanishing higher cohomology. Let $E$ be a sheaf with vanishing higher cohomology. If $a \leq 0$, then $h^0(X,E) \leq r$ and $E \not\cong \OO_X^r$, hence $E$ cannot be globally generated. It is easy to classify $E$ with $a=1$ which are globally generated (see 
\cref{rem-gg}). Hence, we may restrict our attention to Mukai vectors with $a \geq 2$.
\end{Rem}

In parallel to our approach to the weak Brill-Noether property, we give an easy-to-implement, numerical algorithm for checking that the general sheaf in $M_H(\v)$ is globally generated. \cref{thm:summing up inequalities} provides a set of inequalities such that if $\v$ does not satisfy these inequalities and the general sheaf in $M_X(a, dH, r)$ is locally free, then the general sheaf in $M_X(\v)$ is globally generated. These conditions are easy to verify for any given Mukai vector.

\subsection*{Ample bundles} If $E$ is a globally generated vector bundle, then $E(H)$ is globally generated and ample. Therefore, \cref{Thm-introgg} also gives a certificate for the ampleness of bundles on a K3 surface of Picard rank one.  The following is an immediate consequence of \cref{Thm-introgg}.

\begin{Cor}
Let $X$ be a K3 surface such that $\Pic(X)\cong \mathbb{Z}H$ with $H^2=2n$. Let $\v=(r,dH,a)$ be a Mukai vector with $\v^2\geq -2$, $r\geq 2$, $d>0$ and $a\geq 2$. Let $\v_H = (r, (d+r)H, a+(2d+r)n)$.
\begin{enumerate}
\item If $n \geq 2r$, then the generic sheaf in $M_H(\v_H)$ is ample.  
\item If $n>1$ and $d \geq r \left\lfloor \frac{2r}{n} \right\rfloor + 2$, then the generic sheaf in $M_H(\v_H)$ is ample. 
\end{enumerate}
\end{Cor}

\subsection*{The strategy} Let us take a moment to briefly outline our approach, in which Bridgeland stability plays a central role. Let $I_{\Delta}$ denote the ideal sheaf of the diagonal in $X \times X$. We show in \cref{Lem:FM-vanishing} that the dual of the Fourier-Mukai transform $F:= \Phi_{X \to X}^{I_{\Delta}}(E)^{\vee}$ controls the vanishing of cohomology and global generation of $E$. More precisely, we show that if $F$ is a coherent sheaf, then the higher cohomology of $E$ vanishes, and if, furthermore, $F$ is locally free, then $E$ is globally generated. This reduces understanding the higher cohomology and global generation of $E$ to understanding properties of $F$.

Minamide, Yanagida and Yoshioka in \cite{MYY14} exhibit a chamber $\CC$ in the Bridgeland stability manifold, such that for a Bridgeland stability condition $\sigma\in \CC$ in this chamber, the Bridgeland moduli space $M_{\sigma}(r, dH,a)$ is isomorphic to the Gieseker moduli space $M_H(a, dH, r)$ via the correspondence sending $E$ to $\Phi_{X \to X}^{I_{\Delta}}(E)^{\vee}$. Hence, if the generic sheaf $E$ in the moduli space $M_H(r, dH, a)$ is Bridgeland $\sigma$-semistable, then the higher cohomology of $E$ vanishes. Furthermore, if the generic member of $M_H(a, dH, r)$ is locally free, then $E$ is globally generated. The third author has classified the moduli spaces of sheaves on K3 surfaces whose generic member is not locally free (see \cite{Yos99a} and \cref{Prop:OnlyNonLocallyFreeSheaves}). 

This translates the question of weak Brill-Noether into the problem of determining when the generic sheaf $E$ in $M_H(\v)$ is still $\sigma$-semistable for $\sigma\in\CC$.  If not, then there must be a totally semistable Bridgeland wall between the Gieseker chamber and $\CC$. Using the classification of totally semistable walls in \cite{BM14b}, we obtain a numerical algorithm for deciding when the generic sheaf is $\sigma$-semistable. This provides a finite set of Mukai vectors which fail weak Brill-Noether.

Moreover, if the generic $E$ fails to be $\sigma$-semistable, the largest strictly semistable wall provides a canonical resolution of $E$. In practice, this allows one to compute the cohomology of the generic sheaf even when it does not vanish. In \cref{sec-computations}, we develop general techniques for computing the cohomology.

\subsection*{Further Directions} Our investigations here point the way towards a new approach to the (weak) Brill-Noether problem in general, at least for surfaces.   Without further theoretical developments, our techniques here can be applied immediately to the following question.

\begin{Prob}
Classify the Mukai vectors $\v$ that satisfy weak Brill-Noether on K3 surfaces of higher Picard rank.
\end{Prob}
\cref{Lem:FM-vanishing,prop:isom} are applicable when the K3 surface has higher Picard rank; however, the classification of totally semistable walls from \cite{BM14b} becomes much more intricate as the rank of the Picard group increases. 

The necessary ingredients to apply our techniques to other surfaces, at least those of Kodaira dimension zero, largely exist or are easily obtained.  \cref{Lem:FM-vanishing} generalizes to arbitrary surfaces, and the classification of (totally semistable) Bridgeland walls has been carried out for abelian surfaces in \cite{Yos12}, for Enriques surfaces in \cite{NY19}, and for bielliptic surfaces in forthcoming work of the second author.  The final ingredient in our technique is \cref{prop:isom} which applies already to abelian surfaces \cite{MYY14}.  Generalizing each of these results to arbitrary surfaces and applying our technique would solve the following problem.

\begin{Prob}
On an arbitrary surface, classify the Mukai vectors of stable sheaves  that satisfy the weak Brill-Noether property.
\end{Prob}

Even for Picard rank one K3 surfaces, our work  addresses only the first step towards a higher rank Brill-Noether theory.  Now that we have a clearer picture of the generic cohomology, the next step is the study  of the cohomology jumping loci and their geometry.

\begin{Prob}
Describe the cohomology jumping loci in $M_H(\v)$, e.g. their non-emptiness, number of components, dimension, and singularities.
\end{Prob}

The general structure of cohomology jumping loci for moduli spaces of sheaves on arbitrary varieties has been studied in \cite{CMR:BN} and on K3 surfaces specifically in \cite{Leyenson:BN2, Leyenson:BN} under various assumptions.  By studying the possible Harder-Narasimhan filtrations along Bridgeland walls that are not  necessarily totally semistable, using the dimension estimates from \cite{NY19}, it should be possible further our technique here to systematically study these jumping loci, also known as Brill-Noether loci.

A related topic in the study of vector bundles and their moduli is the question of ampleness.  Recently, Huizenga and Kopper have classified moduli spaces whose general member is globally generated and ample on minimal rational surfaces \cite{HK21}. It would be interesting to carry out their program on K3 surfaces. 

\begin{Prob}
Classify the Mukai vectors $\v$ for which the general sheaf in $M_H(\v)$ is ample.
\end{Prob}

Another problem coming out of our classification of Mukai vectors that satisfy weak Brill-Noether is the following.

\begin{Prob}
Compute the cohomology of the tensor product of two general stable sheaves on a K3 surface.
\end{Prob}

 This problem is central to the study of the S-duality conjecture. To the best of our knowledge, the solution of this problem on surfaces is known in full generality only for $\mathbb{P}^2$ \cite{CoskunHuizengaKopper:Tensor}.

\subsection*{Organization of the paper} In \cref{sec-prelim}, we introduce basic facts concerning moduli spaces of sheaves on K3 surfaces and Bridgeland stability.  In \cref{sec-Strategy}, we explain our main strategy in more detail and introduce the totally semistable Bridgeland walls that will play a crucial role in our analysis.  

In \cref{sec-ReductionSSW}, we prove numerical restrictions on the totally semistable Bridgeland walls that arise. This basic analysis suffices to classify Ulrich bundles on K3 surfaces of Picard rank one in \cref{prop:UlrichBundles}. In \cref{sec-HNFiltration}, following \cite{BM14b}, we describe for a generic sheaf the  Harder-Narasimhan filtration along the totally semistable Bridgeland walls that arise.

In \cref{sec-OX[1]}, we derive the final set of inequalities that govern our study of the weak Brill-Noether and global generation problems. The main result is \cref{thm:summing up inequalities}.  In \cref{sec-Minimala}, we show that if a  Mukai vector $(r, dH, a)$ satisfies weak Brill-Noether,  then the Mukai vector $(r, dH, a')$ also satisfies weak Brill-Noether for any $a \leq a'$ . This reduces our initial search for Mukai vectors that fail weak Brill-Noether to those with maximal $a$. 

In \cref{sec-mainconsequences}, we prove our main qualitative theorems.
The main results are Theorems \ref{thm:bound n r}, \ref{thm:WBN sharp} and \ref{thm:finiteness}.
In \cref{sec-initialclassification}, we classify the boundary cases of Mukai vectors that fail weak Brill-Noether. 

In \cref{sec-computations}, we introduce general techniques for computing the cohomology of the general sheaf in cases when $\v$ does not satisfy weak Brill-Noether. The main tool is the canonical resolution coming from the Harder-Narasimhan filtration with respect to Bridgeland stability along a wall.

Finally, in \cref{sec-rank20}, we classify all Mukai vectors with rank at most 20 that fail to satisfy weak Brill-Noether and compute the cohomology of the general sheaf. 

\subsection*{Acknowledgments} We would like to thank Arend Bayer, Aaron Bertram, Jack Huizenga, John Kopper and Emanuele Macr\`{i} for many valuable conversations.

\section{Background results}\label{sec-prelim}
In this section, we review the necessary background concerning  moduli spaces of sheaves on K3 surfaces and Bridgeland stability conditions.  Some excellent references for the material on classical stability are \cite{Fri98,HL10}.

\subsection{The Mukai lattice}
Let $X$ denote a K3 surface and let $\mathrm{NS}(X)$ denote the N\'{e}ron-Severi space of $X$. The algebraic cohomology $H^*_{\alg}(X,\Z)$  of $X$ decomposes as
$$H^*_\alg(X,\Z) = H^0(X,\Z) \oplus \mathrm{NS}(X) \oplus H^4(X,\Z).$$
Let $\Db(X)$ denote the bounded derived category of coherent sheaves on $X$  and let $K(X)$ denote $K$-group of $X$. Define the  \emph{Mukai vector} $\v:K(X)\onto\Hal(X,\Z)$  by 
$$\v(E):=\ch(E)\sqrt{\td(X)} = (r(E),c_1(E),r(E)+\ch_2(E)) \in \Hal(X, \Z),$$
where $\ch(E)$ is the Chern character of $E$ and $\td(X)$ is the Todd class of $X$.
Given two Mukai vectors $\v =(r,c,a)$ and $\v'=(r', c', a')$, the  \emph{Mukai pairing} is defined by  $$\langle(r,c,a),(r',c',a')\rangle:=c\cdot c'-ra'-r'a\in\Z,$$ where $-\cdot-$ is the intersection pairing on $H^2(X,\Z)$.  The Mukai pairing has signature $(2,\rho(X))$ and satisfies $$\langle\v(E),\v(F)\rangle=-\chi(E,F)=-\sum_i(-1)^i\ext^i(E,F)$$ for all $E,F\in\Db(X)$, where $\ext^i(E,F)= \dim (\Ext^i(E,F))$.
The pair $(\Hal(X,\Z),\langle\blank,\blank\rangle)$ is called the \emph{algebraic Mukai lattice}. Given a Mukai vector $\v\in\Hal(X,\Z)$, we denote
its orthogonal complement by
\[
\v^\perp:=\Set{\w\in\Hal(X,\Z)\ | \ \langle\v,\w\rangle=0 }.
\]
Sometimes we will need to take the \emph{dual Mukai vector} $\v^\vee$ of a given Mukai vector $\v=(r,c,a)$ defined by $\v^\vee:=(r,-c,a)$.
A Mukai vector $\v$ is \emph{primitive} if it is not divisible in $\Hal(X,\Z)$.  
A Mukai vector $\v$ is \emph{spherical} if $\v^2=-2$ and \emph{isotropic} if $\v^2=0$.  We say a primitive Mukai vector $\v=(r,c,a)$ is \emph{positive} if $\v^2\geq -2$, and either 
\begin{enumerate}
    \item $r>0$; or
    \item $r=0$, $c$ is effective, and $a\ne 0$; or
    \item $r=c=0$ and $a>0$.
\end{enumerate} 
We will see that positive Mukai vectors and their multiples are the Mukai vectors of semistable sheaves.

\subsection{Gieseker and slope semistability} Let $H$ be an ample divisor on $X$. All sheaves in this paper will be coherent and pure dimensional.  Let $E$ be a pure $d$-dimensional coherent sheaf on $X$. Then the {\em Hilbert} and {\em reduced Hilbert} polynomials of $E$ are defined by
$$P_{E,H}(m) = \chi(E(mH)) = a_d \frac{m^d}{d!} + \mathrm{l.o.t}, \quad  p_{E,H}(m) = \frac{P_{E,H}(m) }{a_d},$$ respectively. A sheaf $E$ is {\em $H$-Gieseker semistable} if for every proper subsheaf $F \subset E$, $p_{F,H}(m) \leq p_{E,H}(m)$ for $m \gg 0$. The sheaf $E$ is {\em $H$-Gieseker stable} if the inequality is strict for every proper subsheaf $F$. 

Given a torsion-free sheaf $E$, define the $H$-slope $\mu_H(E)$ of $E$ by $$\mu_H(E) = \frac{\ch_1(E) \cdot H}{\ch_0(E) H^2}.$$ 
A sheaf $E$ is {\em slope} or {\em $\mu_H$-semistable} if for every proper subsheaf $F \subset E$, $\mu_H(F) \leq \mu_H(E)$. 
The sheaf is {\em slope} or {\em $\mu_H$-stable} if the inequality is strict for every proper subsheaf $F$.  
Using Hierzebruch-Riemann-Roch to write out $p_{E,H}(m)$ for a torsion-free sheaf $E$, one sees that the first term of $p_{E,H}(m)$ is $\mu_H(E)$, so we have implications
$$\mu_H\mbox{-stable}\Longrightarrow\mbox{stable}\Longrightarrow\mbox{semistable}\Longrightarrow\mu_H\mbox{-semistable}.$$

Every torsion-free sheaf $E$ admits a unique {\em Harder-Narasimhan} filtration
$$0=E_0 \subset E_1 \subset \cdots \subset E_n =E$$ such that the successive quotients $F_i = E_i/E_{i-1}$ are semistable with $p_{F_i, H} (m)> p_{F_{i-1},H}(m)$ for all $i$ and $m \gg 0$. Furthermore, a semistable sheaf admits a {\em Jordan-H\"{o}lder} filtration into stable sheaves. While the Jordan-H\"{o}lder filtration need not be unique, the associated graded object is unique. Two semistable sheaves with the same associated graded object  are called {\em S-equivalent}.  There exists a projective moduli spaces $M_{X, H}(\v)$ parameterizing S-equivalence classes of $H$-Gieseker semistable sheaves \cite{Gie77, Mar77, Mar78}. When the surface $X$ or the ample $H$ is implicit, we will denote the moduli space simply by $M_{H}(\v)$ or $M(\v)$.

When $X$ is a K3 surface, the basic properties of the moduli spaces are well-understood. We summarize the key facts that will play a crucial role in our analysis.

\begin{Thm}\label{Thm:ClassicalModSp}
Let $X$ be a K3 surface over an algebraically closed field $k$, and let $\v=m\v_0\in\Hal(X,\Z)$, where $\v_0$ is  a primitive positive Mukai vector and $m>0$.  Then $M_{X,H}(\v)$ is non-empty for any ample divisor $H$.  If $H$ is in-fact \emph{generic}\footnote{We refer to \cite{OGr:ModuliVB} for the definition of generic.  It always exists when $\v_0$ is positive.} with respect to $\v$, then we also have the following claims.
\begin{enumerate}
    \item The  moduli space $M_{X,H}(\v)$ is non-empty if and only if $\v_0^2\geq-2$.
    \item If $m=1$ or $\v_0^2 >0$, then  $\dim M_{X,H}(\v)=\v^2+2$.
    \item When $v_0^2=-2$, then  $M_{X,H}(\v)$ is a single point parameterizing the direct sum of $m$ copies of a spherical bundle. When $v_0^2=0$, then $\dim M_{X,H}(\v) = 2m$.
    \item When $\v_0^2>0$, $M_\sigma(\v)$ is a normal irreducible projective variety with $\Q$-factorial singularities.  
\end{enumerate}
\end{Thm}
We have collected in \cref{Thm:ClassicalModSp} the cumulative and combined work of many mathematicians.  In the form presented here, the first claim is proven in \cite[Theorems 0.1 and 8.1]{Yos01}.  When $H$ is generic, (1) and (2) follow from the more general result \cref{Thm:nNnemptinessModuliK3} below, while (3) is \cite[Lemmas 7.1 and 7.2]{BM14a}.  Finally, (4) is \cite{KLS06,OG:DesingularK3,PR:FactorialityK3}.

The next two theorems record when $M_{X,H}(\v)$ contains $\mu$-stable sheaves or locally free sheaves.  

\begin{Prop}[{\cite[Remarks 2.2 and 3.3]{Yos99a}}]\label{Prop:ExistenceOfSlopeStable} Let $X$ be a K3 surface with $\Pic(X)=\Z H$ and $H^2=2n$. Let $\v=(lr,ldH,a)\in\Hal(X,\Z)$ be a Mukai vector with $\gcd(r,d)=1$ and $M_{X, H}(\v)\neq\varnothing$.  Then $M_{X,H}(\v)^{\mu s}=\varnothing$ if and only if 
\begin{enumerate}
    \item $r\nmid nd^2+1$,  $\v^2=0$, and $\v$ is not primitive; or
    \item $r\mid nd^2+1$ and $\v^2< 2l^2$.
\end{enumerate}
\end{Prop}
\begin{Prop}[{\cite{Yos99a}}]\label{Prop:OnlyNonLocallyFreeSheaves} Let $X$ be a K3 surface with $\Pic(X)=\Z H$ and $H^2=2n$. Let $\v=(r,dH,a)\in\Hal(X,\Z)$ be a Mukai vector.  Then $M_{X,H}(\v)$ consists only of non-locally free sheaves if and only if 
\begin{enumerate}
    \item $\v^2>0$ and either 
    \begin{enumerate}
        \item $\v=(l,0,-1)e^{pH}=(l,lpH,lp^2n-1)$ for some $l,p\in\Z$, or
        \item $\v=(1,0,-l)e^{pH}=(1,pH,p^2n-l)$ for some $l,p\in\Z$; or
    \end{enumerate}
    \item $\v^2=0$ and $\v=m(r_0^2,r_0d_0H,d_0^2n)$ for $m,r_0,d_0\in\Z$ such that $d_0^2n-r_0a_0=-1$.
\end{enumerate}
\end{Prop}
\begin{proof}
If $\v$ is primitive, then the result follows directly from \cite[Prop. 0.5]{Yos99a}.  If $\v$ is non-primitive, then the proof of \cite[Lemma 3.1]{Yos99a} still gives the result without assuming primitivity.   
\end{proof}

\subsection{Bridgeland stability conditions on K3 surfaces} Let $\AA$ be an abelian category which is the heart of a bounded $t$-structure on $\Db(X)$. The {\em central charge} $Z:  K(\AA) \to \mathbb{C}$ is a group homomorphism which we assume factors through $\v: K(\AA) \to \Hal(X, \Z)$.  As in \cite{Bri07}, a {\em Bridgeland stability condition} is a pair $\sigma=(Z, \AA)$ satisfying the following conditions:
\begin{enumerate}
\item For every nonzero object $E \in \AA$, $Z(E)= r e^{i \pi \phi}$ for some $r >0$ and $\phi \in (0,1]$. This condition allows us to define the $\sigma$-slope of a nonzero object $E \in \AA$ by $$\mu_{\sigma}(E) = - \frac{\Re (Z(E))}{\Im(Z(E))}.$$ A nonzero object $E \in \AA$ is called $\sigma$-semistable if for every proper subobject $F \subset E$ in $\AA$, $\mu_{\sigma}(F) \leq \mu_{\sigma}(E)$. 
\item The pair $(Z, \AA)$ satisfies the {\em Harder-Narasimhan property}, namely that every object has a finite Harder-Narasimhan filtration with $\sigma$-semistable quotients of decreasing slopes.
\item For a fixed norm $|\cdot |$ on $\Hal(X,\Z)$, there exists a constant $C>0$ such that for all $\sigma$-semistable $E$ we have $|Z(E)| \geq C |\v(E)|$.
\end{enumerate}
 A Bridgeland stability condition $\sigma$ is called {\em geometric} if for every point $x \in X$, the skyscraper sheaves $k(x)$ are $\sigma$-stable. We now give the main example of a geometric Bridgeland stability condition \cite{Bri08}.

 \begin{Ex}\label{Ex:MainStab}
 Let $\beta,\omega\in\NS(X)_\R$ be two real divisor classes, with $\omega$ ample.  For $E\in\Db(X)$, define $$Z_{\beta,\omega}(E):=\langle e^{\beta+i\omega},\v(E)\rangle.$$
 If $E$ has Mukai vector $(r, c, a)$, then we can write $Z_{\beta, \omega}(E)$ more explicitly as
 $$Z_{\beta, \omega}(E)= -a - r \frac{\beta^2 - \omega^2}{2} + c \cdot\beta + i \omega \cdot ( c - r \beta).$$
Let $\AA_{\beta,\omega}$ be defined by 
\begin{equation*} \label{eq:AK3}
\AA_{\beta,\omega}:=\Set{E\in\Db(X)\ |\ \begin{array}{l}
\bullet\;\;\HH^p(E)=0\mbox{ for }p\not\in\{-1,0\},\\\bullet\;\;
\HH^{-1}(E)\in\FF_{\beta,\omega},\\\bullet\;\;\HH^0(E)\in\TT_{\beta,\omega}\end{array}},
\end{equation*} where $\FF_{\beta,\omega}$ and $\TT_{\beta,\omega}$ are defined by
\begin{enumerate}
    \item $\FF_{\beta,\omega}$ is the set of torsion-free sheaves $F$ such that for every subsheaf $F'\subseteq F$ we have 
    $\Im Z_{\beta,\omega}(F')\leq 0$;
    \item $\TT_{\beta,\omega}$ is the set of sheaves $T$ such that, for every non-zero torsion-free quotient $T\onto Q$, we have $\Im Z_{\beta,\omega}(Q)>0$.
\end{enumerate}
Then the pair $\sigma_{\beta,\omega}=(Z_{\beta,\omega}, \AA_{\beta,\omega})$ is a geometric Bridgeland stability condition and furthermore, up to group actions, all geometric Bridgeland stability conditions on a K3 surface arise this way (see \cite[Proposition 10.3]{Bri08}).\footnote{Since we do not make any significant computational or theoretical use of these group actions here, we omit their definition.  See \cite{BM14a,Bri08} for more details.}
\end{Ex}

\subsubsection{Walls}\label{subsubsec:Walls} The set $\Stab(X)$ of Bridgeland stability conditions on $X$ has the structure of a complex manifold \cite[Corollary 1.3]{Bri07}. Let $\Stabd(X)$ denote the connected component of $\Stab(X)$ containing geometric stability conditions. For a fixed Mukai vector, the space $\Stabd(X)$ admits a well-behaved wall and chamber structure that will be the key to our results. More specifically, given a Mukai vector $\v\in\Hal(X,\Z)$, there exists a locally finite set of \emph{walls} (real codimension one submanifolds with boundary) in $\Stabd(X)$, depending only on $\v$, with the following properties (see \cite{Bri08,Tod08} for the (1) and (2)):
\begin{enumerate}
    \item When $\sigma$ varies in a chamber, that is, a connected component of the complement of the union of walls, the sets of $\sigma$-semistable and $\sigma$-stable objects of class $\v$ do not change. If $\v$ is primitive, then $\sigma$-stability coincides with $\sigma$-semistability for $\sigma$ in a chamber for $\v$.
    \item When $\sigma$ lies on a wall $\WW\subset\Stabd(X)$, there is a $\sigma$-semistable object of class $\v$ that is unstable in one of the adjacent chambers and semistable in the other adjacent chamber. If $\sigma= (Z, \AA)$ lies on a wall, there exists a $\sigma$-semistable object $E$ with a subobject $F \subset E$ in $\AA$ with the same $\sigma$-slope.
    \item Assume $\Pic(X) = \mathbb{Z} H$, where $H$ is the ample generator with $H^2=2n$. Then  writing $\beta = sH$ and $\omega = tH$ with $t > 0$, the stability conditions described in Example \ref{Ex:MainStab} determine a half-plane $\{(s,t ) | t>0\}$ of stability conditions. Let $\v(E)=(r,dH,a)$. If an object $F$ of Mukai vector $\v(F) = (r_1, d_1H, a_1)$ destabilizes $E$, then there are two possibilities. If  $r_1 d = r d_1$, or equivalently $\mu_H(E) = \mu_H(F)$,  then the wall determined by $F$ is a vertical half-line $$s = \frac{a d_1  - a_1 d}{r_1 a - ra_1},\qquad t>0 .$$ Otherwise, the wall determined by $F$ is a semi-circle $C^{\v(E)}_{\v(F)}$ with center $(\alpha,0)$ and radius $\rho$ given by 
    \begin{equation}\label{eqn:CenterRadius}
        \alpha = \frac{ra_1-r_1 a}{2n(rd_1 - r_1 d)}, \qquad \rho^2 = \alpha^2 - \frac{a_1 d - a d_1}{n(rd_1 - r_1 d)}. 
    \end{equation} The distinct walls are disjoint and nested \cite[Theorem 3.1]{Mac04}.
    \item Given a polarization $H\in\Amp(X)$ and the Mukai vector $\v$ of an $H$-Gieseker semistable sheaf, there exists a chamber $\GG$ for $\v$, the \emph{Gieseker chamber}, where the set of $\sigma$-semistable objects of class $\v$ coincides with the set of $H$-Gieseker semistable sheaves \cite[Prop. 14.2]{Bri08}.
\end{enumerate}

\begin{Def}\label{def:generic} Let $\v\in\Hal(X,\Z)$.  A stability condition $\sigma\in\Stabd(X)$ is called \emph{generic} with respect to $\v$ if it does not lie on any wall for $\v$.
\end{Def}

\subsubsection{Moduli stacks and moduli spaces}

For $\sigma \in \Stab^\dagger(X)$, 
let $\MM_\sigma(\v)$ be the moduli stack of $\sigma$-semistable objects $E$
with $\v(E)=\v$ and let $\MM_\sigma^s(\v)$ 
the open substack of $\sigma$-stable
objects. By \cite[Thm. 4.12]{Tod08}, 
$\MM_\sigma(\v)$ is an Artin stack of finite type. We say two objects $E_1$ and $E_2$ in $\MM_\sigma(\v)(k)$ are S-equivalent if they have the same Jordan-H\"{o}lder  factors.  For $\sigma\in\Stabd(X)$ generic with respect to $\v$, $\MM_{\sigma}(\v)$ (resp. $\MM_\sigma^s(\v)$) admits a projective coarse moduli scheme $M_\sigma(\v)$ (resp. $M_\sigma^s(\v)$), which parameterizes S-equivalence
classes of $\sigma$-semistable (resp. $\sigma$-stable) objects $E$ with $\v(E)=\v$ (see \cite{BM14a} and \cite{MYY14}).  

The following result, which generalizes \cref{Thm:ClassicalModSp}, gives precise conditions for nonemptiness of the moduli spaces $M_\sigma(\v)$ and is proven in \cite{BM14b} and \cite{BM14a}.
\begin{Thm}[{\cite[Thm. 2.15]{BM14b},\cite[Theorem 1.3]{BM14a}}]\label{Thm:nNnemptinessModuliK3}
Let $X$ be a K3 surface over $k$, and let $\sigma\in\Stabd(X)$ be a generic stability condition with respect to $\v=m\v_0\in\Hal(X,\Z)$, where $\v_0$ is primitive and $m>0$.
\begin{enumerate}
    \item The coarse moduli space $M_\sigma(\v)$ is non-empty if and only if $\v_0^2\geq-2$.
    \item Either $\dim M_\sigma(\v)=\v^2+2$ and $M_\sigma^s(\v)\neq\varnothing$, or $m>1$ and $\v_0^2\leq0$.
    \item When $\v_0^2>0$, $M_\sigma(\v)$ is a normal irreducible projective variety with $\Q$-factorial singularities.  
\end{enumerate}
\end{Thm}

\subsubsection{Wall-crossing and Birational transformations} As originally envisioned in \cite{Bri08}, there is a beautiful correspondence between crossing Bridgeland walls and birational transformations between the corresponding moduli spaces.  In this  subsection, we summarize the relevant details of this correspondence from \cite{BM14b}, where it is shown how to classify the geometric behavior at a wall in terms of a certain hyperbolic lattice.

Let $\v\in\Hal(X,\Z)$ with $\v^2>0$, and let $\WW$ be a wall for $\v$.  We will say a given Mukai vector $\v_1$ \emph{induces} $\WW$ if $\WW$ is a connected component of the set $\Set{\sigma\in\Stabd(X)\ |\ \frac{Z_\sigma(\v)}{Z_\sigma(\v_1)}\in\R }$.  We say $\sigma_0\in\WW$ is \emph{generic} if it does not belong to any other wall, and we denote by $\sigma_+$ and $\sigma_-$ two generic stability conditions nearby $\WW$ in two opposite adjacent chambers.  Then all $\sigma_\pm$-semistable objects are still $\sigma_0$-semistable, but the existence of $\sigma_0$-stable objects dictates much of the birational behavior exhibited by crossing $\WW$.  While Bayer and Macr\`{i} achieved a complete classification of walls, and the associated birational transformation, we will only be interested in totally semistable walls.  Recall that a wall $\WW$ is called \emph{totally semistable} if $M_{\sigma_0}^s(\v)=\varnothing$ for any $\sigma_0\in\WW$.  That is, every object in $M_{\sigma_\pm}(\v)$ becomes strictly $\sigma_0$-semistable.  Their result gives the following classification of totally semistable walls.
\begin{Prop}\label{Prop:BMClassificationOfTSSWalls}
Let $\WW$ be a wall for $\v$ with $\v^2>0$.  Then $\WW$ is totally semistable if and only if $\WW$ is induced by either 
\begin{enumerate}
    \item\label{enum:SphericalTSS} a spherical class $\v_1$ such that $\langle \v,\v_1\rangle<0$, or
    \item\label{enum:IsotropicTSS} an isotropic class $\v_1$ such that $\langle \v,\v_1\rangle=1$.
\end{enumerate}
\end{Prop}

\section{Strategy }\label{sec-Strategy}
In this section, we describe our approach to studying the weak Brill-Noether property and global generation.  Let  $X$ be a K3 surface with $\Pic(X)={\Bbb Z}H$, $H^2=2n$.  Let $E$ be a stable sheaf with the Mukai vector
$\v(E)=(r,dH,a) \in\Hal(X,\Z)$, $d>0$ and $\chi(E)=r+a \geq 0$.
By Serre duality and stability,  $$H^2(X, E) \cong \Ext^2(\OO_X,E)\cong\Hom(E,
\OO_X)^{\vee}=0.$$
Hence, the weak Brill-Noether property reduces to the vanishing of $H^1(X,E)$ for a generic sheaf $E\in M_H(\v)$.  We will also investigate for which $\v$, the generic sheaf $E\in M_H(\v)$ is globally generated, equivalently the evaluation map
\begin{equation}\label{eq:ev}
f:H^0(X,E) \otimes {\cal O}_X \to E,
\end{equation}
is surjective.

We will study the Brill-Noether and global generation questions using a certain Fourier-Mukai transform.  Let $I_\Delta$ be the ideal sheaf of the diagonal $\Delta \subset X \times X$, let $\pi_1$ and $\pi_2$ denote the two projections from $X \times X$ to the two factors
and let $\Phi_{X \to X}^{I_\Delta}:\Db(X) \to \Db(X)$ be the integral functor whose kernel is $I_\Delta$.  The fundamental fact behind our approach is the following result.

\begin{Lem}\label{Lem:FM-vanishing}
Let $E$ be a coherent sheaf and 
assume that $F:=\Phi_{X \to X}^{I_\Delta}(E)^{\vee}$ is a coherent sheaf.
Then 
\begin{enumerate}
\item
 $H^1(X,E)=H^2(X,E)=0$.
\item
$F$ is torsion free if and only if $E$  fails to be globally generated in at most finitely many points.
\item
$F$ is locally free if and only $E$ is globally generated.
\end{enumerate}
\end{Lem}

\begin{proof}
For each $E \in \Coh(X)$, tensoring the exact sequence $$ 0 \to I_{\Delta} \to {\cal O}_{X\times X} \to {\cal O}_{\Delta} \to 0$$ by $\pi_1^* E$ and pushing forward by $\pi_2$ induces a long  exact sequence,
\begin{equation}\label{eqn:ShortToLongForEvalSequence}
\begin{CD}
0 @>>> \HH^0(\Phi_{X \to X}^{I_\Delta}(E)) @>>> H^0(X,E) \otimes {\cal O}_X
@>{f}>> E\\
@>>> \HH^1(\Phi_{X \to X}^{I_\Delta}(E)) @>>> H^1(X,E) \otimes {\cal O}_X
@>>> 0\\
@>>> \HH^2(\Phi_{X \to X}^{I_\Delta}(E)) @>>> H^2(X,E) \otimes {\cal O}_X
@>>> 0.
\end{CD}
\end{equation}
Note that $f$ in (\ref{eqn:ShortToLongForEvalSequence}) is the evaluation map from (\ref{eq:ev}). By our assumption, $F^\vee=\Phi_{X\to X}^{I_\Delta}(E)^{\vee\vee}=\Phi_{X\to X}^{I_\Delta}(E)$, so $$\HH^i(\Phi_{X\to X}^{I_\Delta}(E))=\HH^i(F^\vee)=\lExt^i(F,\OO_X).$$ Since $\lExt^i(F,\OO_X)$ is supported in codimension at least $i$,  $\HH^i(\Phi_{X\to X}^{I_\Delta}(E))$ is torsion for $i>0$.  The exact sequence (\ref{eqn:ShortToLongForEvalSequence}) implies that $H^i(X,E)=0$ for $i>0$ and  $\HH^2(\Phi_{X\to X}^{I_\Delta}(E))=0$.  If $F$ is torsion free, then $\lExt^1(F,\OO_X)$
is 0-dimensional, and hence
$f$ is surjective in codimension 1.  
Conversely if $f$ is surjective in codimension 1, then $\lExt^1(F,\OO_X)$
is 0-dimensional.  As $\lExt^2(F,\OO_X)=0$, $F$ must be torsion-free  \cite[Proposition 1.1.10]{HL10}.  Finally, $f$ is surjective if and only if $\lExt^1(F,\OO_X)=\HH^1(\Phi_{X\to X}^{I_\Delta}(E))=0$. This holds if and only if $F$ is locally free.
\end{proof}

\cref{Lem:FM-vanishing} reduces cohomology vanishing and global generation to understanding the object $\Phi_{X\to X}^{I_\Delta}(E)^\vee$, which we will study using Bridgeland stability conditions and wall-crossing.  There is a distinguished chamber ${\cal C}$ in  $\Stabd(X)$ such that if $\sigma$ is a stability condition in ${\cal C}$ and $E$ is a $\sigma$-semistable object of class $\v$, then $\Phi_{X\to X}^{I_\Delta}(E)^\vee$ is a Gieseker semistable sheaf of Mukai vector $(a,dH,r)$. We now make this precise.

When $\Pic(X)=\Z H$, by \cref{Ex:MainStab}, any geometric Bridgeland stability condition $\sigma$ is equal to $\sigma_{(s,t)}:=\sigma_{sH,tH}$ for some $s,t\in\R$ with $t>0$ up to the $\widetilde{\GL}_2^+(\R)$-action.  We write $Z_{(s,t)}$ for its central charge and $\mu_{(s,t)}$ for the corresponding slope function.  By \cite[Lemma 6.2]{Bri08}, these $\sigma_{(s,t)}$ are parameterized by the open subset $\H^0$ of the upper half-plane $\H$ defined by \begin{equation}\label{eq:geometric-def}
\H^0=\H\setminus\bigcup_{\v_1\in\Delta^+(X)}\Set{(s,t)\ |\ \langle e^{(s+\ii t)H},\v_1\rangle\in\R_{\leq 0} },
\end{equation}
where $\Delta^+(X)$ is the subset of spherical Mukai vectors $\v_1$ with $\rk\v_1>0$.  By {\cite[Prop. 2.16]{Yos17}}, $\H^0$ contains the regions $U_
\pm$ seen in \cref{Fig:U_pm} and defined by
\begin{equation}\label{eq:U_*}
U_\pm:=
\Set{(s,t) \ |\ 0<\pm s<\sqrt{\frac{1}{2(H^2)}},
\;t>\sqrt{\frac{1}{2(H^2)}}-
\sqrt{\frac{1}{2(H^2)}-s^2} }.
\end{equation}
Furthermore, 
$I_x^{\vee}$ $(x \in X)$ 
is a $\sigma_{(s,t)}$-stable object for $(s,t) \in U_+$ and
$I_x$ $(x \in X)$ is a $\sigma_{(s,t)}$-stable object for $(s,t) \in U_-$.
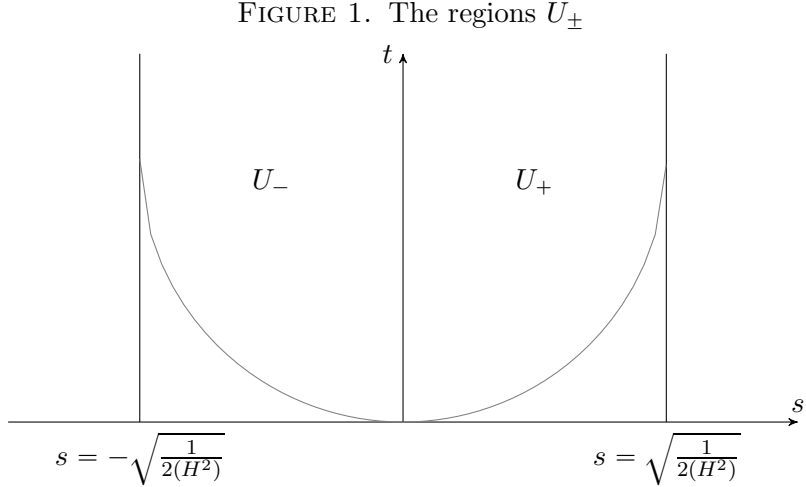
\begin{figure}[h]
\caption{The regions $U_\pm$}\label{Fig:U_pm}
   \begin{tikzpicture}[scale=7]
   %axes
   \draw [->] (-0.75,0) -- (0.75,0) node[above] {$s$};
   \draw[->] (0,0) -- (0,.7) node[left] {$t$};
   \draw[-] (0.5,0) node [below] {$s=\sqrt{\frac{1}{2(H^2)}}$} -- (0.5,.7);
   \draw[-] (-0.5,0) node [below] {$s=-\sqrt{\frac{1}{2(H^2)}}$} -- (-0.5,.7);
   \draw[gray,domain=-0.5:0] plot (\x,{0.5-sqrt(0.25-pow(\x,2))});
  \draw[gray,domain=0:0.5] plot (\x,{0.5-sqrt(0.25-pow(\x,2))});
  \node[below] at (0.25,0.5) {$U_+$};
  \node[below] at (-0.25,0.5) {$U_-$};
   \end{tikzpicture}
\end{figure}

Given the Mukai vector $\v=(r,dH,a)$ of a stable sheaf $E$, let $C^\v_0$ be the semi-circular wall defined by $\mu_{(s,t)}(I_x^{\vee}[1])=\mu_{(s,t)}(E)$, equivalently by
\begin{equation}
t^2+s\left(s-\frac{2a}{d(H^2)} \right)=0.
\end{equation}
 Observe that $C^\v_0$ contains the origin, and moreover, if $1 \gg s>0$ and $t>\sqrt{s\left(\frac{2a}{d(H^2)}-s \right)}$, then
$\mu_{(s,t)}(I_x^{\vee}[1])>\mu_{(s,t)}(E)$, so we let ${\cal C}$ be the chamber containing these points and whose closure contains
$C^\v_0$.  The following result is  crucial for studying $\Phi_{X\to X}^{I_\Delta}(E)^\vee$.

\begin{Prop}[{\cite[Thm. 4.9]{MYY14} }]\label{prop:isom}
Assume that $(s,t) \in {\cal C}$. Then
we have an isomorphism
$$
M_{\sigma_{(s,t)}}(r,dH,a) \cong M_H(a,dH,r)
$$
by sending $E \in M_{\sigma_{(s,t)}}(r,dH,a)$ to
$\Phi_{X \to X}^{I_\Delta}(E)^{\vee}$.
\end{Prop}

Let $0<s_0\ll 1$ and $t_0\gg 0$.  Then $\sigma_{(s_0,t_0)}$ is in the Gieseker chamber $\GG$, and the Bridgeland moduli space  $M_{\sigma_{(s_0,t_0)}}(r,dH,a)$ is isomorphic to the Gieseker moduli space  $M_H(r,dH,a)$ of semistable sheaves.  If there is a path in $U_+$ from the Gieseker chamber  to ${\cal C}$ which does not intersect any totally semistable walls, then the generic $E\in M_H(r,dH,a)$ is $\sigma_{(s,t)}$-semistable for $(s,t)\in\CC$.  Proposition \ref{prop:isom} then implies that the Fourier-Mukai transform $\Phi_{X\to X}^{I_\Delta}(E)^\vee\in M_H(a,dH,r)$.  In particular, $\Phi_{X\to X}^{I_\Delta}(E)^\vee$ is a torsion-free coherent sheaf so that $H^i(X,E)=0$ for $i>0$ and the evaluation map is surjective in codimension one by Lemma \ref{Lem:FM-vanishing}.  If, moreover, the generic sheaf in $M_H(a,dH,r)$ is locally free, then the generic $E\in M_H(r,dH,a)$ is globally generated.

\begin{figure}[h]
\caption{The walls between $\GG$ and $\CC$}\label{Fig:Walls}
\begin{tikzpicture}[scale=0.5]
\draw[->,thick] (-20,0) --(1,0) node[above] {$s$};
\draw[->,thick] (-11.3,0) -- (-11.3,14) node[left] {$t$};
\draw[-,thick] (-1,0) node [below] {$\frac{d}{r}$}--(-1,14) ;
\node[below=.05cm] at (-10,4.9) {$\CC$};
\node[above=.2cm] at (-6,9) {$\GG$};
%\node[below=.05cm] at (-6.8,4.9) {$C_0$};
\begin{scope}
\clip (-20,0) rectangle (1.2,14);
\foreach \x/\y in {3.8/2.9,4.6/3.9,5.4/4.90,6.3/5.9,7.2/6.9,8.1/7.93,9/8.94,10/9.95} {
\draw (-1cm-\x cm,0cm) circle (\y cm);
}
\end{scope}
\end{tikzpicture}
\end{figure}

Otherwise, since the Bridgeland walls are nested semicircles, there exists a totally semistable wall between $C^\v_0$ and the Gieseker chamber $\GG$ as in \cref{Fig:Walls}. In order to study this totally semistable wall we must study the following set of Mukai vectors: 
\begin{Def}\label{Def:DestabilizingMukaiVectors}
For a Mukai vector $\v=(r,dH,a)$ with $d>0$,
let $D_\v$ be the set of Mukai vectors
$\v_1=(r_1,d_1 H,a_1)$ such that
%$\v_1 \ne v({\cal O}_X)$,
\begin{equation}
 \v_1^2 =-2\epsilon\; (\epsilon=0,1),\;\;
d \geq d_1 > 0,\;\;\langle\v,\v_1\rangle<\v_1^2+2,\;\;\text{and}\;\;
\frac{a_1 d-a d_1}{r_1 d-r d_1}>0.
\end{equation}
\end{Def}

The following proposition motivates the definition above.
\begin{Prop}\label{prop:birational}
Let $\v=(r,dH,a)$ be a Mukai vector such that
$r \geq 0$, $d>0$, and $ \v^2  \geq -2$. 
Then the following conditions are equivalent.
\begin{enumerate}
\item\label{enum:Giesker to C}
$M_H(\v) \cap M_{\sigma_{(s,t)}}(\v) \ne \varnothing$
for $(s,t) \in \CC$.
\item\label{enum:no bad Mukai vectors} 
$D_\v=\varnothing$ and $a\geq -r$.
\end{enumerate}
Moreover, if these conditions are satisfied and $E\in M_H(\v)$ is generic, then  $H^i(X,E)=0$ for $i>0$.
\end{Prop}

\begin{proof}
Suppose that $M_H(\v)\cap M_{\sigma_{(s,t)}}(\v)=\varnothing$ for $(s,t)\in\CC$, which is equivalent to the existence of a totally semistable wall between $\CC$ and $\GG$.  Let the largest totally semistable wall be $C^\v_{\v_1}$ with $\v_1 =(r_1, d_1H, a_1)$. By \cref{Prop:BMClassificationOfTSSWalls}, $\v_1$ is either isotropic ($\v_1^2=0$) with $\langle \v,\v_1\rangle<2$ or spherical ($\v_1^2=-2$) with $\langle \v,\v_1\rangle<0$.
Moreover, since $C^\v_0$ contains the origin, $C^\v_{\v_1}$ intersects the vertical line $s=0$ at the point $(0,t)$ with  $0<t= \sqrt{\frac{a_1d-ad_1}{n(r_1d-rd_1)}}$. 
As $\Im Z_{(s,t)}(\v_1),\Im Z_{(s,t)}(\v-\v_1)\geq 0$ for $(s,t)\in C^\v_{\v_1}$, we have 
$$0\leq d_1-r_1s,(d-d_1)-(r-r_1)s,$$
so we must have $0\leq d_1\leq d$ as $s$ can be arbitrarily small for $(s,t)\in C^\v_{\v_1}$.  
If $d_1>0$, then $\v_1\in D_\v$.  Otherwise, $d_1=0$, so from $-2r_1a_1=\v_1^2=-2\epsilon$ for $\epsilon=0,1$, we get $\epsilon=1$ and $r_1 = a_1 = \pm 1$. Hence, $\v_1=(1,0,1)=\v(\OO_X)$ or $\v_1=-(1,0,1)=\v(\OO_X[1])$.  As $\OO_X\in\AA_{(sH,tH)}[-1]$ for $(s,t)\in U_+$, it cannot define a wall for $\v$ in $U_+$.  So the only possibility is $\v_1=\v(\OO_X[1])$.  As $0>\langle \v,
\v(\OO_X[1])\rangle=r+a$, so $a<-r$.

For the converse, by \cref{Prop:BMClassificationOfTSSWalls}, if $a\geq -r$, then $\langle \v,
\v(\OO_X[1])\rangle=r+a\geq 0$, so $\OO_X[1]$ cannot define a totally semistable wall.  As $D_\v=\varnothing$ as well, there cannot be any other totally semistable walls either, so the proposition follows.
\end{proof}

\begin{Rem}\label{Rem:symmetry}
We have a bijective correspondence
\begin{equation*}
\begin{split}
    D_{(r,dH,a)}&\to D_{(a,dH,r)}\\
    (r_1,d_1H,a_1)&\mapsto(a_1,d_1H,r_1)
\end{split}.
\end{equation*}
\end{Rem}

\section{Reductions on possible totally semistable walls}\label{sec-ReductionSSW}
Proposition \ref{prop:birational} provides a numerical criterion for determining the Mukai vectors for which the weak Brill-Noether property might fail.  
We give further restrictions on the possible totally semistable walls.  
In Lemma \ref{Lem:no isotropic}, we show that $D_{\v}$ does not contain isotropic vectors.  
Then we consider two cases, depending on the sign of $a$.  
We first show in \cref{Prop:a<=0} that when $a\leq 0$ there is at most one totally semistable wall between $\GG$ and $\CC$, which we classify completely.  
Then in \cref{Lem:only spherical with r_1>0}, we show that if $a> 0$, any $\v_1 \in D_{\v}$  satisfies $r_1d - rd_1 > 0$ and $a_1d-ad_1>0$.
We observe in \cref{prop:UlrichBundles} that our purely numerical reductions already give enough information to classify Ulrich bundles.
\subsection{Initial reductions on $D_\v$}\label{sec:reductions}

\begin{Lem}\label{Lem:no isotropic}
Let $\v=(r,dH,a)$ be a Mukai vector such that $r\geq 0$, $d>0$, and $\v^2\geq -2$.   Then $D_\v$ does not contain isotropic vectors.
\end{Lem}
\begin{proof}
If $\v_1\in D_\v$ satisfies $\v_1^2=0$, then $nd_1^2=r_1a_1$, so $r_1$ and $a_1$ always have the same sign and cannot vanish. 

Now we consider two cases, depending on the sign of $r_1d-rd_1$.  If $r_1d-rd_1>0$, then $r_1>r\frac{d_1}{d}\geq 0$.  The last condition in the definition of $D_\v$ gives $a_1d-ad_1>0$ as well. Since $a_1 d - a d_1$ is an integer, in particular we have $a_1 d -1 \geq ad_1$. Using the fact that $a_1 = \frac{nd_1^2}{r_1}$, we conclude that $$\frac{nd_1 d}{r_1} - \frac{1}{d_1} \geq a.$$ Hence,
\begin{equation}
\begin{split}
\langle \v,\v_1 \rangle =& 2ndd_1-r_1 a-r a_1
\geq 2ndd_1-\left(\frac{nd_1 d}{r_1} - \frac{1}{d_1}\right)r_1-r \frac{nd_1^2}{r_1}\\
=& \frac{1}{r_1 d_1}\left(
(r_1 d-r d_1)nd_1^2+r_1^2 \right)\\
\geq &\frac{1}{r_1d_1}(nd_1^2+r_1^2)\geq\frac{1}{r_1d_1}(d_1^2+r_1^2)\geq\frac{1}{r_1d_1}(2d_1r_1)= 2,
\end{split}
\end{equation}
a contradiction.

If instead $r_1d-rd_1<0$, then $a_1 d - a d_1 < 0$.  We dispense quickly with the case $r=0$.  In this case, $0>r_1d$ so $r_1$ is negative and 
$$0>a_1d-ad_1=\left(\frac{nd_1^2}{r_1}\right)d-ad_1.$$
Dividing by $d_1>0$ and multiplying by $r_1<0$, we have 
$$ndd_1>r_1a,$$
and combining with $2>\langle\v,\v_1\rangle=2ndd_1-r_1a$ gives
$$ndd_1>r_1a>2ndd_1-2,$$
so $ndd_1\geq 2ndd_1$, a contradiction.

Now we suppose that $r>0$.  We break this case into two, and assume first that $\v^2=-2$.  Then $a=\frac{nd^2+1}{r}>0$, and 
$$
0>a_1d-ad_1=\frac{ndd_1^2}{r_1}-\frac{nd^2d_1}{r}-\frac{d_1}{r}=\frac{ndd_1}{rr_1}(rd_1-r_1d)-\frac{d_1}{r}=\frac{d_1}{r}\left(\frac{nd}{r_1}(rd_1-r_1d)-1\right),
$$
which implies that \begin{equation}\label{eqn:inequality on rd_1-r_1d}
\left(\frac{nd}{r_1}\right)(rd_1-r_1d)<1.
\end{equation}   

If $r_1$ and $a_1$ are negative, the assumption $a>0$ implies that $\langle \v,\v_1\rangle=2ndd_1-r_1a-ra_1>2$, a contradiction. Hence, we may assume that $r_1$ and $a_1$ are positive.  Expressing $\langle\v,\v_1\rangle$, we get 
\begin{align}\label{eqn:equality for v,v_1}
\begin{split}
\langle \v,\v_1\rangle&=2ndd_1-r_1a-ra_1=2ndd_1-r_1\left(\frac{nd^2+1}{r}\right)-r\left(\frac{nd_1^2}{r_1}\right)\\
&=\frac{n}{rr_1}(2dd_1rr_1-r^2d_1^2-r_1^2d^2)-\frac{r_1}{r}=-\frac{n}{rr_1}(rd_1-r_1d)^2-\frac{r_1}{r},
\end{split}
\end{align}
from which it's clear that $\langle \v,\v_1\rangle<0$.  Substituting \eqref{eqn:inequality on rd_1-r_1d} into \eqref{eqn:equality for v,v_1}, we get $$0>\langle \v,\v_1\rangle>-\left(\frac{rd_1-r_1d}{rd}\right)-\frac{r_1}{r}=-\frac{d_1}{d}+\frac{r_1}{r}-\frac{r_1}{r}\geq -1,$$
which is impossible since $\langle \v,\v_1\rangle$ is an integer.

We may therefore assume that $\v^2\geq 0$.  Then $a_1d-ad_1<0$ implies that $a_1<\frac{ad_1}{d}$.  If $r_1$ is positive, then $$nd_1^2=r_1 a_1<r_1\frac{ad_1}{d}.$$  The condition $\v^2\geq 0$ is equivalent to $\frac{a}{d}\leq \frac{d}{r}n$, so $r_1d-rd_1<0$ gives $$\frac{a}{d}\leq\frac{d}{r}n<\frac{d_1}{r_1}n,$$ which leads to a contradiction since then $$nd_1^2<\frac{r_1ad_1}{d}<\frac{r_1d_1^2n}{r_1}=nd_1^2.$$  
So we may suppose that $r_1$ and $a_1$ are negative.  

Solving $$2>\langle\v,\v_1\rangle=2ndd_1-r_1a-ra_1$$  for $a$, gives 
$$a<\frac{2ndd_1-2-ra_1}{r_1}.$$
Hence, we have $$a_1\frac{d}{d_1}<a<\frac{2ndd_1-2-ra_1}{r_1}.$$  Multiplying by $r_1$ reverses the inequality to give $$ndd_1=nd_1^2\left(\frac{d}{d_1}\right)=r_1a_1\left(\frac{d}{d_1}\right)>2ndd_1-2-ra_1.$$
Thus $$1\leq ndd_1<2+ra_1\leq 2+a_1<2,$$ which is an immediate contradiction because of the two strict inequalities and the fact that all quantities involved are integers.
\end{proof}
Next we analyze the Mukai vectors $\v$ with $a \leq 0$ and classify all the totally semistable walls. 

\begin{Prop}\label{Prop:a<=0}
Let $\v=(r,dH,a)$ be a Mukai vector such that $r\geq 0$, $d>0$, and $a\leq 0$.  Then there are no totally semistable walls between $\GG$ and $\CC$ unless $\chi(\v)=r+a< 0$, in which case the unique totally semistable wall is defined by $\OO_X[1]$.  In particular, $\v$ satisfies weak Brill-Noether.
\end{Prop}
\begin{proof}
We begin by showing that if $a\leq 0$, then $D_\v=\varnothing$.  

By \cref{Lem:no isotropic}, if $\v_1\in D_\v$, then $\v_1^2=-2$ and $\langle\v,\v_1\rangle<0$.  From $\v_1^2=-2$, we get $r_1a_1=nd_1^2+1$ so that  $r_1$ and $a_1$ have the same sign.  Moreover, rearranging $0>\langle\v,\v_1\rangle=2ndd_1-r_1a-ra_1$, we get \begin{equation}\label{eqn:base}ar_1>2ndd_1-ra_1.\end{equation}

We  consider two cases based on the sign of $r_1d-rd_1$.  If $r_1d-rd_1>0$, then we have $r_1>r\frac{d_1}{d}\geq 0$.  Thus  $a_1>0$ as well.  Dividing \eqref{eqn:base} by $r_1$ and using $a\leq0$, we get $$0\geq a>\frac{2ndd_1-ra_1}{r_1}.$$ Hence, $r_1>0$ implies $$2ndd_1<ra_1=r\left(\frac{nd_1^2+1}{r_1}\right).$$  Multiplying by $r_1$ and using $r_1>r\frac{d_1}{d}$,  we obtain $$2nrd_1^2<2ndd_1r_1<nrd_1^2+r.$$  Rearranging this becomes $$0\leq r(nd_1^2-1)<0,$$ a contradiction.

Therefore, we must have $r_1d-rd_1<0$ and $a_1d-ad_1<0$. Since $a \leq 0$, we must also have  $a_1<0$  and $r_1<0$.  As $a_1d-ad_1$ is an integer, we must in fact have $a_1d-ad_1\leq-1$.  Similarly, $\langle\v,\v_1\rangle\leq -1$, so dividing \eqref{eqn:base} by $r_1$ we can bound $a$:
$$\frac{a_1d+1}{d_1}\leq a\leq\frac{2ndd_1+1-ra_1}{r_1}.$$
Multiplying this by $r_1d_1<0$, we  get $$r_1a_1d+r_1\geq2ndd_1^2+d_1-ra_1d_1.$$  Rearranging this we get $$2ndd_1^2+d_1\leq r_1a_1d+ra_1d_1+r_1\leq r_1a_1d+r_1=ndd_1^2+d+r_1.$$
Isolating $r_1$ gives 
$$r_1\geq ndd_1^2+d_1-d=d(nd_1^2-1)+d_1\geq d_1>0,$$
a contradiction.
Thus $D_\v=\varnothing$, as claimed.

It now follows from \cref{prop:birational} that there are no totally semistable walls between $\GG$ and $\CC$ if $a\geq-r$ and thus that $H^1(X,E)=H^2(X,E)=0$ for generic $E\in M_H(\v)$.  Moreover, it follows that there is a unique totally semistable wall given by $\OO_X[1]$ if $\chi(\v)=r+a<0$.  Thus $\OO_X[1]$ must be a destabilizing quotient and the Harder-Narasimhan filtration for stability conditions below the wall defined by $\OO_X[1]$ is given by \begin{equation}\label{eqn:HNfiltration}
0\to R_{\OO_X[1]}(E)\to E\to(\OO_X[1])^{\oplus(-r-a)}\to 0,\end{equation}
where $R_{\OO_X[1]}(E)\in M_{\sigma}(-a,dH,-r)$ (see \cite[Proposition 6.8]{BM14b}).  Since $\v'=(-a,dH,-r)$ satisfies $\chi(\v')>0$, the first statement of the proposition implies that the generic $F\in M_H(\v')$ is $\sigma$-stable as $\v'$ has no totally semistable walls and $H^1(X,F)=H^2(X,F)=0$.  Thus for generic $E\in M_\sigma(\v)$, $R_{\OO_X[1]}(E)=F\in M_H(\v')$.  Taking the long exact sequence of cohomology sheaves corresponding to \eqref{eqn:HNfiltration}, we get that $E$ sits in a short exact sequence of sheaves $$0\to\OO_X^{\oplus(-r-a)}\to F\to E\to 0.$$  Taking the long exact sequence of cohomology for this short exact sequence gives that $H^0(X,E)=H^2(X,E)=0$ for generic $E\in M_H(\v)$, as required.
\end{proof}
Finally, when $a\geq 0$, we only need to consider the possibility that $r_1d-rd_1>0$ by the following result.
\begin{Lem}\label{Lem:only spherical with r_1>0}
Let $\v=(r,dH,a)$ be a Mukai vector such that
$r,a \geq 0$, $d>0$.
Then $\v_1=(r_1,d_1 H,a_1) \in D_\v$ satisfies $r_1d-rd_1>0$.
\end{Lem}

\begin{proof}
Suppose instead that $r_1d-rd_1<0$.  Since $-2=\v_1^2=2nd_1^2-2r_1a_1$, $r_1$ and $a_1$ have the same sign and cannot vanish. If $r_1>0$, then  the assumption $r_1d-rd_1<0$ implies that $$n\frac{d}{r}<n\frac{d_1}{r_1}, \quad \mbox{and} \quad a_1d-ad_1<0.$$ Hence, $$nd_1^2+1=r_1a_1<r_1\left(\frac{ad_1}{d}\right).$$  Furthermore, the condition that $\v^2\geq -2$ is equivalent to $$\frac{a}{d}\leq \frac{nd}{r}+\frac{1}{rd}.$$    Combining these inequalities, we have $$nd_1^2+1<r_1d_1\left(\frac{a}{d}\right)\leq r_1d_1\left(\frac{nd}{r}+\frac{1}{rd}\right)<r_1d_1\left(\frac{nd_1}{r_1}+\frac{1}{rd}\right)=nd_1^2+\frac{r_1d_1}{rd}.$$
 The inequalities $r_1d-rd_1<0$ and $d_1\leq d$ force $r_1<r$, so that $\frac{r_1d_1}{rd}\leq 1$, a contradiction.  

If instead $r_1<0$ and $a_1<0$, then $\langle \v,\v_1\rangle=2ndd_1-r_1a-ra_1\geq 2+a+r\geq 2$, leading to a contradiction with the definition of $D_{\v}$.
\end{proof}
\subsection{Classifying Ulrich bundles}
Recall that an {\em Ulrich bundle} on a polarized surface $(Y,A)$ is a bundle $E$ such that all the cohomology groups of $E(-A)$ and $E(-2A)$ vanish. As an application of our discussion of $D_{\v}$, we can classify Chern classes of  Ulrich bundles on Picard rank one K3 surfaces and recover the following theorem of Aprodu, Farkas and Ortega \cite{AFO12} (see also \cite{Fae19}). 
\begin{Prop}\label{prop:UlrichBundles}
Let $X$ be a K3 surface with $\Pic(X)=\Z H$.  There exists an Ulrich bundle of rank $r$ with respect to $mH$ if and only if $2\mid rm$.  Moreover, when an Ulrich bundle of rank $r$ exists, it has Mukai vector $\v=\left(r,\left(\frac{3rm}{2}\right)H,r(2m^2n-1)\right)$.  In particular, there exists an Ulrich bundle of any rank $r\geq 2$ with respect to $2H$.
\end{Prop}
\begin{proof}
The conditions $\chi(E(-mH))=0=\chi(E(-2mH))$ imply that
\begin{align}
\begin{split}
        &r+a+rm^2n-2mdn=0\\
        &r+a+4rm^2n-4mdn=0.
\end{split} 
\end{align}
Solving for $d$ and $a$ gives $\v=(r,dH,a)=\left(r,\left(\frac{3rm}{2}\right)H,r(2m^2n-1)\right)$.  As $$\v^2=r^2\left(n\left(\frac{m}{2}\right)^2+1\right)>0,$$ there exists $E\in M_H(\v)$ if and only if $2\mid rm$.  To see that the generic such $E$ is Ulrich we apply \cref{Prop:a<=0} to $\v'=\v(E(-mH))$ and $\v''=\v((E(-2mH))^\vee)$.  Since $\v'=(r,\left(\frac{rm}{2}\right)H,-r)=\v''$, it follows from \cref{Prop:a<=0} that the generic $E'\in M_H(\v')$ satisfies $H^i(X,E')=0$ for all $i$.  As $E'(-mH)^\vee$ is generic in $M_H(\v')$ as well, we see that $H^i(X,E'(-mH)^\vee)=0$ for all $i$ as well.  It follows from Serre duality that $H^i(X,E'(-mH))=0$ for all $i$, so $E=E'(mH)$ is Ulrich.
\end{proof}

\section{The Harder-Narasimhan filtration of the generic sheaf}\label{sec-HNFiltration}
When $M_H(\v)\cap M_{(sH,tH)}(\v)=\varnothing$ for $(s,t)\in\CC$, there exists a totally semistable wall between the Gieseker chamber and $\CC$ and we cannot apply \cref{Lem:FM-vanishing} and \cref{prop:isom}.  However, we obtain a  Harder-Narasimhan filtration of the generic sheaf in $M_H(\v)$ and can compute its cohomology  using this filtration. In this section, we will study the properties of this filtration.

By \cref{Lem:no isotropic},  the maximal totally semistable wall is defined by some $\v_1\in D_\v$ with $\v_1^2=-2$ and $\langle\v,\v_1\rangle<0$. Let $C^\v_{\v_1}$ be the corresponding semicircle defining the wall. 
Let $\sigma:=\sigma_{(s,t)}$ be a stability condition such that 
$(s,t) \in C^\v_{\v_1}$ and  $0<s \ll 1$. Let $\sigma_-, \sigma_+\in U_+$ be stability conditions sufficiently close to $\sigma$ such that $\sigma_-$ is above $C^\v_{\v_1}$ and $\sigma_+$ is inside $C^\v_{\v_1}$.  Since $C^\v_{\v_1}$ is the maximal totally semistable wall, the generic object of $M_{\sigma_-}(\v)$ is a Gieseker semistable sheaf $E$ with $\v(E)=\v$.

Let $\AA$ be the subcategory of $\Db(X)$ consisting of
$\sigma$-semistable objects $E'$ with 
$\phi_\sigma(E') =\phi_\sigma(E)$, where
$E \in M_H(\v)$.  Let $\HHH$ be the hyperbolic lattice spanned by
$\v(E')$ ($E' \in {\cal A}$).  As $C^\v_{\v_1}$ is defined by a spherical class $\v_1\in\HHH$, there are two possible cases to consider by \cite[Prop. 6.3]{BM14b}.
\subsection{One spherical object}
The first possibility is that $\HHH$ contains a unique spherical class up to sign.  Then we are in Case (b) of \cite[Prop. 6.3]{BM14b}, so there is a unique $\sigma$-stable spherical object $T_1\in\AA$ with $\v(T_1)\in\HHH$ and $\v(T_1)=\v_1$ is the unique \emph{effective} spherical class in $\HHH$.  The next lemma describes the Harder-Narasimhan filtration of the generic sheaf in this case.
\begin{Lem}\label{lem:max-HNF2}
Assume that $\HHH$ contains a unique effective spherical class $\v_1$.  Let $T_1$ be the corresponding $\sigma$-stable spherical object.
 Then for a general sheaf $E \in M_{H}(\v)$, there is an exact sequence
\begin{equation}\label{eq:isotropic HN filtration}
0 \to \Hom(T_1,E) \otimes T_1 \to E \to F \to 0
\end{equation}
where $\Ext^i(T_1,E)=0$ for $i \ne 0$ and $F$ is a $\sigma$-stable object such that $\v(F)^2=\v^2$.  Moreover, $T_1$ is a stable spherical vector bundle.
 \end{Lem} 
\begin{proof}
Since $\langle\v,\v_1\rangle<0$, either $\Hom(T_1, E)$ or $\Hom(E, T_1)$ is nonzero. Since $T_1$ is $\sigma$-stable, $T_1$ is either a subobject or quotient of every $E\in M_H(\v)$ which destabilizes $E$ with respect to $\sigma_+$.  We claim that $T_1$ is in fact a destabilizing subobject.  

Otherwise, $\Hom(E,T_1)\ne 0$ and there  exists a surjection $E\onto T_1$ in $\AA$.  Let $F$ be the kernel. Hence, we get a short exact sequence in $\AA$, $$0\to F\to E\to T_1\to 0.$$  Taking cohomology sheaves, we conclude  that $F$ is a sheaf.  Moreover, by \cref{Lem:only spherical with r_1>0}, if we write $\v_1=(r_1,d_1H,a_1)$, then $r_1 d-rd_1>0$, which implies that $\mu(F)>\mu(E)$, contradicting the Gieseker stability of $E$.  Thus $\Hom(E,T_1)=0$ and $T_1$ is a destabilizing subobject of a general $E\in M_H(\v)$.  By taking cohomology sheaves of the destabilizing sequence, $T_1 \in \Coh(X)$. Hence, $T_1$ is a simple and rigid sheaf, which must be a Gieseker stable locally free sheaf by \cite[Prop. 3.14]{Muk87b}.

The fact that \eqref{eq:isotropic HN filtration} gives the Harder-Narasimhan filtration follows from \cite[Lemmas 6.8 and 8.3]{BM14b}.  As $F$ and $T_1$ are non-isomorphic, $\sigma$-stable objects of the same slope, $\Hom(T_1,F)=0$, so it follows from $\v(F)=\v-\hom(T_1,E)\v_1$ and $\v(F)^2=\v^2$ that
$$\hom(T_1,E)=-\langle \v,\v_1\rangle=\hom(T_1,E)-\ext^1(T_1,E)+\ext^2(T_1,E).$$  By Serre duality, $\ext^2(T_1,E)=\hom(E,T_1)=0$ from above, so it follows that $\ext^1(T_1,E)=0$ as we wanted.
\end{proof}
\subsection{Two stable spherical objects} \label{subsec:two stable spherical}

Otherwise, by Case (c) of \cite[Prop. 6.3]{BM14b},  there are exactly two $\sigma$-stable spherical objects $T_0$ and $T_1$ in ${\cal A}$.  We now describe the Harder-Narasimhan filtration of the general sheaf in this case.

\begin{Lem}
\label{lem:max-HNF1}
Assume that $\AA$ contains exactly two $\sigma$-stable spherical objects $T_0$ and $T_1$ with $\v(T_i)\in\HHH$ for $i=0,1$ and that $\langle\v,\v(T_1)\rangle<0$.  
\begin{enumerate}
    \item\label{enum:spherical class non-isotropic} If $\v^2=-2$, then the unique $E\in M_H(\v)$ is in the abelian category generated by $T_0$ and $T_1$.  In particular, if $h^1(T_i)=0$ for $i=0,1$, then $h^1(E)=0$.
    \item\label{enum:positive class non-isotropic} If $\v^2\geq0$, then the general $E\in M_H(\v)$ sits in an exact sequence
$$
0 \to F_1 \to E \to F_2 \to 0
$$ 
in ${\cal A}$
such that $F_1$ is generated by $T_i$, $i=0,1$,  
$F_2=\Phi(E)$ for an equivalence $\Phi:\Db(X) \to \Db(X)$, and $F_2$ is $\sigma$-stable .
If $h^1(T_i)=0$ for $i=0,1$, then 
$h^1(E)=h^1(F_2)$. 
\end{enumerate}
\end{Lem}

\begin{proof}

Set $\u_0:=\v(T_0)$ and $\u_1:=\v(T_1)$.
By \cite[Prop. 6.3]{BM14b}, $\langle \u_0,\u_1 \rangle>2$ and the effective cone of ${\cal H}_{\Bbb R}$ is generated by 
$\u_0$ and $\u_1$.  Hence,  $\v_1=a\u_0+b\u_1$ with $a,b\in\Z_{\geq 0}$. Since $\langle \v, \v_1 \rangle <0$, without loss of generality, we may assume that $\langle\v,\u_1\rangle<0$.  One can check that  $C^\v_{\v_1}=C^\v_{\u_1}$ and $\u_1\in D_\v$.  Thus we may set $\v_1=\u_1=(r_1,d_1 H,a_1)$. Recall by \cref{Lem:only spherical with r_1>0} that 
$r_1>0$ and $r_1 d-r d_1>0$.  The same argument as in \cref{lem:max-HNF2} proves that $T_1$ is a destabilizing subobject of the generic $E\in M_H(\v)$ and that moreover, $T_1$ is a spherical Gieseker stable locally free sheaf.  We recall further from \cite[Prop. 6.3(c)]{BM14b} that in this case $\HHH$ is non-isotropic.

Given a spherical class $\u\in\HHH$, the spherical reflection $\rho_\u$ is defined by $\rho_\u(\v)=\v+\langle\v,\u\rangle\u$.
Let $\u_i$ ($i \geq 2)$ be the $(-2)$-vectors defined by
\begin{equation}
\begin{split}
\u_2=& \rho_{\u_1}(\u_0),\\
\u_i=& -\rho_{\u_{i-1}}(\u_{i-2})\; (i \geq 3). 
\end{split}
\end{equation}

For $i\geq 0$, let $T_i^-$ be the unique $\sigma_-$-stable object with $\v(T_i^-)=\u_i$.  In particular, $T_i^-=T_i$ for $i=0,1$.  If $\v^2=-2$, then by \cite[\S 6]{BM14b}, there exists $i$ such that $E=T_i^-$.
We note that $T_i^- \in \AA_\sigma$ $(i>1)$ are generated by $T_0$ and $T_1$ by \cite[Lemma 6.2]{BM14b}.  This gives part (1).

Otherwise, $\v^2>0$ since $\HHH$ is non-isotropic, and 
$$
\phi_{\sigma_-}(T_1^-)<\phi_{\sigma_-}(T_2^-)<\cdots<\phi_{\sigma_-}(E)<
\phi_{\sigma_-}(T_0^-).
$$
For $i \geq 1$, we set 
\begin{equation}
{\cal F}_i=\langle T_1^-,...,T_i^- \rangle,\;
{\cal T}_i:=\{ E \in {\cal A} \mid \Hom(E,F)=0, F \in {\cal F}_i \}.
\end{equation}
Then $({\cal T}_i,{\cal F}_i)$ is a torsion pair of $\AA$. 
Set ${\cal A}_i:=\langle {\cal T}_i,{\cal F}_i[1] \rangle$ to be the tilting of $\AA$ at this torsion pair for $i\geq 1$, and set $\AA_0=\AA$.
Then $T_0 \in {\cal A}_i$ for all $i\geq 0$.
For a spherical object $T$, let $R_T:\Db(X) \to \Db(X)$ be the spherical functor,
i.e., 
$$
R_T(E):=\Cone({\bf R}\Hom(T,E) \otimes T \to E),\;E \in \Db(X).
$$
We have equivalences $R_{T_i^-}:{\cal A}_i \to {\cal A}_{i-1}$
and $R_{T_{i-1}^-} \circ R_{T_i^-}$ preserves stability
(cf. \cite[section 6.2]{NY19}).

For $m \geq 1$, we set
\begin{equation}
{\cal C}_m:=\{ x \in {\cal H} \mid \langle x, \u_m \rangle \leq 0,
\langle x,\u_{m+1} \rangle \geq 0 \},
\end{equation}
and we let $\CC_0:=\Set{x\in\HHH \mid 0\leq\langle x,\u_i\rangle,i=0,1}$.  
Assuming that $\v^2 \geq 0$, it follows from $\langle\v,\u_1\rangle<0$ that there is an $m_0\in\N$ such that
$\v \in {\cal C}_{m_0}$.  
We set $\Psi:=R_{T_1^-} \circ R_{T_2^-} \circ \cdots \circ R_{T_{m_0}^-}$.
Then by \cite[Section 6.2]{NY19} $\Psi$ is an equivalence ${\cal A}_{m_0} \cong {\cal A}$
which induces an isomorphism $M_{\sigma_-}(\v) \cong M_{\sigma_-}(\Psi(\v))$ if $m_0$ is even (resp., $M_{\sigma_-}(\v) \cong M_{\sigma_+}(\Psi(\v))$ if $m_0$ is odd), 
where $\Psi(\v) \in {\cal C}_0$.
By \cite[Lem. 6.5]{BM14b}, there is a $\sigma$-stable object 
$E' \in M_{\sigma_\pm}(\Psi(\v))$.
Hence $M_{\sigma_-}(\v)$ contains an irreducible object of ${\cal A}_{m_0}$
(\cite[Prop. 6.8]{BM14b}).

 For a general $E \in M_{\sigma_-}(\v)=M_H(\v)$, we set 
$E^{m_0}:=E$ and 
$E^i:=R_{T_{i+1}^-} \circ R_{T_{i+2}^-} \circ \cdots \circ R_{T_{m_0}^-}(E)$
$(0 \leq i <m_0)$.  We prove by induction on $i$ that each $E^i$ satisfies the conclusion of the lemma.  If $i=0$, then by definition $\v(E^0)=\Psi(\v)\in\CC_0$ and $E^0\in M_{\sigma_\pm}(\Psi(\v))$ is $\sigma$-stable and thus an irreducible object of $\AA_0=\AA$.  Thus we may take $F_2=E^0$, $F_1=0$, and $\Phi=\id$ proving the lemma when $i=0$.  

Now suppose we have shown that $E^{i-1}$ sits in an exact sequence $$0\to F_1^{i-1}\to E^{i-1}\to F_2^{i-1}\to 0$$in $\AA$ with $F_1^{i-1}$ generated by $T_i$, $i=0,1$ and $F_2^{i-1}$ $\sigma$-stable.    
As $R_{T_i^-}$ induces an equivalence between $\AA_i$ and $\AA_{i-1}$, it follows by induction that $E^i$ is an irreducible object of ${\cal A}_i$ since $E_0$ is an irreducible object of $\AA_0$.  Thus  $\Hom(T_i^-[1],E^i)=\Hom(E^i,T_i^-[1])=0$.  Applying $R_{T_i^-}$ to $E^i$, we get the exact triangle
\begin{equation}\label{eq:A_i}
 T_i^- \otimes \Hom(T_i^-,E^i) \to E^i \to E^{i-1} \to 
 T_i^- \otimes \Hom(T_i^-,E^i)[1].
\end{equation}
Since $T_i^-, E^i,E^{i-1} \in {\cal A}$,
\eqref{eq:A_i} is regarded as an exact sequence in ${\cal A}$: $$0\to T_i^-\otimes\Hom(T_i^-,E^i)\to E^i\to E^{i-1}\to 0.$$  Define $F_2^i:=F_2^{i-1}$ and $F_1^i$ to be the kernel of the composition of surjections in $\AA$: $$E^i\onto E^{i-1}\onto F_2^{i-1}.$$  Then $F_1^i$ sits in a short exact sequence $$0\to T_i^-\otimes\Hom(T_i^-,E^i)\to F_1^i\to F_1^{i-1}\to 0$$ in $\AA$.  The induction hypothesis then gives the claim since $T_i$ is generated by $T_0$ and $T_1$.  Thus $E=E^{m_0}$ sits in the required short exact sequence with $F_2=E_0=R_{T_{1}^-} \circ R_{T_{2}^-} \circ \cdots \circ R_{T_{m_0}^-}(E)$, as required.  The second claim follows. 
\end{proof}

\section{Comparing with the wall defined by $\OO_X[1]$}\label{sec-OX[1]}
In this section, we observe that we can divide $D_\v$  into two groups, corresponding to whether the totally semistable wall lies above or below the wall defined by $\OO_X[1]$, and we prove that to study the weak Brill-Noether problem, we can ignore those totally semistable walls below the wall defined by $\OO_X[1]$.  In order to prove that $\OO_X[1]$ defines a wall, we must first prove that $\OO_X[1]$ is $\sigma$-stable throughout $U_+$:
\begin{Lem}\label{Lem:O_X[1] is stable}
$\OO_X[1]$ is $\sigma_{(s,t)}$-stable for $(s,t)\in U_+$.
\end{Lem}
\begin{proof}
Suppose that $\OO_X[1]$ is not $\sigma_{(s,t)}$-semi-stable.  Then by Mukai's Lemma (see  \cite[Lemma 6.1]{BM14b}), we have an exact sequence in $\AA_{(s,t)}$ $$0\to A\to \OO_X[1]\to B^{\oplus k}\to 0$$ with $B$ a $\sigma_{(s,t)}$-stable spherical object such that $\mu_{(s,t)}(\OO_X[1])>\mu_{(s,t)}(B)$.  The long exact sequence of cohomology objects $$0\to \HH^{-1}(A)\to\OO_X\to\HH^{-1}(B)^{\oplus k}\to\HH^0(A)\to 0$$ implies that $B[-1]\in\FF_{(sH,tH)}$ is a simple and rigid sheaf.  Hence, $B[-1]$ is Gieseker stable by \cite[Prop. 3.14]{Muk87b}.  Let  $\v(B[-1])=(r_1,d_1H,a_1)$.  Then $\Hom(\OO_X,B[-1])\neq 0$ implies that $d_1\geq 0$ and $B[-1]\in\FF_{(sH,tH)}$ implies that $d_1-r_1s\leq 0$.  Thus $s\geq\frac{d_1}{r_1}$.  

If $t$ were sufficiently large, then $\OO_X[1]$ would be $\sigma_{(s,t)}$-stable and thus we would have $$\mu_{(s,t)}(\OO_X[1])<\mu_{(s,t)}(B).$$ Hence, $(s,t)$ is below the semi-circular wall $C^\v_{\v_1}$ defined by $\v_1=\v(B)=-(r_1,d_1H,a_1)$.  The equation of $C^\v_{\v_1}$ is $$t=\sqrt{\frac{2}{H^2}-\frac{2(r_1-a_1)}{d_1 H^2}s-s^2},$$ and $C^\v_{\v_1}$ intersects the boundary curve of $U_+$ from \eqref{eq:U_*} at the point $$\left(\frac{2d_1 \left(r_1-a_1\right)}{2(r_1-a_1)^2+d_1^2 H^2},\frac{2d_1^2}{\sqrt{\frac{2}{H^2}} \left(2(r_1-a_1)^2+d_1^2 H^2\right)}\right).$$  Since $\frac{2d_1 \left(r_1-a_1\right)}{2(r_1-a_1)^2+d_1^2 H^2}\leq\frac{d_1}{r_1}$ $\leq s$,   and $(s,t)$ is below $C^\v_{\v_1}$,  $(s,t)$ must also be below the boundary curve in \eqref{eq:U_*}, a contradiction to $(s,t)\in U_+$.  

Thus $\OO_X[1]$ is $\sigma_{(s,t)}$-semistable for all $(s,t)\in U_+$.  Since $U_+$ is an open set and $\v(\OO_X[1])$ is primitive, $\OO_X[1]$ must in fact be $\sigma_{(s,t)}$-stable throughout $U_+$, as claimed.
\end{proof}

Now we consider $\v_1\in D_\v$ whose wall $C^\v_{\v_1}$ is below the wall defined by $\OO_X[1]$.  Let $(0,t_1)$ denote the intersection of the semi-circular wall $C^\v_{\v_1}$ with the line $s=0$.  Then $\frac{da_1-d_1a}{dr_1-d_1r}=\frac{H^2}{2}t_1^2$.  On the other hand, the wall defined by $\OO_X[1]$ intersects the $t$-axis in the point $(0,\sqrt{\frac{2}{H^2}})$.  Thus we see that $C^\v_{\v_1}$ lies below the wall defined by $\OO_X[1]$ if and only if $a_1d-ad_1<r_1d-rd_1$.  We show that we can ignore such elements of $D_\v$ in studying the weak Brill-Noether problem.  
\begin{Lem}\label{Lem:O_X[1] largest wall}
Let $\v=(r,dH,a)$ be a Mukai vector such that $r,a\geq 0$, $d>0$, and $\v^2\geq -2$. 
\begin{enumerate}
\item If $a_1d-ad_1<r_1d-rd_1$ for all $\v_1\in D_\v$, then 
 $\v$ satisfies weak Brill-Noether.
 \item Moreover, if $\v$ satisfies $\v^2\geq 0$ and  $a_1 d-a d_1 \leq r_1 d-rd_1$ for all $\v_1 \in D_\v$, then $\v$ satisfies weak Brill-Noether.
\end{enumerate}
\end{Lem}
\begin{proof}
The hypothesis in (1) is equivalent to every totally semistable wall being below the wall defined by $\OO_X[1]$, which is not itself totally semistable.  Thus the generic sheaf $E\in M_H(\v)$ is $\sigma_{(s,t)}$-stable for $0<s\ll 1$, $t=\sqrt{\frac{2}{H^2}}-\epsilon$, and $0<\epsilon\ll 1$, i.e. in the adjacent chamber below the wall defined by $\OO_X[1]$.  By \cref{Lem:O_X[1] is stable}, $\OO_X[1]$ is $\sigma_{(s,t)}$-stable.  Hence, $\Hom(E,\OO_X[1])=0$ for $\sigma_{(s,t)}$ as above since then $\mu_{(s,t)}(E)>\mu_{(s,t)}(\OO_X[1])$ and both objects are $\sigma_{(s,t)}$-stable (for generic $E\in M_H(\v)$). By Serre duality, $$H^1(X,E)^\vee=\Ext^1(\OO_X,E)^\vee\cong\Ext^1(E,\OO_X)=\Hom(E,\OO_X[1])=0.$$  Furthermore, the vanishing of $H^2(X,E)$ follows since $$H^2(X,E)=\Ext^2(\OO_X,E)=\Hom(E,\OO_X)^\vee=0$$ by (classical) stability and the fact that $d>0$.

In (2), we may suppose that the maximal totally semistable wall $C=C^\v_{\v_1}$ is induced by $\v_1$ with  $a_1d-ad_1=r_1d-rd_1$, i.e. $C=C^\v_{\v(\OO_X[1])}$. Then we are in the situation discussed in 
\cref{subsec:two stable spherical}.  As $\OO_X[1]$ is $\sigma$-stable throughout $U_+$ and $\langle\v,\v(\OO_X[1])\rangle=r+a>0$,  we may assume that $T_0=\OO_X[1]$, so that $\u_0=\v(\OO_X[1])$, and $\v_1=\u_1=\v(T_1)$ is the Mukai vector of the $\sigma_-$-stable spherical destabilizing subobject $T_1$ of every $E\in M_{\sigma_-}(\v)=M_H(\v)$.  

Since $\v^2\geq 0$ we have $\v\in\CC_m$ for some $m\in\N$.  As $\OO_X[1]\in\AA_i$ for all $i\geq 0$ and the generic object of $M_H(\v)$ is an irreducible object of $\AA_m$, as we noted in the proof of \cref{lem:max-HNF1}.  
If $\Hom(E,\OO_X[1]) \ne 0$, then since $\cal A_m$ is Artinian, we have an exact sequence in $\cal A_m$
$$0 \to E_1 \to \OO_X[1] \to E_2 \to 0$$
such that $E_1$ is generated by $E$ and $\Hom(E,E_2)=0$.
Since $$\Hom(E_1,E_2)=0 \quad \mbox{and} \quad \Ext^1(\OO_X[1],\OO_X[1])=0,$$ it follows that $\Ext^1(E,E)=0$, which shows that $E$ is rigid, contradicting $\v^2\geq 0$.  Thus $$\Hom(E,\OO_X[1])=0$$ for generic $E\in M_H(\v)$.  The vanishing of $H^1(X,E)$ and $H^2(X,E)$ then follow as before.
\end{proof}

\begin{Ex}\label{ex:Fibonacci}
In this example, we generalize \cref{Ex:PullbackTangentBundle} significantly.  
Assume that $n=1$ and that
${\cal A}$ contains ${\cal O}_X(H)$.
Then, by repeating the argument of \cref{Lem:O_X[1] is stable}, we see that ${\cal O}_X(H)$ is a $\sigma_{(s_0,t_0 )}$-stable object,
where $(s_0,t_0)$ is on the semi-circular wall where $\OO_X(H)$ and $\OO_X[1]$ have the same slope and $0<s_0\ll 1$.
Let $F_i$ be the Fibonacci numbers. 
For $\v=F_{i+1}\v({\cal O}_X(H))+F_{i-1}\v({\cal O}_X[1])$,
there is a stable sheaf $E \in M_H(\v)$
fitting in an exact sequence
\begin{equation}\label{eq:Fibonacci counterexamples}
0 \to {\cal O}_X(H)^{\oplus F_{i+1}} \to E \to {\cal O}_X[1]^{\oplus F_{i-1}} \to 0.
\end{equation} 
For example, we may take $E$ to be the pullback from $\P^2$ of the twist of a Steiner bundle (see for example \cite[Thm 1.4 and Example 1.5]{Hui13}).  In particular, $E$ is $\sigma_{(s_0,t_0 )}$-semistable.
Since $E$ is $\sigma_{(s_0,t )}$-stable for $t \gg 0$,
$E$ must in-fact be $\sigma_{(s_0,t )}$-stable for all $t >t_0$.
Therefore $M_H(\v) \cap M_{\sigma_{(s_0,t)}}(\v) \ne \varnothing$
for $t>t_0$.  For even $i$, we have $\v^2=2$, so applying Lemma \ref{Lem:O_X[1] largest wall},
we get $\v$ satisfies weak Brill-Noether.  For odd $i$, $\v^2=-2$, and thus $E$ as in \eqref{eq:Fibonacci counterexamples} is the unique element of $M_H(\v)$.  Taking cohomology sheaves, we can express $E$ as in the following short exact sequence $$0\to\OO_X^{\oplus F_{i-1}}\to\OO_X(H)^{\oplus F_{i+1}}\to E\to 0,$$ from which we see that $h^1(X,E)=F_{i-1}$.  Observe that $\v(E)=(F_i,F_{i+1}H,F_{i+2})$.

By imitating a similar construction for Steiner bundles on higher dimensional projective spaces, we can construct similar counterexamples to weak Brill-Noether when $n>1$.
\end{Ex}

We summarize the discussion in \cref{sec-ReductionSSW,sec-Strategy,sec-OX[1]} in the following theorem.

\begin{Thm}\label{thm:summing up inequalities}
Let $\v=(r,dH,a)$ be a Mukai vector such that $r,a\geq 0$, $d>0$, and $\v^2\geq -2$.  Let $D_\v^{BN}\subset D_\v$  be the set of Mukai vectors $\v_1=(r_1,d_1H,a_1)$ satisfying $$0<dr_1-d_1r\leq da_1-d_1a.$$
If $\v$ does not satisfy weak Brill-Noether, then $D_\v^{BN}\neq\varnothing$.  Suppose that $\v$ satisfies weak Brill-Noether and $a\geq 2$.  If the generic $E\in M_H(\v)$ is not globally generated, then either $D_\v \neq\varnothing$ or $M_H(a,dH,r)$ consists of non-locally free sheaves.
\end{Thm}
\begin{proof}
If $D_\v^{BN}=\varnothing$, then by \cref{prop:birational,Lem:O_X[1] largest wall} all totally semistable walls in $U_+$ (if there are any) are below the wall defined by $\OO_X[1]$ so that the generic $E\in M_H(\v)$ has $H^i(X,E)=0$ for $i>0$.  If  $D_\v=\varnothing$, then by Propositions \ref{prop:isom} and \ref{prop:birational}, we have that $\Phi_{X\to X}^{I_\Delta}(E)^\vee\in M_H(a,dH,r)$ for generic $E\in m_H(\v)$.  If $M_H(a,dH,r)$ generically consists of locally free sheaves, then the generic $E\in M_H(\v)$ is globally generated by Lemma \ref{Lem:FM-vanishing}.
\end{proof}

\begin{Rem}\label{rem-numerics}
Let $\v=(r,dH,a)$ be a Mukai vector such that $r,a\geq 0$, $d>0$, and $\v^2\geq -2$. If $\v$ fails to satisfy weak Brill-Noether, then Theorem \ref{thm:summing up inequalities} concretely asserts the existence of a Mukai vector $\v_1=(r_1,d_1H,a_1)$ satisfying the following inequalities.  
\begin{enumerate}
\item\label{enum:spherical} $\v_1^2=-2$;
\item\label{enum:d ineq} $0<d_1\leq d$;
\item\label{enum:positive r_1} $dr_1-d_1r>0$;
\item\label{enum:positive a_1} $da_1-d_1a>0$;
\item\label{enum:totally semistable spherical} $0>\langle \v,\v_1\rangle=2ndd_1-r_1a-ra_1$;
\item\label{enum:H^1 vanishing} $0<dr_1-d_1r\leq da_1-d_1a$
\end{enumerate}
In particular, given $\v$ if there does not exist $\v_1$ satisfying these inequalities, then $\v$ satisfies weak Brill-Noether. Observe that given $\v$, checking for the existence of $\v_1$ is an easy numerical task.
\end{Rem}

\section{Counterexamples to Weak Brill-Noether of minimal square}\label{sec-Minimala}
Let $\v=(r,dH,a)$ be a Mukai vector with $r,d>0$ and $\v^2\geq -2$, and consider the related Mukai vector $\v' = (r, dH, a-c)$ for $c \geq 0$. In this section, using elementary modifications, we show that if $\v$ satisfies weak Brill-Noether, then so does $\v'$. In particular, in classifying counterexamples to weak Brill-Noether, we make the task easier by restricting our search to those Mukai of maximal $a$, or equivalently minimal $\v^2$.

\begin{Prop}\label{prop:minimal-a}
Let $\v=(r,dH,a)$ be a Mukai vector with $r, d>0$, and $\v^2\geq-2$. Let $c > 0$ be an integer. 
If there exists $E \in M_H(\v)$ such that $H^1(X,E)=0$, then $(r,dH,a-c)$ satisfies weak Brill-Noether.  In particular, if some $E\in M_H\left(r,dH,\left\lfloor\frac{nd^2+1}{r}\right\rfloor\right)$ satisfies $H^1(X,E)=0$, then $\v$ satisfies weak Brill-Noether.
\end{Prop}

\begin{Rem}
The condition $\v^2\geq -2$ is equivalent to $a\leq\frac{nd^2+1}{r}$, so for fixed $(r,d)$, $\left(r,dH,\left\lfloor\frac{nd^2+1}{r}\right\rfloor\right)$ is the Mukai vector of a stable sheaf of rank $r$ and degree $d$ with smallest square.  \cref{prop:minimal-a} tells us that we may focus our efforts on studying this Mukai vector.  
\end{Rem}

\begin{proof}[Proof of \cref{prop:minimal-a}]
We prove the proposition by taking general elementary modifications at points. General elementary modifications preserve $\mu$-(semi)stability and the property of a sheaf having at most one nonzero cohomology group \cite[Lemma 2.7]{CoskunHuizenga:WBN}. Unfortunately, elementary modifications do not preserve Gieseker semistability in general, so we will need to take some care.

Let ${\cal M}_H(\v)^{\mu ss}$ be the moduli stack 
of $\mu$-semistable
sheaves $E$ with $\v(E)=\v$, and  let ${\cal M}_H(\v)^{\mu s}$ be the open substack of ${\cal M}_H(\v)^{\mu ss}$
consisting of $\mu$-stable sheaves.  We write a Mukai vector as $\v=(lr_0,ld_0 H,a)$, where
$\gcd(r_0,d_0)=1$.
Then $ \v^2 =-2$ if and only if $l=1$ and
$r_0 a=d_0^2 n+1$. 
In particular, $r_0 \mid (d_0^2 n+1)$.

We first assume that
either $r_0 \nmid (d_0^2 n+1)$ or $r_0\mid(d_0^2 n+1)$ but $c$ satisfies $c\geq-\frac{\langle\v,\v_0\rangle}{r_0}$, where $\v_0=(r_0,d_0H,a_0)$ and $a_0=\frac{nd_0^2+1}{r_0}$ so that $\v_0^2=-2$.
By assumption we have $E \in {\cal M}_H(\v)$ such that $H^1(X,E)=0$.  In particular, we must have $$0\leq h^0(X,E)=h^0(X,E)+h^2(X,E)=\chi(X,E)=r+a.$$ 
For a general quotient
$$
f:E \to \oplus_{i=1}^c k_{x_i},
$$
the map on global sections is either surjective or injective, depending on whether $c\leq r+a$ or $c>r+a$, respectively, and we must have $\Ker f\in {\cal M}_H(\v')^{\mu ss}$ where $\v'=(lr_0,ld_0 H,a-c)$.  Thus $H^1(X,\Ker f)=0$ or $H^0(X,\Ker f)=0$, respectively, and from stability and $ld_0>0$ we see that $H^2(X,\Ker f)=0$, so $\Ker f$ has at most one non-trivial cohomology group.  The condition $c\geq-\frac{\langle\v,\v_0\rangle}{r_0}$ is equivalent to $\langle\v',\v_0\rangle\geq 0$.  Thus by \cref{Thm:ClassicalModSp} and either Lemma 2.3 and Proposition 2.4 or Section 3.3 of \cite{Yos99a}, respectively, we have that $\MM_H(\v')^{\mu s}$ is an irreducible, open, and dense substack of $\MM_H(\v')^{\mu ss}$.  As the vanishing of $H^1$ or $H^0$, respectively, is an open condition, it follows that $\v'$ satisfies weak Brill-Noether, and the generic sheaf in $M_H(\v')$ is locally free as long as $lr_0>1$.

We next assume that $r_0 \mid (d_0^2 n+1)$ and $c<-\frac{\langle\v,\v_0\rangle}{r_0}$.  Let $E_0$ be the unique $\mu$-stable locally free sheaf with
$\v(E_0)=\v_0=(r_0,d_0 H,a_0)$, where $r_0 a_0=d_0^2 n+1$.  
Then $c<-\frac{\langle\v,\v_0\rangle}{r_0}$ is equivalent to $\langle\v',\v_0\rangle<0$ and $c\geq 1$ implies that $\langle\v,\v_0\rangle<0$ as well. 
Observe that we may write $\v=l\v_0-b\v(k_x)$ with $b\in\Z_{\geq 0}$, where $k_x$ is the skyscraper sheaf of a point $x\in X$.  
It was proven in \cite[Thm. 2.3]{Yos03} that the Fourier-Mukai functor 
$\Phi_{X \to X}^{{\cal E}[1]}:\Db(X) \to \Db(X)$ gives an isomorphism
\begin{equation}
\begin{matrix}
{\cal M}_H(l\v_0-b\v(k_x)) & \to & {\cal M}_H((br_0-l)\v_0^{\vee}-b \v(k_x))\\
E & \mapsto & \Phi_{X \to X}^{{\cal E}[1]}(E^{\vee}),
\end{matrix}
\end{equation} where 
\begin{equation}
{\cal E}:=\Ker(E_0^{\vee} \boxtimes E_0 \to {\cal O}_\Delta).
\end{equation}

Since $\langle (br_0-l)\v_0^{\vee}-b \v(k_x),\v_0^{\vee} \rangle=-\langle \v,\v_0 \rangle>0$, it follows from \cite[Section 3.3]{Yos99a} that
${\cal M}_H((br_0-l)\v_0^{\vee}-b \v(k_x))^{\mu s}$ is an open dense substack of
${\cal M}_H((br_0-l)\v_0^{\vee}-b \v(k_x))$.  Moreover, if $(br_0-l)r_0>1$, then a general member $F$ of ${\cal M}_H((br_0-l)\v_0^{\vee}-b \v(k_x))^{\mu s}$
is locally free.  By stability, for such an $F$ we have $\Hom(E_0,F^{\vee})=\Hom(F^{\vee},E_0)=0$, so it follows that a general member $E\in{\cal M}_H(\v)$ fits into an exact sequence
\begin{equation}\label{eq:E_0:1}
0 \to F^{\vee} \to E \to E_0^{\oplus (2l-br_0)} \to 0.
\end{equation}
If instead $(br_0-l)r_0=1$, then $r_0=1$ and $\v=(l,0,-1)e^{d_0 H}$.
In this case, by \cite[Proposition 3.4]{Yos99a}
we have an exact sequence
\begin{equation}\label{eq:E_0:2}
0 \to E \to {\cal O}_X(d_0 H)^{\oplus l} \to A \to 0
\end{equation}
where $A$ is a 0-dimensional torsion sheaf of length $l+1$.

We claim that our hypothesis that there exists $E\in\MM_H(\v)$ with $H^1(X,E)=0$ implies that $H^1(X,E_0)=0$.  Indeed, observe first that $H^1(X,E')=0$ for the generic $E'\in\MM_H(\v)$.  Now if $(br_0-l)r_0=1$, then $E_0=\OO_X(d_0H)$, so $H^1(X,E_0)=H^1(X,\OO_X(d_0H))=0$ by Kodaira vanishing since $d_0>0$.  Otherwise $(br_0-l)r_0\geq 2$, and the generic sheaf $E\in\MM_H(\v)$ sits in the exact sequence \eqref{eq:E_0:1}.   Then $F^\vee$ is a $\mu$-stable locally free sheaf of positive slope, so we must have $H^2(X,F^\vee)=0$.  As $H^1(X,E)=0$, it follows from the long exact sequence associated to \eqref{eq:E_0:1} that $H^1(X,E_0)=0$.

Returning to proving that $\v'$ satisfies weak Brill-Noether, we first write $\v'=\v-c\v(k_x)=l\v_0-(b+c)\v(k_x)$ and note that it suffices to assume that $\v'$ is primitive.  Since $c>0$, $\v'^2>-2$.  If either $\v'$ is isotropic or $((b+c)r_0-l)r_0=1$, then
every $E \in M_H(\v')$ is the kernel of
$E_0^{\oplus l} \to A$, where $A$ is an Artinian 
sheaf.
Indeed, suppose first that $\v'^2=0$.  Then $l=(b+c)r_0$, so $$\v'=l\v_0-(b+c)\v(k_x)=(b+c)(r_0\v_0-\v(k_x)),$$ and thus $b+c=1$ by primitivity.  As $b\in\Z_{\geq 0}$ and $c\in\N$, we must have $b=0$ and $c=1$.  Then by \cite[Section 3.3]{Yos99a} any $E'\in M_H(\v')$ sits in a short exact sequence \begin{equation}\label{eq:E_0:3}
0 \to E' \to E_0^{\oplus r_0} \to k_{x} \to 0 ,
\end{equation}
so we may take $A=k_x$ and then
$H^1(X,E')=0$ for all $E' \in M_H(\v')$.
If instead $r_0=1$ and $b+c=l+1$, then as in \eqref{eq:E_0:2}, we may take
$A=\oplus_{i=1}^{l+1} k_{x_i}$.  In this case the claim follows from \cref{Lem:E_0} below.

Otherwise $r_0 \geq 2$ or $(b+c)r_0-l \geq 2$.  By \cref{Lem:E_0}, the kernel of the generic quotient $f:E_0^{\oplus l}\to\bigoplus_{i=1}^{b+c}k_{x_i}$ satisfies $\Hom(E_0,\Ker f)=0$ and either  $H^1(X,\Ker f)=0$ or $H^0(X,\Ker f)=0$.  By \cref{prop:E_0}, $\MM_H(\v')$ is an irreducible, open, and dense substack of the same irreducible component of $\MM_H(\v')^{\mu ss}$ as $\Ker f$ for generic $f$.  The proposition then follows from the openness of $H^1$ or $H^0$ vanishing as before.
\end{proof}

In the remainder of the section we prove the lemmas cited in the proof of \cref{prop:minimal-a}.  Recall that we write $\v=l\v_0-e\v(k_x)$ where $\v_0=(r_0,d_0H,a_0)$ such that $\v_0^2=-2$, and we may assume that $H^1(X,E_0)=0$.  Moreover, we may assume that $\v^2\geq 0$ and $\langle\v,\v_0\rangle<0$ so that $2l>e r_0\geq l$.

\begin{Lem}\label{Lem:l}
Let $p\leq r_0 t$.  Then for a general quotient $f:E_0^{p} \to T$, with $T$ a 0-dimensional torsion sheaf of length $t$, 
$\Hom(E_0,\Ker f)=0$.
\end{Lem}

\begin{proof}
%Let $f_1,f_2,...,f_{r_0}$ be a basis of 
%$\Hom(E_0,{\Bbb C}_x) \cong {\Bbb C}^{oplus r_0}$.

Let $E_0 \otimes \Hom(E_0,T) \to T$
be the evaluation map.  As $\hom(E_0,T)=r_0t$, for a $p$-dimensional subspace $V \subset \Hom(E_0,T)$,
we consider
$$
f:E_0 \otimes V \to E_0 \otimes \Hom(E_0,T) \to T.
$$
Then 
$\Hom(E_0,E_0 \otimes V) \to \Hom(E_0,E_0 \otimes \Hom(E_0,T))$
is injective and
$$
\Hom(E_0,E_0 \otimes \Hom(E_0,T))\to \Hom(E_0,T)
$$
is an isomorphism.
Hence 
$\Hom(E_0,\Ker f)=0$.
\end{proof} 

\begin{Lem}\label{Lem:E_0}
Assume that $e r_0\geq l$, that is, $ \v^2 \geq 0$.
For a general quotient
$f:E_0^{\oplus l} \to \oplus_{i=1}^e k_{x_i}$,
$\Hom(E_0,\Ker f)=0$ and $\Ker f$ has at most one non-trivial cohomology group.
\end{Lem}

\begin{proof}
We write $l=m r_0+p$ ($0 \leq p <r_0$).
Since $er_0\geq l$, $(e-m)r_0\geq p \geq 0$.
We set 
$$
{\cal E}_x:=\Ker(E_0 \otimes \Hom(E_0,k_x) \to {k}_x).
$$
Let $F$ be the kernel of a general quotient
$E_0^{\oplus p} \to \oplus_{i=1}^{e-m}k_{y_i}$.
As $p\leq (e-m)r_0$, it follows from Lemma \ref{Lem:l} that both $\Hom(E_0,\EE_x)$ and $\Hom(E_0,F)$ vanish.
Then $E:=\oplus_{i=1}^m {\cal E}_{x_i} \oplus F$ satisfies
$\Hom(E_0,E)=0$.
Since $\Hom(E_0,\Ker f)=0$ is an open condition,
we get the first claim.

If we further assume that $l \chi(E_0) \geq e$, then the generic such $f$ induces a surjection on global sections, so $H^1(X,E_0)=0$ implies that $H^1(X,\Ker f)=0$ for a general $f$.  On the other hand, if $e>l\chi(E_0)$, then the generic such $f$ induces an injection on global sections, so $H^0(X, \Ker f)=0$. Either way, $\Ker f$ has at most one nontrivial cohomology group, as required.
\end{proof}

\begin{Lem}\label{prop:E_0}
Let ${\cal M}_H(\v)^{\mu ss,0}$ be the open substack
of ${\cal M}_H(\v)^{\mu ss}$ consisting of 
$E \in {\cal M}_H(\v)^{\mu ss}$ such that 
$\Hom(E_0,E)=0$.
If $r_0 \geq 2$ or $e r_0-l \geq 2$, then
 ${\cal M}_H(\v)$ is an open and dense
substack of ${\cal M}_H(\v)^{\mu ss,0}$.
In particular ${\cal M}_H(\v)^{\mu ss,0}$ is
irreducible.
\end{Lem}

\begin{proof}
The proof is similar to \cite[Lem. 2.3]{Yos99a}.
For the Harder-Narasimhan filtration
$$
0 \subset F_1 \subset F_2 \subset \cdots \subset F_s=E
$$
of $E$,
$F_i/F_{i-1}$ are semi-stable sheaves with 
$\v(F_i/F_{i-1})^2  \geq 0$.
We set $\v_i:=\v(F_i/F_{i-1})$.
Then it is sufficient to prove
\cite[(2.11), (2.13)]{Yos99a}.

We first assume that $r_0 \geq 2$.
Then \cite[(2.10)]{Yos99a} holds.
Hence \cite[(2.11), (2.13)]{Yos99a} hold.

We next assume that 
$r_0=1$ and $e r_0-l \geq 2$.
In this case, we may assume that $\v_0=\v({\cal O}_X)$
and $\v=(l,0,-a)$ with $a \geq 2$.
If $\v_1^2 >0$, then
$\langle \v_i,\v_j \rangle \geq 2$.
Hence \cite[(2.11)]{Yos99a} holds.
If $\v_1^2 =0$, then
we write $\v_1=l_1 \v_1'$, where $\v_1'$ is primitive.
Observe that $$\langle \v_1,\v-\v_1 \rangle-l_1'=l_1'(\langle \v_1',\v \rangle-1) \geq l_1'.$$
Hence 
\begin{equation}
\begin{split}
& \sum_{i<j} \langle \v_i,\v_j \rangle-(l_1'+s-1)+1\\
= & (\langle \v_1,\v-\v_1 \rangle-l_1')+
\sum_{1<i<j}\langle \v_i,\v_j \rangle-(s-1)+1\\
\geq & l_1'+(s-1)(s-2)-(s-1)+1 \geq 1.
\end{split}
\end{equation}
Hence \cite[(2.13)]{Yos99a} holds.

The irreducibility is a consequence of \cite[Thm. 1.4]{Yos03}.
\end{proof}

\section{Consequences of Theorem \ref{thm:summing up inequalities}}\label{sec-mainconsequences}
Let $X$ be a K3 surface with $\Pic(X)=\mathbb{Z} H$ and $H^2 =2n$.  In this section, we derive  consequences of Theorem \ref{thm:summing up inequalities}. 

\subsection{Uniform bounds on $n$} We first show that if $n \geq r$, then the generic sheaf in $M_H(\v)$ has no higher cohomology. We begin by noting several useful numerical observations.

\begin{Lem}\label{lem:rewrite product}
Let $\v = (r,dH, a)$ be a Mukai vector with $r\geq 0, d,a >0$ and $\v^2\geq -2$.  Let $\v_1=(r_1, d_1H, a_1) \in D_{\v}$. Set $k = da_1 - a d_1$ and $m=d r_1 -  d_1r$. Then 
\begin{equation}\label{eqn-maininequality}
   \langle \v,\v_1\rangle = \frac{1}{r_1 d_1} (m(nd_1^2-1)- 2rd_1+k r_1^2) 
\end{equation}
\end{Lem}
\begin{proof}
We have $\langle \v,\v_1\rangle =2ndd_1-r_1 a-ra_1.$  
By Remark \ref{rem-numerics} \ref{enum:spherical},  $a_1 = \frac{nd_1^2+1}{r_1}$. By assumption $a = \frac{da_1 - k}{d_1}$. Substituting for $a$ and $a_1$ using these two relations, we obtain
\begin{align*}
    \langle \v,\v_1\rangle &= 2ndd_1 -\frac{1}{d_1} (d (nd_1^2 +1) - kr_1) - \frac{r}{r_1} (nd_1^2 +1) \\ &=\frac{1}{r_1 d_1} ((d r_1 -  d_1r)(nd_1^2-1)- 2rd_1+k r_1^2).
\end{align*}
This is the desired formula.
\end{proof}

\begin{Lem}\label{lem:useful consequences}
Let $\v = (r,dH, a)$ be a Mukai vector with $r\geq 2, d,a >0$ and $\v^2\geq -2$. Set $m=d r_1 -  d_1r$.
\begin{enumerate}
\item If $\v_1\in D_\v$, then $r_1<r$ and $$d_1<\min\left( \frac{2r}{n}, \frac{2r}{mn} + \frac{1}{\sqrt{n}}\right).$$ 
\item Moreover, if $\v_1 \in D_\v^{BN}$, then $d_1 < \frac{2r}{mn}$.
\end{enumerate}
\end{Lem}
\begin{proof}
Let $\v_1 \in D_{\v}$. Set $k = da_1 - a d_1$ and $m =r_1 d-rd_1$. By Remark \ref{rem-numerics} \ref{enum:positive r_1} and \ref{enum:positive a_1} both $k$ and $m$ are positive integers. First, we show that $r_1<r$.  Using Remark \ref{rem-numerics} \ref{enum:totally semistable spherical} and Lemma \ref{lem:rewrite product}, we have 

\begin{align*}
0>\langle \v,\v_1\rangle &= \frac{1}{r_1 d_1} (m(nd_1^2-1)- 2rd_1+k r_1^2) \geq\frac{1}{r_1 d_1}((nd_1^2-1)-2rd_1 +k r_1^2)\\
&=\frac{1}{r_1 d_1}((n-1)d_1^2+(d_1-r)^2-r^2+k r_1^2-1)\geq\frac{1}{r_1 d_1}(r_1^2-r^2-1)\\
&=\frac{1}{r_1d_1}((r_1-r)(r_1+r)-1),
\end{align*}
Hence, $r_1\leq r$.  Since $\v_1^2=-2$, it follows that $\gcd(r_1,d_1)=1$.  If $r_1=r$, then $d_1\neq r$ since $r\geq 2$.  Thus $(d_1-r)^2\geq 1$ so that if $r_1=r$ then \begin{equation}\label{eqn:r_1<r and d_1<2r} 0>(n-1)d_1^2+(d_1-r)^2-r^2+r_1^2-1\geq(n-1)d_1^2\geq 0,\end{equation} a contradiction.  Thus we have $r_1<r$.  

For the inequalities on $d_1$, observe that as $m$ and $n$ are positive integers, the quantity
\begin{equation}\label{eqn:nonnegative quantity}mnd_1^2 - 2r d_1 - m + kr_1^2
\end{equation}
is nonnegative unless 
\begin{equation}\label{eqn:quadratic formula}
    d_1 < \frac{2r + \sqrt{4r^2 +4m^2n-4mnkr_1^2}}{2mn}.
\end{equation}
If $m=1$, since $4nk r_1^2 \geq 4n$, we conclude that $d_1 < \frac{2r}{n}$ as desired. Similarly, if $\v_1 \in D_{\v}^{BN}$, then $m \leq k$ and we conclude that $d_1 < \frac{2r}{mn}$. This concludes the proof of part (2) of the lemma.
Returning to the case when $\v_1\in D_\v$ and  $m \geq 2$ and noting that $$\sqrt{4r^2 +4m^2n-4mnkr_1^2} < 2r + 2m \sqrt{n},$$ we see that  
\begin{equation}\label{eqn:more precise}
d_1 < \frac{2r}{mn} + \frac{1}{\sqrt{n}} \leq \frac{r}{n} + \frac{1}{\sqrt{n}}.
\end{equation}

To conclude the proof of the lemma, let us show that $d_1 < \frac{2r}{n}.$  If instead $d_1 \geq \frac{2r}{n}$, then it follows that $\frac{r}{n} <\frac{1}{\sqrt{n}} \leq 1$. Since $d_1$ is a positive integer, we must have $d_1 =1$, so $n \geq 2r$. Since $m \geq 2$, the quantity in \eqref{eqn:nonnegative quantity} satisfies $$mn-2r-m+kr_1^2 \geq 2rm - 2r -m +1 = (2r-1)(m-1) >0,$$ a contradiction. 
\end{proof}

\begin{Thm}\label{thm:bound n r}
 Let $\v=(r,dH,a)$ be a Mukai vector with $r\geq 2$, $d,a >0$    and $\v^2\geq -2$ on a K3 surface $X$ with $\Pic(X) = \mathbb{Z} H$ and $H^2 =2n$. If $n \geq r,$ then $H^1(X,E)=0$ for the generic sheaf $E\in M_H(\v)$.
\end{Thm}

\begin{proof}
By Proposition \ref{prop:minimal-a}, we may assume that $a = \left\lfloor \frac{nd^2 +1}{r} \right\rfloor$. 
By Theorem \ref{thm:summing up inequalities}, it suffices to prove that $D_\v^{BN} = \varnothing$.  
Assume to the contrary that $\v_1= (r_1, d_1H, a_1) \in D_\v^{BN}$. Set $m = d r_1 - r d_1$. By Lemma \ref{lem:useful consequences}, we have $d_1 < \frac{2r}{mn}.$ Since $d_1$ is a positive integer, $n \geq r$ implies that $m=d_1=1$. Hence, $d= \frac{r+1}{r_1}$ and $a_1 = \frac{n+1}{r_1}$. In particular, $r_1$ divides both $r+1$ and $n+1$.  Moreover, since $n \geq r$ $$a = \left\lfloor \frac{nd^2 +1}{r} \right\rfloor = \left\lfloor \frac{n(r+1)^2 +r_1^2}{rr_1^2} \right\rfloor \geq \left\lfloor \frac{r(n+1)(r+1) + r_1^2}{r r_1^2} \right\rfloor = \frac{(n+1)(r+1)}{r_1^2}.$$
By Remark \ref{rem-numerics}\ref{enum:H^1 vanishing}, we have 
$$1 \leq a_1 d - a \leq  \frac{(n+1)}{r_1} \frac{(r+1)}{r_1} - \frac{(n+1)(r+1)}{r_1^2}=0,$$ a contradiction. 
\end{proof}

\begin{Rem}
Theorem \ref{thm:bound n r} is sharp. Let $X$ be a K3 surface as in Theorem \ref{thm:bound n r}. Let $$\v = (n+1, (n+2)H, n^2+3n+1).$$ Then the unique bundle $E \in M_H(\v)$ has resolution given by 
$$0 \to \OO_X \to \OO_X(H)^{n+2} \to E \to 0$$ and $h^1(X,E)=1$.
\end{Rem}

\subsection{Uniform cohomology vanishing} In this subsection, we give a uniform effective bound on $d$ that guarantees that weak Brill-Noether holds.

\begin{Thm}\label{thm:WBN sharp} 
Let $\v = (r, dH, a)$ be a Mukai vector with $r \geq 2, a \geq 0, d>0$ and $\v^2 \geq -2$ on a K3 surface with $\Pic(X) = \mathbb{Z} H$ and $H^2=2n$.
If $$d \geq r \left\lfloor  \frac{r}{n} \right\rfloor +2, $$ then $\v$ satisfies weak Brill-Noether.
\end{Thm}

\begin{proof}

By Proposition \ref{prop:minimal-a}, it suffices to prove the theorem under the additional assumption that $a = \left\lfloor \frac{nd^2+1}{r}\right\rfloor$. By Theorem \ref{thm:bound n r}, we may also assume that $r>n$, hence $\left\lfloor \frac{r}{n} \right\rfloor \geq 1$. In the proof, we will make these additional simplifying assumptions. By Theorem \ref{thm:summing up inequalities}, it suffices to show that $D_{\v}^{BN} = \varnothing$. 

Suppose to the contrary that $\v_1 \in D_\v^{BN}$. Set $m = r_1 d - r d_1$ and $k = da_1 - a d_1$. By Lemma \ref{lem:useful consequences} (2), we have that $d_1 < \frac{2r}{mn}$.
We will use the following observation several times.

\begin{Rem}\label{obs-r1}If $r_1=1$, then the inequality $m \leq d a_1 - a d_1$ implies that $d_1 \leq \frac{r}{nm}$.  Indeed, if $r_1 =1$, then $d=m+rd_1$ and $a_1 = nd_1^2 +1$. Hence $$m \leq (m+rd_1)(nd_1^2+1) - \left\lfloor \frac{n (m+rd_1)^2 +1}{r}\right\rfloor d_1 =m + rd_1- nmd_1^2 - \left\lfloor \frac{nm^2+1}{r} \right\rfloor d_1.$$ Rearranging, we see that $nm d_1^2 \leq r d_1$ and thus $d_1\leq\frac{r}{mn}$, as claimed.\end{Rem}

Returning to the proof of the theorem, if $m \geq 4$, then $d_1 \leq \left\lfloor \frac{r}{2n} \right\rfloor.$ Since $d_1$ is a positive integer, $\left\lfloor \frac{r}{2n} \right\rfloor \geq 1$. From $\left\lfloor \frac{r}{n} \right\rfloor \geq 2 \lfloor \frac{r}{2n} \rfloor$, we see that $$m = r_1 d - r d_1 \geq r_1 r \left\lfloor \frac{r}{n}\right \rfloor + 2 r_1 - r \left\lfloor \frac{r}{2n} \right\rfloor = (2r_1 -1) r \left\lfloor \frac{r}{2n}\right\rfloor + 2 r_1 \geq (2r_1-1) r + 2r_1 \geq r+2.$$ Hence, $d_1 < \frac{2r}{n((2r_1-1)r+2r_1)}$. Since $d_1$ is a positive integer, we must have $n=r_1=d_1=1$. In that case, $m=d-r\geq r^2 -r+2 \geq 2r$ since $r\geq 2$. Hence, $d_1 < 1$, which is a contradiction.  Thus we must have $1\leq m\leq 3$.

If $m=3$, then $d_1 \leq \left\lfloor \frac{2r}{3n} \right\rfloor$. Since $d_1$ is a positive integer, we may assume $\left\lfloor \frac{2r}{3n} \right\rfloor \geq 1$. In that case, observe that
$$\left\lfloor \frac{r}{n} \right\rfloor > \left\lfloor \frac{2r}{3n}  \right\rfloor \quad  \mbox{unless} \quad \left\lfloor \frac{r}{n} \right\rfloor = \left\lfloor \frac{2r}{3n} \right\rfloor=1.$$
If $\left\lfloor \frac{r}{n} \right\rfloor > \left\lfloor \frac{2r}{3n} \right\rfloor$, then
$$3=m \geq r_1 r \left\lfloor \frac{r}{n} \right\rfloor + 2 r_1 - r \left\lfloor \frac{2r}{3n} \right\rfloor \geq r_1 r + 2r_1 \geq 4,$$
a contradiction. Hence, we must have $\left\lfloor \frac{r}{n} \right\rfloor = \left\lfloor \frac{2r}{3n} \right\rfloor=1$ and 
$$m=3 \geq r_1 r + 2r_1 -r = (r_1-1)r + 2r_1.$$
We conclude that $r_1=d_1=1$ and $d=r+3$. By Remark \ref{obs-r1}, we must have $1= d_1 \leq \frac{r}{3n} < \frac{2}{3}$, which is a contradiction. We conclude that $m \leq 2.$

If $m=2$, then $d_1 \leq \left\lfloor \frac{r}{n} \right\rfloor.$ Hence,$$m \geq r_1 r \left\lfloor \frac{r}{n} \right\rfloor + 2 r_1 - r \left\lfloor \frac{r}{n} \right\rfloor \geq (r_1-1) r \left\lfloor \frac{r}{n} \right\rfloor + 2r_1.$$ If $r_1 \geq 2$, it follows that $2=m \geq r+4 >2$, a contradiction. Thus  $r_1=1$.  Hence, by Remark \ref{obs-r1}, we must have $d_1 \leq \left\lfloor \frac{r}{2n} \right\rfloor$. Since $d_1$ is a positive integer, we may assume the latter is positive. Then, we conclude that $$m =2 \geq r \left\lfloor \frac{r}{n} \right\rfloor + 2 - r \left\lfloor \frac{r}{2n} \right\rfloor \geq r \left\lfloor \frac{r}{2n} \right\rfloor + 2 > 2,$$ a contradiction.

We must therefore have $m=1$. Note that $2 \left\lfloor \frac{r}{n} \right\rfloor + 1 \geq \left\lfloor \frac{2r}{n} \right\rfloor.$ 
If $r_1 \geq 3$, then $$m=1 \geq r_1 r \left\lfloor \frac{r}{n} \right\rfloor + 2 r_1 - r \left\lfloor \frac{2r}{n} \right\rfloor \geq 3r \left\lfloor \frac{r}{n} \right\rfloor + 2r_1 - r \left(2\left\lfloor\frac{r}{n}\right\rfloor+1\right) \geq 2r_1 \geq 6,$$ a contradiction. 
On the other hand, if $r_1=1$, then by Remark \ref{obs-r1} $d_1 \leq \left\lfloor \frac{r}{n} \right\rfloor$. 
Hence, $$m =1 \geq r \left\lfloor \frac{r}{n} \right\rfloor + 2 - r \left\lfloor \frac{r}{n} \right\rfloor \geq 2,$$ another contradiction. 

We conclude that $m=1$ and $r_1=2$. In this case, the inequality $1 \leq a_1 d - a d_1$  implies that $d_1 < \frac{r+3}{n}$. Indeed, from the equalities $2d=1+rd_1$, $2a_1= nd_1^2 +1$ and $a = \left\lfloor \frac{nd^2+1}{r} \right\rfloor$, we see that $n, d_1$ and $r$ are all odd and 
$$a \geq \frac{nr d_1^2 + 2nd_1}{4} - \frac{3}{4}.$$ Hence,
$$1  \leq a_1d - ad_1 \leq  \frac{rnd_1^3 + nd_1^2 + rd_1+1}{4} - \frac{nrd_1^3 + 2nd_1^2 }{4} + \frac{3}{4}d_1   \leq  \frac{(r+3)d_1 - nd_1^2 + 1 }{4}.$$
This easily implies $d_1 < \frac{r+3}{n}$. If $n>1$, then  $n \geq 3$ since $n$ is odd. Then
$d_1 \leq \frac{r}{n} + 1.$ We obtain
$$ m=1 \geq 2 r \left\lfloor \frac{r}{n} \right\rfloor + 4 - r \left\lfloor \frac{r}{n} \right\rfloor - r \geq 4,$$ which is a contradiction.
We thus conclude $m=n=1$ and $r_1=2$ and $d_1 \leq r+2$. Substituting
$$m=1 \geq 2r^2 + 4 - r^2 -2r = r^2 -2r + 4 \geq 4.$$ This contradiction shows that $D_{\v}^{BN}= \varnothing$ and completes the proof of the theorem.
\end{proof}

\begin{Rem}
Theorem \ref{thm:WBN sharp} is sharp. For example, setting $\v = (n+1, (n+2)H, n^2+3n+1)$, we see that the generic sheaf in $M_H(\v)$ has nonvanishing $H^1$. Furthermore, if $d= r \left\lfloor \frac{r}{n} \right\rfloor +1$, then $D_{\v}^{BN}$ is not always empty. For example, let $$\v= \left(r, \left(r \left\lfloor \frac{r}{n} \right\rfloor +1\right) H, nr\left\lfloor \frac{r}{n} \right\rfloor^2 + 2n \left\lfloor \frac{r}{n} \right\rfloor + \left\lfloor \frac{n+1}{r} \right\rfloor \right).$$ Then $\v_1 = (1,\left\lfloor \frac{r}{n} \right\rfloor H, n \left\lfloor \frac{r}{n} \right\rfloor^2 +1) \in D_\v^{BN}.$  The cohomology of the generic sheaf may vanish even when $D_\v^{BN} \not= \varnothing$, but this will require a more detailed analysis, which we will undertake in the rest of the paper. 
\end{Rem}

\subsection{Finiteness of counterexamples} As a corollary of our discussion so far, we see that given $r \geq 2$ there are only finitely many Mukai vectors $\v= (r, dH, a)$ with $\v^2 \geq -2$ such that the generic sheaf in $M_H(\v)$ has more than one nonzero cohomology group. Starting in \cref{sec-initialclassification}, we will turn to the classification of these Mukai vectors.

\begin{Thm}\label{thm:finiteness}
Let $X$ be a K3 surface with $\Pic(X) = \mathbb{Z} H$ and $H^2=2n$. Fix $r \geq 2$. Then there are finitely many tuples $(n, \v=(r,dH,a))$ with $d >0$ such that $\v$ does not satisfy weak Brill-Noether.
\end{Thm}

\begin{proof}
By Theorem \ref{thm:bound n r} and Proposition \ref{prop:minimal-a}, $\v$ satisfies weak Brill-Noether unless $n<r$. Hence, there are only finitely many possible values for $n$. Fix $n$. By Theorem \ref{thm:WBN sharp} and Proposition \ref{prop:minimal-a},  $\v$ satisfies weak Brill-Noether unless $d \leq r \left\lfloor \frac{r}{n} \right\rfloor +1$. Hence, for each $n$, there are only finitely many possible values of $d$ for which $\v$ fails weak Brill-Noether.  Fix $n$ and $d$. By Proposition \ref{Prop:a<=0}, if $a \leq 0$, then $\v$ satisfies weak Brill-Noether. Since $\v^2 = 2nd^2 - 2ra \geq -2,$ we always have that $nd^2 +1 \geq ra$. Hence, there are finitely many possible values of $a$ for which $\v$ fails weak Brill-Noether.
\end{proof}

\begin{Rem}
Given an arbitrary, fixed polarized surface $(X,H)$ and a fixed rank $r \geq 2$, there are only finitely many Chern characters where the moduli space fails to satisfy weak Brill-Noether (see \cite[Theorems 3.6 and 3.7]{CoskunHuizenga:ICM}). \cref{thm:finiteness} shows that this finiteness remains true even if we vary $X$ over all K3 surfaces of Picard rank one.  
\end{Rem}

\subsection{Uniform global generation} In this subsection, we obtain a uniform sufficient condition for the generic sheaf in $M_H(\v)$ to be globally generated.

\begin{Prop}\label{prop:uniform gg n}
Let $\v=(r,dH,a)$ be a Mukai vector with $r\geq 2$, $d>0$, $a\geq 2$ and $\v^2\geq -2$ on a K3 surface $X$ with $\Pic(X) = \mathbb{Z} H$ and $H^2 =2n$. If $n \geq 2r$, then  $D_\v=\varnothing$ and the generic $E\in M_H(\v)$ is globally generated.
\end{Prop}

\begin{proof}
Under the assumptions of the theorem,  Theorem \ref{thm:bound n r} implies that the higher cohomology of the generic sheaf in $M_H(\v)$ vanishes. By Theorem \ref{thm:summing up inequalities}, if $D_{\v} = \varnothing$ and $M_H(a, dH, r)$
has a locally free sheaf, then the generic sheaf in $M_H(\v)$ is globally generated. By Proposition \ref{Prop:OnlyNonLocallyFreeSheaves}, if $M_H(a, dH, r)$ does not have any locally free sheaves, then:
\begin{enumerate}
\item Either $\v^2 >0$ and $(a, dH, r)$ has the form $(l, lp H, lp^2n-1)$ or $(1, pH, p^2n-1)$ for some integers $p,l$. The assumption $a\geq 2$ rules out the second possibility. The assumption $n \geq 2r$ implies that $lp^2n-1 \geq 2r-1>r$ (as $r\geq 2$) and  precludes the first possibility.
\item Or $\v^2 =0$ and $(a, dH, r) = b (r_0^2, r_0d_0 H, d_0^2 n)$ for integers $b,r_0, d_0$ which satisfy $d_0^2 n - r_0 a_0=-1$ for some integer $a_0$. The assumption $n \geq 2r$ precludes this possibility since $d_0^2 n>r$.
\end{enumerate}
We are thus reduced to checking that $D_{\v} = \varnothing$. Suppose $\v_1 = (r_1, d_1H, a_1) \in D_{\v}$. By Remark \ref{rem-numerics}\ref{enum:totally semistable spherical} and Lemma \ref{lem:rewrite product}, we have 
$$0 > m(nd_1^2 -1) - 2rd_1 + kr_1^2 \geq 2r (d_1^2 -d_1) -1 + kr_1^2 \geq 0,$$
a contradiction.  Hence, $D_{\v} = \varnothing$ and the proposition holds.

\end{proof}

\begin{Rem}
\cref{prop:uniform gg n} is sharp. If $n \leq 2r-1$ and $\v=(r, (r+1)H, a)$ with 
$$rn+2n-r< a \leq \min\left(nr+n+r,nr+2n+\left\lfloor\frac{n+1}{r}\right\rfloor\right),$$ then $\v_1=(1, H, n+1) \in D_{\v}$. We will analyze these cases in greater detail in the rest of the paper. 
\end{Rem}

\begin{Thm}\label{thm:sufficient condition on d} Let $\v=(r,dH,a)$ be a Mukai vector with $r\geq 2$, $d>0$, $a\geq 2$ and $\v^2\geq -2$ on a K3 surface $X$ with $\Pic(X) = \mathbb{Z} H$ and $H^2 =2n$. Assume that  $$d \geq r\left\lfloor \frac{2r}{n} \right\rfloor +r.$$ If $n=1$, assume further that $2d \geq 2a + r$. Then $D_\v=\varnothing$ and the generic $E\in M_H(\v)$ is globally generated.
\end{Thm}
\begin{proof}
As 
$$r\left\lfloor\frac{2r}{n}\right\rfloor+r\geq2r\left\lfloor\frac{r}{n}\right\rfloor+2\geq r\left\lfloor\frac{r}{n}\right\rfloor+2,$$
the higher cohomology of the generic sheaf in $M_H(\v)$ vanishes by Theorem \ref{thm:WBN sharp}.
By Proposition \ref{prop:uniform gg n}, we may assume that $n < 2r$. Consequently, $d \geq 2r$.  By Theorem \ref{thm:summing up inequalities}, if $D_{\v} = \varnothing$ and $M_H(a, dH, r)$
has a locally free sheaf, then the generic sheaf in $M_H(\v)$ is globally generated.  By Proposition \ref{Prop:OnlyNonLocallyFreeSheaves}, if $M_H(a, dH, r)$ does not have any locally free sheaves, then:
\begin{enumerate}
\item Either $\v^2 >0$ and $(a, dH, r)$ has the form $(l, lp H, lp^2n-1)$ or $(1, pH, p^2n-1)$ for some integers $p,l$. The assumption $a\geq 2$ rules out the second possibility and the assumption $d \geq 2r \geq r+2$ precludes the first possibility.
\item Or $\v^2 =0$ and $(a, dH, r) = b (r_0^2, r_0d_0 H, d_0^2 n)$ for integers $b,r_0, d_0$ which satisfy $d_0^2 n - r_0 a_0=-1$ for some integer $a_0$. By assumption $$d = br_0d_0 \geq r\left\lfloor \frac{2r}{n} \right\rfloor +r = 2 n b^2 d_0^4 + bd_0^2n.$$ Hence, $r_0 \geq 2nb d_0^3 + d_0 n$. Then there cannot be any integers $a_0$ that satisfy $d_0^2 n - r_0 a_0=-1$.
\end{enumerate}
We conclude that under the assumptions of the theorem, $M_H(a, dH, r)$ has locally free sheaves. To prove the theorem it suffices to show $D_{\v} = \varnothing$.

For a contradiction, assume that $\v_1=(r_1,d_1H,a_1)\in D_\v$.  Set $m= r_1d -rd_1$.  
By Lemma \ref{lem:useful consequences} (1), $d_1 \leq \left\lfloor \frac{2r}{n}\right\rfloor$, so $n \leq 2r$ since $d_1$ is a positive integer.
Hence,
$$m \geq r_1 r \left\lfloor \frac{2r}{n} \right\rfloor + r r_1  -r \left\lfloor \frac{2r}{n}\right\rfloor = (r_1-1) r\left\lfloor \frac{2r}{n} \right\rfloor + r r_1.$$
If $r_1\geq 2$, then $m \geq 3 r$. By Lemma \ref{lem:useful consequences} (1), $$d_1 < \frac{2r}{mn} + \frac{1}{\sqrt{n}} \leq \frac{2}{3 n} + \frac{1}{\sqrt{n}}.$$
Hence, $d_1=1$ and $n=1$ or $2$. Since $r_1 a_1 = nd_1^2 +1$, we conclude that  either $n=1$ and $\v_1 = (2, H, 1)$ or $n=2$ and $\v_1 = (3, H, 1)$. In the latter case, $m \geq 5r$ and $1= d_1 < \frac{1}{5} + \frac{1}{\sqrt{2}},$ which is a contradiction. If $n=1$ and $\v_1 =(2, H, 1)$, 
$$0> \langle \v, \v_1 \rangle = 2d -2a -r.$$ The latter is at least $0$ by assumption, which is a contradiction.

We may assume that $r_1=1$. Then $m \geq r$. Hence, $d_1 < \frac{2}{n} + \frac{1}{\sqrt{n}}.$ We conclude that the possible pairs are $(n,d_1) = (1,1), (1,2), (2,1)$ or $(3,1)$.
 If $d_1=1, n=3$, then 
$$0>\langle \v, \v_1 \rangle = 2m - 2r + k \geq 2r -2r+k = k >0,$$ a contradiction. 
If $d_1=1, n=2$, then $m = d -r \geq r^2.$ Since $r \geq 2$, $$0>\langle \v, \v_1 \rangle \geq r^2 - 2r +k \geq k > 0,$$ another contradiction.
If $d_1=2, n=1$, then $m=d-2r \geq 2r^2-r = r(2r-1) \geq 3r$. Hence, $$0>\langle \v, \v_1 \rangle = \frac{1}{2} (3m -4r + k)  \geq\frac{1}{2}\left( 9r - 4r +k \right)> 0,$$ which is a contradiction.
Finally, if $r_1=d_1=n=1$, then  $\v_1 = (1, H, 2)$. In that case, since $2d \geq 2a+r$
$$\langle \v, \v_1 \rangle = 2d - a - 2r \geq d - \frac{3}{2}r \geq r^2 - \frac{r}{2} >0.$$ This concludes the proof that under  the assumptions of the theorem $D_\v = \varnothing$ and with it the proof of the theorem. 
\end{proof}

\begin{Rem}\label{rem:gg n is 1}
When $n=1$, $X$ is a double cover of $\mathbb{P}^2$ and the line bundles $\mathcal{O}_X(H)$ and $\mathcal{O}_X(2H)$ are not very ample and their sections are pullbacks of sections of $\mathcal{O}_{\mathbb{P}^2}(1)$ and $\mathcal{O}_{\mathbb{P}^2}(2)$, respectively. Even when $r=1$, $I_Z(2H)$ is not globally generated when the length of the zero-dimensional scheme is at least $1$.
\end{Rem}

\begin{Rem}\label{rem-gg}
Being globally generated is not an open condition. However, it is an open condition in the locus where the higher cohomology vanishes. To characterize the Mukai vectors $\v=(r,dH,a)$ that satisfy weak Brill-Noether and for which the generic sheaf in $M_H(\v)$ is globally generated, one may concentrate on the case $a \geq 2$. Indeed, recall that $\chi(\v)=r+a$, so if $\v$ satisfies weak Brill-Noether and the generic $E\in M_H(\v)$ is globally generated, then $\chi(\v)\geq 0$ and there is a short exact 
$$0\to M\to\OO_X^{r+a}\mor[f] E\to 0.$$
In particular, we must have $a\geq 0$ for the evaluation map $f\colon\OO_X^{r+a}\to E$ to be surjective.  If $a=0$, then $h^0(E)=r$ and if $M\ne 0$, it would have to be a torsion sheaf, which is impossible, so $f\colon E \isomor \OO_X^r$, contradicting $a=0$.  Finally, if $a=1$, then $E$ has a resolution of the form
$$0 \to \OO_X(-dH) \to \OO_X^{r+1} \to E \to 0.$$
Equating Mukai vectors, we see that $r= 1 + nd^2$, so $\v=(nd^2+1,dH,1)$ and $\v(E)^2=-2$.  The unique $E\in M_H(\v)$ satisfies $E\cong\Phi_{X\to X}^{I_\Delta^\vee[2]}(\OO_X(-dH))$.
\end{Rem}

\section{Initial Classifications}\label{sec-initialclassification}
In this section, we classify Mukai vectors with small invariants for which weak Brill-Noether fails. We will concentrate on the cases when $0\leq r \leq 3$, when $r$ is small relative to $n$, and when $0 \leq a \leq 2$. We will study the Bridgeland resolution at the maximal totally destabilizing wall in greater detail to compute the cohomology.

\subsection{Quotients of negative rank}
The quotients in the Bridgeland resolution may have negative rank. The following lemma describes the possibilities. 

\begin{Prop}
Let $r,d,a>0$ and  $(s,t)\in U_+$.
\begin{enumerate}
    \item If $(s,t)$ is sufficiently close to $(0,0)$, then $M_{(s,t)}(-r,dH,a)\cong M_H(a,dH,-r)$ given by $E\mapsto \Phi_{X\to X}^{I_\Delta}(E)^\vee$.
    
    \item If $r>a$, then there is a unique totally semistable wall $W$. For $1 \leq i \leq 2$, let $\sigma_i:=\sigma_{(sH,t_i H)}$ be stability conditions in the two chambers separated by the wall with
$t_2 < 1 < t_1$.
\begin{enumerate}
\item
For the generic $E \in M_{\sigma_1}(-r,dH,a)$, 
$\Phi_{X \to X}^{I_\Delta}(E)^{\vee}$ is a two-term complex fitting in the distinguished triangle 
\begin{equation}\label{eq:(a,d,-r)}
(\OO_X[1])^{\oplus (r-a)} \to \Phi_{X \to X}^{I_\Delta}(E)^{\vee} \to R_{\OO_X[1]}(\Phi_{X \to X}^{I_\Delta}(E)^{\vee}), 
\end{equation}
where $R_{\OO_X[1]}(\Phi_{X \to X}^{I_\Delta}(E)^{\vee})\in M_H(r,dH,-a)$
\item
If $E \in M_{\sigma_2}(-r,dH,a)$ is generic, then
$\Phi_{X \to X}^{I_\Delta}(E)^{\vee} \in M_H(a,dH,-r)$.
\end{enumerate}
\item
If $r \leq a$,
Then $\Phi_{X \to X}^{I_\Delta}(E)^{\vee} \in M_H(a,dH,-r)$ for a generic
$E \in M_{\sigma_{(sH,tH)}}(-r,dH,a)$.
\end{enumerate}
\end{Prop}

\begin{proof}
Since $a>0$, \cref{prop:isom} implies part (1).

We now set $\v:=(-r,dH,a)$ and classify the totally semistable walls for $\v$.  By part (1) for $\v'=(a,dH,-r)$, any $F\in M_H(\v')$ has the form $F=\Phi_{X\to X}^{I_\Delta}(E)^\vee$ where $E\in M_{(s,t)}(\v)$ and $(s_0,t_0)$ is sufficiently close to $(0,0)$.  By \cref{Prop:a<=0}, there are no totally semistable walls for $\v'$ in $U_+$ unless $\chi(\v')=a-r<0$, in which case the unique totally semistable wall in $U_+$ is given by $\OO_X[1]$.  Thus if $r \leq a$, then there is no totally semistable wall for $\v'$ between $\GG$ and $\CC$.  Thus for arbitrary $(s',t')\in U_+$, the generic $F\in M_H(\v')$ is $\sigma_{(s',t')}$-stable and  if $F=\Phi_{X\to X}^{I_\Delta}(E)^\vee$, then $E\in M_{(s,t)}(\v)$ is generic, where now $(s,t)$ is arbitrary, giving part (3).

Now we prove part (2).  If $r>a$ so that $\chi(\v')=a-r<0$, there is a unique totally semistable wall for $\v'$  between $\GG$ and $\CC$, corresponding to $\OO_X[1]$.  Thus there is a unique totally semistable wall $W$ in $U_+$ for $\v$.  Let $\sigma_i=\sigma_{(s,t_i)}$ such that $(s,t_1)$ (resp. $(s,t_2)$) is above (resp. below) $W$.  Then for the generic $E\in M_{\sigma_1}(\v)$, we have $\Phi_{X\to X}^{I_\Delta}(E)^\vee\in M_{\sigma_2}(\v')$, and similarly, for the generic $E\in M_{\sigma_2}(\v)$, we have $\Phi_{X\to X}^{I_\Delta}(E)^\vee\in M_{\sigma_1}(\v')\cap M_H(\v')$, giving (b).  Thus by \eqref{eqn:HNfiltration} from the proof of  \cref{Prop:a<=0}, for the generic $E\in M_{\sigma_1}(\v)$, $\Phi_{X\to X}^{I_\Delta}(E)^\vee$ sits in a distinguished triangle
\begin{equation}\label{eqn:negativerankHN}
    (\OO_X[1])^{\oplus (r-a)} \to \Phi_{X \to X}^{I_\Delta}(E)^{\vee} \to R_{\OO_X[1]}(\Phi_{X \to X}^{I_\Delta}(E)^{\vee}), 
\end{equation}
where $R_{\OO_X[1]}(\Phi_{X \to X}^{I_\Delta}(E)^{\vee})\in M_H(r,dH,-a)$, as required.
Now, 
$$M_{(sH,tH)}(r,dH,a)=\{E | E \in M_H(r,dH,a) \}$$
for $s<d/r$
and
$M_{(sH,tH)}(r,dH,a)=\{F^{\vee} | F \in M_H(r,-dH,a) \}$ for $s>d/r$.

For a generic $E[-1] \in M_{\sigma_1}(r,-dH,-a)$,
since $s>0>-d/r$, $F=(E[-1])^{\vee} \in M_H(r,dH,-a)$.
\end{proof}

\subsection{Ranks 0 and 1} In this subsection, we show that all Mukai vectors $\v=(r, dH, a)$ with $d>0$ and $r=0,1$ satisfy weak Brill-Noether. We also show that if in addition $a\geq 2$, then the generic sheaf in $M_H(\v)$ is globally generated unless $n=r=1$, $d=2$ and $2 \leq a \leq 4$. We begin by a useful lemma.

\begin{Lem}\label{lem:WallsFor n>=r}
Let $\v=(r,dH,a)$ be a Mukai vector with $\v^2\geq -2$,  $r,a\geq 0$, and $d>0$.  Assume that $n\geq r$ and $\v_1 = (r_1, d_1H, a_1) \in D_\v$.  Set $m=r_1d -r d_1$ and $k=a_1 d - a d_1$. Then $d_1=1$ and one of the following holds:
\begin{enumerate}
    \item $r_1d=r+1$, $r_1a_1=n+1$ and $r_1\leq \sqrt{\frac{2r-n}{k}}$; or
    \item $n=r$, $\v=(r,(r+2)H,r^2+3r+1)$ and $\v_1=(1,H,r+1)$; or
    \item $n=r=1$, $\v=(1,dH,2d-1)$ ($d\geq 2$) and $\v_1=(1,H,2)$.
\end{enumerate}
\end{Lem}
\begin{proof}
By Lemma \ref{lem:useful consequences}, if $\v_1=(r_1, d_1H, a_1)  \in D_\v$, then $d_1 < \frac{2r}{n}$. Since $n \geq r$ and $d_1$ is a positive integer, we conclude that $d_1=1$. Hence, $r_1a_1 = n+1$. By Lemma \ref{lem:rewrite product} and Remark \ref{rem-numerics} \ref{enum:totally semistable spherical}, we have 
$$0>m(n-1)-2r + kr_1^2.$$ 
If $n\geq 2$, then as $n \geq r$, we must have $m \leq 2$. We conclude that:
\begin{enumerate}
\item $m=r_1d - r =1$ and $r_1 \leq \sqrt{\frac{2r-n}{k}}$; or
\item $m=2$ and hence, $n=r,$ $k=r_1=1$. We conclude that $\v_1 = (1, H, r+1)$ and $\v = (r, (r+2)H, r^2+3r+1)$; or
\item $n=r=1$ and $k=r_1=1$, so that $\v_1= (1, H, 2)$ and $\v= (1, dH, 2d-1)$ and $d \geq 2$ since $m= d-1 \geq 1$.
\end{enumerate} 
\end{proof}

If we apply this lemma in ranks zero and one, we get the following two results:
\begin{Prop}\label{prop:Rank0}
If $\v=(0,dH,a)$ is a Mukai vector with $d>0$, then $\v$ satisfies weak Brill-Noether.  Moreover, if $a\geq 2$, then the generic $E\in M_H(\v)$ is globally generated.
\end{Prop}
\begin{proof}
Let $\v=(0,dH,a)$ be a Mukai vector  with $d>0$ and $a\geq 0$.  \cref{lem:WallsFor n>=r} implies that $D_\v=\varnothing.$ By Theorem \ref{thm:summing up inequalities}, the generic $E\in M_H(\v)$ has $H^1(X,E)= 0$. On the other hand, if $a <0$, then \cref{Prop:a<=0} implies that $\v$ satisfies weak Brill-Noether.

Now let $a \geq 2$. Suppose the generic $E\in M_H(\v)$ is not globally generated.  Since $D_\v=\varnothing$, \cref{thm:summing up inequalities} implies that the generic sheaf in $M_H(a,dH,0)$ is not locally free.  By \cref{Prop:OnlyNonLocallyFreeSheaves}, this is only possible  if $(a,dH,0)=(l,lpH,lp^2n-1)$, as $a\geq 2$ by assumption and $\v^2>0$. Solving this equality gives $l=p=n=1$, so that $a=1$ contrary to our assumption that $a\geq 2$. 
\end{proof}

\begin{Prop}\label{prop:Rank1}
Let $\v=(1,dH,a)$ be a Mukai vector such that $d>0$ and $\v^2\geq -2$.  Then $\v$ satisfies weak Brill-Noether.  If $a\geq 2$, then the generic $E\in M_H(\v)$ is globally generated unless $n=1$ and $\v=(1,2H,a)$ for $a=2,3,4$.
\end{Prop}
\begin{proof}
When $r=1$, it follows from \cref{lem:WallsFor n>=r} that if $D_\v\neq\varnothing$, then $n=1$, $d\geq 2$, $\v=(1,dH,2d-1)$, and $D_\v=\{\v_1\}$, where $\v_1=(1,H,2)$.  Indeed, if $n \geq 2$, then $D_\v=\varnothing$ as $$1\leq r_1 \leq \sqrt{\frac{2r -n}{k}}.$$   

Therefore, except for the case $n=1$ and $\v =(1, dH, 2d-1)$, $D_\v=\varnothing$, so \cref{thm:summing up inequalities} implies that $\v$ satisfies weak Brill-Noether. Moreover, by \cref{Prop:OnlyNonLocallyFreeSheaves}, $M_H(a, dH, 1)$ contains locally free sheaves when $a \geq 2$ unless $n=1$ and $\v=(1,2H,2)$ or $\v=(1,2H,4)$. Hence, the generic sheaf in $M_H(\v)$ is globally generated except in these two cases. 

Returning to the case when $n=1$ and $\v= (1, dH, 2d-1)$, we observe that $D_{\v_0} = \varnothing$ where $\v_0=(1,dH,d^2+1)$. \cref{thm:summing up inequalities} and \cref{prop:minimal-a} then imply that $\v$ satisfies weak Brill-Noether in this case as well. 
For the question of global generation,  we consider the Harder-Narasimhan filtration for the generic $E\in M_H(\v)$ along the wall $\WW$ induced by $\v_1=(1,H,2)$, which is the unique effective spherical class in $\HHH_\WW$.  \cref{lem:max-HNF2} then implies that the generic $E\in M_H(\v)$ sits in a short exact sequence \begin{equation}\label{eqn:rank one counterexample}
  0\to \OO_X(H)\to E\to F\to 0,  
\end{equation} where $\OO_X(H)$ is the unique stable bundle of Mukai vector $\v_1$ and $F$ is a $\sigma$-stable object such that $\v(F)=(0,(d-1)H,2d-3)$.  By \cref{lem:WallsFor n>=r}, $D_{\v_1}=\varnothing$ and by \cref{prop:Rank0}, $D_{\v(F)}=\varnothing$, so applying $\Phi_{X\to X}^{I_\Delta}(\blank)^\vee$ to \eqref{eqn:rank one counterexample}, we get $$0\to\Phi_{X\to X}^{I_\Delta}(F)^\vee\to\Phi_{X\to X}^{I_\Delta}(E)^\vee\to\Phi_{X\to X}^{I_\Delta}(\OO_X(dH))^\vee\to 0.$$ By \cref{prop:isom,prop:birational}, $\Phi_{X\to X}^{I_\Delta}(\OO_X(dH))^\vee\in M_H(2,H,1)$ and $\Phi_{X\to X}^{I_\Delta}(F)^\vee\in M_H(2d-3,(d-1)H,0)$ for generic $F\in M_H(0,(d-1)H,2d-3)$.  
If $d\geq 3$, then the generic element of $M_H(2d-3,(d-1)H,0)$ is locally free by \cref{Prop:OnlyNonLocallyFreeSheaves}, so for generic $E\in M_H(\v)$, we  have $\Phi_{X\to X}^{I_\Delta}(E)^\vee$ is locally free  since $\Phi_{X\to X}^{I_\Delta}(\OO_X(dH))$ is locally free. Thus by \cref{Lem:FM-vanishing}, for $d\geq 3$ the generic $E\in M_H(\v)$ satisfies $H^i(X,E)=0$ for $i>0$ and is globally generated. 

It remains to consider $n=1$ and $\v=(1,2H,a)$ for $2\leq a\leq 4$.\footnote{The additional case when $a=3$ is the $d=2$ case of $\v=(1,dH,2d-1)$.}  In this case, the generic $E\in M_H(\v)$ is not globally generated by \cref{rem:gg n is 1}.
\end{proof}

\subsection{Ranks 2 and 3} We next classify moduli spaces of rank 2 and 3 sheaves where weak Brill-Noether fails. We begin by a numerical observation.

\begin{Lem}\label{prop:condensed inequality}
Let $\v=(r,dH,a)$ be a Mukai vector such that $r\geq 2$,$a\geq 0$,$d>0$, and $\v^2\geq -2$.  If the generic $E\in M_H(\v)$ has $H^1(X,E)\neq 0$, then $D_{\v_0}^{BN}\neq\varnothing$, where $\v_0=(r,dH,\left\lfloor\frac{nd^2+1}{r}\right\rfloor)$.  Moreover, for any $\v_1=(r_1,d_1H,a_1)\in D_{\v_0}^{BN}$, if we set $m=r_1d -r d_1$ and $k=a_1 d - a_0 d_1$, then
\begin{equation}\label{eqn:deriving finiteness, 1}
1\leq k= a_1d-a_0d_1= \frac{1}{rr_1}\left(-nmdd_1+rd-r_1d_1\right)+d_1\left\{\frac{nd^2+1}{r}\right\}
\end{equation}
and
\begin{equation}\label{eqn:finiteness inequality}
1\leq ndd_1m\leq rd-r_1d_1+rr_1d_1\left\{\frac{nd^2+1}{r}\right\}-rr_1,
\end{equation}
where $\{ \blank \}$ denotes the fractional part.  In particular, $1\leq d_1nm<r+r_1(r-2)$.
\end{Lem}
\begin{proof}
The first claim follows from \cref{prop:minimal-a}, so to prove the remaining claims it suffices to assume that $\v=\v_0$.  If $\v_1=(r_1,d_1H,a_1)\in D_{\v}^{BN}$, the condition \ref{rem-numerics}\ref{enum:positive a_1} becomes 
\begin{align*}
\begin{split}
1&\leq k=a_1d-a_0d_1=d\left(\frac{nd_1^2+1}{r_1}\right)-d_1\left(\frac{nd^2+1}{r}-\left\{\frac{nd^2+1}{r}\right\}\right)\\
&=\frac{1}{rr_1}\left(rd(nd_1^2+1)-r_1d_1(nd^2+1)\right)+d_1\left\{\frac{nd^2+1}{r}\right\}\\
&=\frac{1}{rr_1}\left(ndd_1(rd_1-r_1d)+rd-r_1d_1\right)+d_1\left\{\frac{nd^2+1}{r}\right\}\\
&=\frac{1}{rr_1}\left(-nmdd_1+rd-r_1d_1\right)+d_1\left\{\frac{nd^2+1}{r}\right\}.
\end{split}
\end{align*}
 We thus obtain \eqref{eqn:deriving finiteness, 1}. Rearranging this equation and using condition \ref{rem-numerics}\ref{enum:positive r_1} gives 
$$
1\leq ndd_1m\leq rd-r_1d_1+rr_1d_1\left\{\frac{nd^2+1}{r}\right\}-rr_1.
$$
Using the facts that $d_1\leq d$ and $\left\{\frac{nd^2+1}{r}\right\}\leq\frac{r-1}{r}$, we get
$$1\leq ndd_1m\leq rd-r_1d_1+r_1d_1(r-1)-rr_1<rd+r_1d_1(r-2)\leq rd+r_1d(r-2),$$
and dividing by $d$ gives the final inequality in the lemma.
\end{proof}

\begin{Thm}\label{prop:rankstwoandthree}
Let $\v=(r,dH,a)$ be a Mukai vector such that $2 \leq r \leq 3$, $a\geq 0$, $d>0$, and $\v^2\geq -2$.  If the generic $E\in M_H(\v)$ satisfies $H^1(X,E)\ne 0$, then 
\begin{enumerate}
\item $n=1$ and $\v=(2,3H,5)$, or
\item $n=1$ and $\v=(3,4H,5)$, or
\item $n=2$ and $\v=(3,4H,11)$.
\end{enumerate}
In these cases $h^1(E)=1$.
\end{Thm}
\begin{proof}
By \cref{prop:minimal-a}, we may first consider the cases when $a= \left\lfloor \frac{nd^2+1}{r} \right\rfloor.$ 

First, suppose $r=2$. If $\v_1 \in D_{\v_0}^{BN}$, then by \cref{prop:condensed inequality}
$$1 \leq d_1mn <2,$$ 
so it follows from $0< r_1 <r$ that 
$n=r_1=d_1 = m=1$. 
We conclude that $n=1$, $\v_1=(1,H,2)$ and $\v_0=(2, 3H, 5)$. For $E \in M_H(2,3H,5)$, $h^1(E)=1$ by \cref{ex:Fibonacci}. 

Next suppose $E \in M_H(2, 3H, 4)$ is a generic sheaf. Then $\u_0=-(1,-H,2)$ and $\u_1=(1,H,2)$ are the Mukai vectors of the two unique $\sigma$-stable objects $T_0=\OO_X(-H)[1]$ and $T_1=\OO_X(H)$ with Mukai vector in $\HHH$, respectively.  As $R_{T_1}(\v)=(0,H,0)$ is a minimal vector in its orbit, \cref{lem:max-HNF2} implies that $E$ fits in an exact sequence
$$0\to\OO_X(H)^{\oplus 2}\to E\to R_{T_1}(E)\to 0,$$ where $R_{T_1}(E)\in M_\sigma(0,H,0)$ is $\sigma$-stable.  By the proof of \cref{prop:Rank0} $D_{(0,H,0)}=\varnothing$, so the generic $L\in M_H(0,H,0)$ is still $\sigma$-stable.  Thus $R_{T_1}(E)\in M_H(0,H,0)$.   Then $H^1(X,E)=0$ by taking cohomology of the short exact sequence. \cref{prop:minimal-a} and \cref{thm:summing up inequalities} now imply the proposition when $r=2$.

Next, suppose $r=3$. Let $\v_1 \in D_{\v_0}^{BN}$. By \cref{thm:bound n r} and \cref{lem:useful consequences}, we may assume $1 \leq n \leq 2$ and $1 \leq r_1 \leq 2$. First, suppose $n=2$. Then there are no spherical classes with $r_1=2$, so we must have $r_1=1.$ By \cref{prop:condensed inequality}, it follows that
$$1\leq 2d_1m<4,$$
so $d_1=m=1$.  
We conclude that when $n=2$,  $D_{\v_0}^{BN} = \varnothing$ unless $\v_0=(3, 4H, 11)$, in which case $D_{\v_0}^{BN}=\{(1, H, 3) \}$. 

We claim that the unique $E_0\in M_H(3,4H,11)$ has $h^1(E_0)=1$. Indeed, the sheaf $E_0$ fits in a short exact sequence $$0\to\OO_X(H)^{\oplus 4}\to E_0\to R_{\OO_X(H)}(E_0)\to 0,$$ by \cref{lem:max-HNF1}, where $R_{\OO_X(H)}(E_0)\in M_\sigma(-1,0,-1)$.  As $\OO_X[1]$ is $\sigma$-stable of the same Mukai vector, we must have $R_{\OO_X(H)}(E_0)\cong\OO_X[1]$.  Taking cohomology gives $h^1(E_0)=1$.

On the other hand, we claim that the generic $E\in M_H(3,4H,10)$ has $H^1(X,E)=0$.  It is easy to see that $D_{(3,4H,10)}=\{(1,H,3)\}$, so  \cref{lem:max-HNF1} implies that the generic $E\in M_H(3,4H,10)$ fits in a short exact sequence $$0\to \OO_X(H)^{\oplus 3}\to E\to R_{\OO_X(H)}(E)\to 0,$$ where $R_{\OO_X(H)}(E)\in M_\sigma(0,H,1)$. By \cref{prop:Rank0}, the generic $L\in M_H(0,H,1)$ is $\sigma$-stable because $D_{(0,H,1)}=\varnothing$, so $(0,H,1)=R_{\OO_X(H)}(3,4H,10)$ is a minimal vector in its orbit,  and $R_{\OO_X(H)}(E)\in M_\sigma(0,H,1)\cap M_H(0,H,1)$.  Then taking cohomology gives $H^1(X,E)=0$, as required. By \cref{thm:summing up inequalities} and \cref{prop:minimal-a}, we conclude that when $r=3$ and $n=2$, $\v$ satisfies weak Brill-Noether unless $\v = (3, 4H, 11)$.

Now we may suppose that $n=1$. If $r_1=2$, then \cref{prop:condensed inequality} implies that $1 \leq d_1 (2d-3d_1) \leq 4$.  Since $r_1$ divides $d_1^2 +1$, $d_1$ and $m=2d-3d_1$ must both be odd. We conclude that the only possibilities are: 
\begin{enumerate}
    \item $d_1=3$, $m=1$, in which case $\v_1 = (2, 3H, 5)$ and $\v_0 = (3, 5H, 8)$. We will discuss this case below.
    \item $d_1=1$, $m=3$, in which case $\v_1 = (2, H, 1)$ and $\v_0 = (3, 3H, 3)$. In this case, $a_1 d - a_0 d_1 =0$, which is a contradiction.
    \item $d_1=m=1$, in which case $\v_1= (2, H, 1)$ and $\v_0 = (3, 2 H, 1)$. In this case, $a_1 d - a d_1 =1 =m$. Since $D_{\v_0}= \{ \v_1 \}$, by \cref{Lem:O_X[1] largest wall}, the cohomology of the generic sheaf in $M_H(\v_0)$ vanishes.
    \end{enumerate}

If $r_1 =1$, then \cref{prop:condensed inequality} implies that $1 \leq d_1 (d-3d_1) \leq 3$. The possibilities are:
\begin{enumerate}
    \item $d_1=3$, $m=1$, in which case $\v_1 = (1, 3H, 10)$ and $\v_0 = (3, 10H, 33)$. In this case, $a_1 d - a d_1 =1 =m$, so by \cref{Lem:O_X[1] largest wall}, the cohomology of the generic sheaf in $M_H(\v_0)$ vanishes.
    \item $d_1=1$, $m=3$, in which case $\v_1 = (1, H, 2)$ and $\v_0= (3, 6H, 12)$. In this case, $a_1 d - a_0 d_1 = 0$, which is a contradiction. 
   \item $d_1=2$, $m=1$, in which case $\v_1= (1, 2H, 5)$ and $\v_0 = (3, 7H, 16)$. We will discuss this case below.
   \item $d_1=1$, $m=2$, in which case $\v_1 = (1, H, 2)$ and $\v_0 = (3, 5H, 8)$. Since $$D_{\v_0}= \{(1, H, 2),(2, 3H, 5)\}$$  and for both of these $a_1d - a_0 d_1 = m$, by \cref{Lem:O_X[1] largest wall}, the cohomology of the generic sheaf in $M_H(\v_0)$ vanishes.
    \item $d_1=m=1$, in which case $\v_1 = (1, H, 2)$ and $\v_0 = (3, 4H, 5)$.  We will discuss this case below.
\end{enumerate}

We need to discuss the remaining two cases. First, let $\v_0=(3,7H,16)$ and $\v_1=(1,2H,5)=\v(\OO_X(2H))$.  By \cref{prop:Rank0}, the generic $L\in M_H(0,H,1)$ is $\sigma$-stable because $D_{(0,H,1)}=\varnothing$.  Thus $(0,H,1)=R_{\OO_X(2H)}(3,7H,16)$ is a minimal vector in its orbit, so by \cref{lem:max-HNF1} the generic $E\in M_H(3,7H,16)$ sits in a short exact sequence $$0\to \OO_X(2H)^{\oplus 3}\to E\to R_{\OO_X(2H)}(E)\to 0,$$ where $R_{\OO_X(2H)}(E)=L\in M_H(0,H,1)\cap M_\sigma(0,H,1)$.  Taking cohomology we see that $H^1(X,E)=0$.

Finally, let $\v_0 = (3, 4H, 5)$. The generic $L\in M_H(0,H,-1)$ is $\sigma$-stable because $D_{(0,H,-1)}=\varnothing$.  Thus $(0,H,-1)=R_{\OO_X(H)}(3,4H,5)$ is a minimal vector in its orbit, so by \cref{lem:max-HNF1} the generic $E\in M_H(3,4H,5)$ fits in a short exact sequence $$0\to \OO_X(H)^{\oplus 3}\to E\to R_{\OO_X(H)}(E)\to 0,$$ where $R_{\OO_X(H)}(E)=L\in M_H(0,H,-1)\cap M_\sigma(0,H,-1)$.  The Mukai vector $(0,H,-1)$ satisfies weak Brill-Noether by \cref{prop:Rank0}, so we must have $H^1(X,L)=\C$ since $\chi(L)=-1$.  We conclude that $h^1(X,E)=1$. 

It is easy to see that  $D_{(3,4H,4)}=\{(1,H,2)\}$.  It follows from \cref{lem:max-HNF1} that the generic $E\in M_H(3,4H,4)$ sits in a short exact sequence $$0\to \OO_X(H)^{\oplus 2}\to E\to R_{\OO_X(H)}(E)\to 0,$$ where $R_{\OO_X(H)}(E)\in M_{\sigma}^s(1,2H,0)$.  If  $R_{\OO_X(H)}(E)\in M_H(1,2H,0)$, then $h^1(E)=0$ by \cref{prop:Rank1} and the long exact sequence of cohomology.  By \cref{lem:WallsFor n>=r}, $D_{(1,2H,0)} =\varnothing$ as the only possibility is $(1,H,2)$ which doesn't pair negatively with $(1,2H,0)$.  Thus the generic member of $M_H(1,2H,0)$ is $\sigma$-stable, as required. We conclude that when $n=1$ and $r=3$, $M_H(\v)$ satisfies weak Brill-Noether except when $\v=(3, 4H, 5)$.
\end{proof}

\begin{Rem}
One can carry out this analysis for increasing rank, but the number of possibilities grows rapidly. In \cref{sec-rank20}, we will classify all of the Mukai vectors that violate weak Brill-Noether up to rank 20 with the aid of a computer search. 
\end{Rem}

\subsection{Classification of moduli spaces where weak Brill-Noether fails when $r$ is small relative to $n$}
By \cref{thm:bound n r}, all Mukai vectors with $r\leq n$ satisfy weak Brill-Noether. In this subsection, we study the cases when $n<r \leq 3n$.  We split our classification into two theorems studying the possibilities when $n<r\leq 2n$ and $2n<r\leq 3n$, respectively.  In both theorems, it is useful to express certain important quantities, such as $\v^2$, in terms of $k,m,r_1,d_1,n$, and $r$, where $\v_1\in D_\v$.   
\begin{Lem}\label{obs-v_0^2}
Let $\v=(r,dH,a)$ be a Mukai vector such that $r,a\geq 0$ and $d>0$.  For $\v_1=(r_1,d_1H,a_1)\in D_\v$, set $k=a_1d-ad_1$ and $m=r_1d-rd_1$.  Then $r_1$ divides $r-d_1mn$, so using the division algorithm we may write $r-d_1mn=r_1^2q+r_1s$ where $q\geq 0$ and $0\leq s<r_1$.  Thus

$$\v^2 = \frac{2kr}{d_1} + \frac{\left(2nm - \frac{2r}{d_1}\right)\left(rd_1 + m\right)}{r_1^2}=\frac{2}{d_1}\left(kr-(qr_1+s)\left(\frac{rd_1+m}{r_1}\right)\right).$$

Moreover, 
\begin{equation*}
\begin{split}R_{\v_1}(\v)=\left(-(qr_1+s)r_1+\frac{1}{d_1}\left(-m+kr_1^2\right),\left(-d_1(qr_1+s)+kr_1\right)H,-(qr_1+s) a_1+kd_1n\right).
\end{split}
\end{equation*}
\end{Lem} 
\begin{proof}
Since $r_1 d-d_1 r=m$ and $r_1 a_1=d_1^2 n+1$, $r_1(ma_1-d)=d_1(d_1 mn-r)$.
As $\gcd(r_1,d_1)=1$, $r_1$ divides $r-d_1mn$, so using the division algorithm we may write $r-d_1mn=r_1^2q+r_1s$ where $q\geq 0$ and $0\leq s<r_1$.
Solving for $d$ in the definition of $m$, for $a$ in the definition of $k$, and for $a_1$ in the equation $\v_1^2=-2$, we get

\begin{equation}
\begin{split}
\frac{\v^2}{2}=nd^2-ra=& d \left(dn-\frac{ra_1}{d_1} \right)+\frac{rk}{d_1}\\
=& d \left(dn-\frac{r}{d_1}\frac{d_1^2 n+1}{r_1} \right)+\frac{rk}{d_1}\\
=& \frac{d}{d_1 r_1}(d_1 mn-r)+\frac{rk}{d_1}\\
=& \frac{1}{d_1}(rk-(r_1 q+s)d)\\
=&\frac{1}{d_1}\left(rk-(qr_1+s)\left(\frac{rd_1+m}{r_1}\right)\right).
\end{split}
\end{equation}

For the claim about the spherical reflection, by \cref{lem:rewrite product} we have 
$$\langle\v,\v_1\rangle=\frac{m(nd_1^2-1)-2rd_1+kr_1^2}{r_1 d_1}.$$

Using this we compute the components of $R_{\v_1}(\v)=\v+\langle\v,\v_1\rangle\v_1$.  For the rank, we get

\begin{equation}
\begin{split}
 r+\frac{r_1}{r_1 d_1}(m(nd_1^2-1)-2rd_1+kr_1^2)
=& \frac{1}{d_1}(d_1(mnd_1-r)-m+kr_1^2)\\
=& -(qr_1+s)r_1+\frac{1}{d_1}(-m+kr_1^2);
\end{split}
\end{equation}
for the first Chern class we get

\begin{equation}
\begin{split}
 d+\frac{d_1}{r_1 d_1}(m(nd_1^2-1)-2rd_1+kr_1^2)
=& \frac{1}{r_1}(r_1 d+m(nd_1^2-1)-2rd_1+kr_1^2)\\
=& \frac{1}{r_1}(d_1(mnd_1-r)+kr_1^2)\\
=& \frac{1}{r_1}(-d_1(qr_1+s)r_1+kr_1^2)\\
=& -d_1(qr_1+s)+kr_1;
\end{split}
\end{equation}
and for the final component we get
\begin{equation}
\begin{split}
 & a+\frac{a_1}{r_1 d_1}(m(nd_1^2-1)-2rd_1+kr_1^2) \\
=& \frac{a_1 d-k}{d_1}+\frac{a_1}{r_1 d_1}(m(nd_1^2-1)-2rd_1+kr_1^2)\\
=& \frac{1}{r_1 d_1}(a_1(r_1 d+m(nd_1^2-1)-2rd_1+kr_1^2)-kr_1)\\
=& \frac{1}{r_1 d_1}(a_1((mnd_1-r)d_1+kr_1^2)-kr_1)\\
=& \frac{1}{r_1 d_1}(a_1(-(qr_1+s)r_1 d_1+kr_1^2)-kr_1)\\
=&-a_1(qr_1+s)+\frac{1}{r_1d_1}(kr_1(nd_1^2+1)-kr_1)\\
=&-a_1(qr_1+s)+nkd_1.
\end{split}
\end{equation}

\end{proof}

\begin{Thm}\label{Thm:n<r<=2n}
Let $X$ be a K3 surface such that $\Pic(X)= \mathbb{Z} H$ with $H^2 =2n$.
Let $\v=(r,dH,a)$ be a Mukai vector such that $n<r\leq 2n$, $d>0$.  Then the generic $E\in M_H(\v)$ satisfies $H^1(X,E)\ne 0$ if and only if $$\v=\left(n+r_1^2,\left(\left(\frac{n+1}{r_1}\right)+r_1\right)H,\left(\frac{n+1}{r_1}\right)^2+n\right)$$ for some $r_1\mid n+1$ and $1 \leq r_1\leq\sqrt{n}$, in which case $h^1(X,E)=1$.
\end{Thm}
\begin{Rem}
In the exceptional cases of \cref{Thm:n<r<=2n}, we have $\v^2=-2$.  In particular, if $\v$ as in \cref{Thm:n<r<=2n} satisfies $\v^2\geq 0$, then $H^1(X,E)=0$ for the generic $E\in M_H(\v)$.
\end{Rem}
\begin{proof}
By \cref{prop:minimal-a}, we may first suppose that $\v=\v_0=(r,dH,\left\lfloor\frac{nd^2+1}{r}\right\rfloor)$ and suppose that $\v_1\in D_{\v}^{BN}$ induces the largest totally semistable wall.  By \cref{lem:useful consequences}, $d_1 < \frac{2r}{mn} \leq \frac{4}{m}$. Hence, $1 \leq m d_1 \leq 3$. 

If $d_1=2$ or $3$, then $m=1$. In these two cases, we will now show that $k=a_1d -ad_1=1$ and $\v^2 > 0$. It will then follow from  \cref{Lem:O_X[1] largest wall} that $\v$ satisfies weak Brill-Noether in these cases.

When $d_1=2$, we have $r_1d =2r+1$ and $r_1a_1 = 4n+1$. Substituting  $r\leq 2n$ and $d_1=2$ into Equation (\ref{eqn:deriving finiteness, 1}) of \cref{prop:condensed inequality}, we obtain
$$1 \leq a_1 d - 2a \leq -\frac{2}{r} + 2 \left\{ \frac{nd^2+1}{r} \right\} \leq - \frac{2}{r} + 2 \frac{r-1}{r} < 2.$$ Hence, $a_1d-a d_1=1$. By \cref{obs-v_0^2},
$$\v^2 =  r + \frac{(2n-r)(2r+1)}{r_1^2} \geq r >0,$$ as claimed. 

Similarly, when $d_1=3$, we have $r_1d=3r+1$ and $r_1 a_1= 9n+1$. By \cref{lem:rewrite product}, 
$$ 0 > 9n -1 -6r + kr_1^2 \geq 9n-1-12n + kr_1^2.$$
Hence, $3n \geq kr_1^2$. Substituting  $r\leq 2n$ and $d_1=3$ into  Equation (\ref{eqn:deriving finiteness, 1})  of \cref{prop:condensed inequality}, we obtain
$$1 \leq a_1 d- 3a \leq \frac{1}{rr_1} ( -nd -3r_1) + 3 \left\{ \frac{nd^2+1}{r} \right\} \leq - \frac{3n}{r_1^2} + 3 \frac{r-1}{r} \leq -k + 3 \frac{r-1}{r}. $$
Hence, $k=a_1d - a d_1 =1$. By \cref{obs-v_0^2}, $$\v^2=\frac{2}{3}\left(r+\frac{(3n-r)(3r+1)}{r_1^2}\right)>0,$$ as claimed.

We are thus reduced to considering the case $d_1=1$.  Substituting $d_1=1$ and $r \leq 2n$ into \eqref{eqn:finiteness inequality}, we get $$1\leq ndm\leq rd-r_1+rr_1\left\{\frac{nd^2+1}{r}\right\}-rr_1<2nd-r_1.$$ Hence, $m=1$,  $r_1 d = r+1$ and $r_1 a_1 = n+1$. By \cref{lem:rewrite product}, 
$$ 0 > (n -1) -2r + kr_1^2 \geq -3n-1 + kr_1^2.$$ Thus $kr_1^2 \leq 3n$ and $r_1 \leq \sqrt{3n}$. In particular,  $r_1 \leq n$. 

By \cref{obs-v_0^2}, using $d_1=1$, $m=1$, we can write
$$r-n = r_1^2 q + s r_1,$$ where $0 \leq s < r_1$.
Let $f = \{ \frac{nd^2+1}{r}\}$. Then $$a_1 d - a = \frac{(n+1)(r+1)}{r_1^2} - \frac{n (r+1)^2 + r_1^2}{rr_1^2} + f = q + \frac{s}{r_1} + \frac{q-1}{r} + \frac{s}{rr_1} + f.$$
We claim that either $a_1 d -a =1$ and $q=1$, $s=f=0$ or $a_1 d -a = q+1$. Since $r>n$, if $s=0$, then $q \geq 1$ and $$a_1d-a = q + \frac{q-1}{r} + f.$$ Since $a_1d -a$ is an integer and $q-1<r$, we conclude that $q \leq  a_1 d - a \leq q+1$. Furthermore, $a_1d -a =q$ only if $q=1$ and $f=0$. We may now assume that $s\geq 1$ and $q \geq 0$. Since $r_1^2 q < r$, we have that $\frac{q}{r} < \frac{1}{r_1^2}$. Hence, 
$$\frac{1}{rr_1} \leq \frac{s}{r_1} + \frac{q-1}{r} + \frac{s}{rr_1} + f \leq \frac{s}{r_1} + \frac{1}{r_1^2} + f <2.$$ Since $a_1d -a$ is an integer, we conclude that $a_1d-a=q+1$.

First, suppose that $(q,s) \not= (1,0)$.  Then $k=q+1$, and substituting $m=1=d_1$ into \cref{obs-v_0^2} gives $$\v'=(r',d'H,a')=R_{\v_1}(\v)=(-1+r_1(r_1-s),(r_1-s)H,n-q-a_1s),$$ which satisfies $0\leq r'<r$, and $d'>0$.  Since $r_1d-r=m=1$,  \cite[Theorem 2.5]{Yos99b}  implies that the generic $E\in M_H(\v)$ sits in a short exact sequence $$0\to E_1^{\oplus -\langle\v,\v_1\rangle}\to E\to F\to 0,$$ where $E_1\in M_H(\v_1)$ and $F\in M_H(\v')$ is generic.  Since $$\v^2=(\v')^2=\frac{2}{r_1}(r_1(r_1-s)(r_1q+s)+(r_1-s)n-(r_1q+s))>0,$$ it follows by induction that the generic $F\in M_H(\v')$ satisfies $H^1(X,F)=0$.  Since $r_1 \leq n$, \cref{thm:bound n r} implies that $H^1(X,E_1)=0$.  The long exact sequence of cohomology now implies that $H^1(X,E)=0$ for the generic $E\in M_H(\v)$.

Finally, consider the case $(q,s)=(1,0)$. We then have  $r=n+r_1^2$, and $k=1$, so that setting $m=1=d_1$ in \cref{obs-v_0^2} we get
$$\v'=(r',d'H,a')=R_{\v_1}(\v)=(-1,0,-1)=\v(\OO_X[1]).$$  
Thus $\v^2=(\v')^2=-2$, so the unique $E\in M_H(\v)$ fits in the short exact sequence $$0\to \OO_X\to E_1^{\oplus d}\to E\to 0,$$ from which it is clear that $H^1(X,E)=\C$ since $H^1(X,E_1)=0$ by \cref{thm:bound n r}.

To finish the proof, we need to show that the generic $E\in M_H(\v)$ has $H^1(X,E)=0$, where $\v=\v_0-(0,0,1)$ and $\v_0=(n+r_1^2,((\frac{n+1}{r_1})+r_1)H,(\frac{n+1}{r_1})^2+n)$ as in the previous case.  
From $r_1d-r=1$ we see that the generic $E\in M_H(\v)$ sits in an exact sequence 
$$0\to E_1^{\oplus -\langle\v,\v_1\rangle}\to E\to F\to 0,$$
 where $F\in M_H(\v')$ is generic with 
$$\v'=R_{\v_1}(\v)=\v-a_1(r_1,H,a_1)=(r_1^2-1,r_1H,n-1).$$  
We have $H^1(X,F)=0$ by \cref{thm:bound n r}.  Thus $H^1(X,E)=0$, as required.  
\end{proof}

Our next theorem classifies the failure of weak Brill-Noether when $2n < r \leq 3n$.

\begin{Thm}\label{Thm:2n<r<=3n}
Let $X$ be a K3 surface such that  $\Pic(X) = \mathbb{Z}H$ with $H^2=2n$.
Let $\v=(r,dH,a)$ be a Mukai vector such that $2n<r\leq3n$, $d>0$, and $a>0$.  The  generic $E\in M_H(\v)$ satisfies $H^1(X,E)\ne0$ if and only if $\v$ belongs to one of the following three cases:
\begin{enumerate}
    \item\label{enum:2n<r<=3n torsion quotient} $\v=(r,(r+1)H,nr+2n)$ with $2n<r\leq 3n$;
    \item\label{enum:2n<r<=3n O_X[1] quotient} $\v=(n+r_1^2,((\frac{n+1}{r_1})+r_1)H,(\frac{n+1}{r_1})^2+n)$, where $r_1\mid n+1$ and $\sqrt{n} < r_1 \leq\sqrt{2n}$;
    \item\label{enum:2n<r<=3n weird quotient}
    $\v=(3n,(3n+2)H,3n^2+4n+1)$ with $n>1$.
\end{enumerate}
\end{Thm}

\begin{proof}
By \cref{prop:minimal-a}, it suffices to first assume $\v=\v_0=(r,dH,\left\lfloor\frac{nd^2+1}{r}\right\rfloor)$.  Then suppose that $\v_1=(r_1,d_1H,a_1)\in D_{\v}^{BN}$ induces the largest totally semistable wall.  Set $m=r_1d - rd_1$. By \cref{lem:useful consequences}, $d_1 < \frac{2r}{mn} \leq \frac{6}{m}$. Hence, $1 \leq m d_1 \leq 5$. In particular, if $3 \leq d_1 \leq 5$, then $m=1$ and if $d_1=2$, then $1\leq m \leq 2$.

\subsection*{The case $d_1=5$} If $d_1 =5$, then $m= r_1 d - 5r =1$. Substituting $m=1$ and $r \leq 3n$ into \cref{prop:condensed inequality} \eqref{eqn:deriving finiteness, 1}, we obtain 
$$1 \leq a_1 d - 5a = \frac{1}{rr_1}(-5nd+rd-5r_1)+5\left\{\frac{nd^2+1}{r}\right\}\leq  \frac{1}{rr_1} (-2nd-5r_1) + 5 < 5.$$ Since $d\equiv a_1\pmod 5$, $a_1d$ must be a square modulo $5$. We conclude that $a_1d-5a=1$ or $4$. 

If $a_1d-5a=1$, then by \cref{obs-v_0^2} 
$$\v^2=\frac{2}{5}\left(r + \frac{(5n-r)(5r+1)}{r_1^2}\right)>0.$$
   \cref{Lem:O_X[1] largest wall} eliminates this case.

Now suppose that $k=a_1d-5a=4$.  Then by \cref{obs-v_0^2} 
$$\v'=(r',d'H,a')=R_{\v_1}(\v)=\left(5n-r+\frac{4r_1^2-1}{5},\left(\frac{5(5n-r)+4r_1^2}{r_1}\right)H,20n+\frac{(25n+1)(5n-r)}{r_1^2}\right),$$ where $0<r'<r$ and $d',a'>0$.  Since $r_1d-5r=1$, by \cite[Theorem 2.5]{Yos99b}, the generic $E\in M_H(\v)$ sits in a short exact sequence $$0\to E_1^{\oplus -\langle\v,\v_1\rangle}\to E\to F\to 0,$$ where $E_1\in M_H(\v_1)$ and $F\in M_H(\v')$ is generic.  Note that $\v^2=(\v')^2=\frac{2}{5} \left(\frac{(5 r+1) (5 n-r)}{r_1^2}+4 r\right)>4n$, so $\v'$ cannot be one of the cases \ref{enum:2n<r<=3n torsion quotient}-\ref{enum:2n<r<=3n weird quotient} and thus by induction on the rank we may assume that $H^1(X,F)=0$.  Substituting $m=1$, $k=4$, $d_1=5$ and $r\leq 3n$ into \cref{lem:rewrite product}, we see that $4r_1^2 \leq 5n$, hence  $r_1\leq n$.  By  \cref{thm:bound n r}, $H^1(X,E_1)=0$, and thus $H^1(X,E)=0$ as required.

\subsection*{The case $d_1=3,4$} If $d_1 = 3$ or $4$, then $m=r_1d - d_1r =1$. Substituting $m=1$ and $r \leq 3n$ into \cref{prop:condensed inequality} \eqref{eqn:deriving finiteness, 1}, we obtain 
 $$1\leq a_1d-d_1 a=\frac{1}{rr_1}(-ndd_1+rd-d_1r_1)+d_1\left\{\frac{nd^2+1}{r}\right\} <-\frac{d_1}{r} + d_1 <d_1.$$  Since $r_1a_1 \equiv r_1 d \equiv 1 \pmod{d_1}$, we must have $a_1d\equiv 1\pmod{d_1}$. Hence,  $a_1d-d_1a=1$.  It is easy to check that  $\v^2 >0$ in both cases, therefore \cref{Lem:O_X[1] largest wall} guarantees that $\v$ satisfies weak Brill-Noether in these cases.

 \subsection*{The case $d_1=2$} If $d_1 =2$, then $m=1$ or $2$. If $m=2$, then substituting $r \leq 3n$ into \cref{prop:condensed inequality} \eqref{eqn:deriving finiteness, 1}, we obtain
 $$ 1 \leq a_1d - 2a = \frac{1}{rr_1} (-4nd+rd - 2r_1) + 2 \left\{\frac{nd^2+1}{r}\right\} < 2.$$
  We conclude that $k=a_1 d - a d_1 =1 < m$, contrary to our assumption that $\v_1\in D_\v^{BN}$ so that $k\geq m$. We must therefore have $m=r_1 d -2r=1$ and $r_1a_1 = 4n+1$. Hence, $r_1, a_1$, $d_1$ and $a_1 d - 2a$ are all odd integers.  By \cref{obs-v_0^2}, we may write 
 $$r-2n = q r_1^2 + s r_1,$$ 
 where $0 \leq s < r_1$. Set 
 $$A:=a_1d-2\left(\frac{nd^2+1}{r}\right)=2q+\frac{r_1q+s}{r_1r}+2\left(\frac{s}{r_1}-\frac{1}{r}\right),$$ so that $A \leq a_1d - 2a < A +2$. Observe that $A > 2q$ unless $s=0$ and $q=1$ or $2$, in which case $A=2q - \frac{1}{r}$ or $2q$, respectively. Since $a_1d - 2a \geq A$ is an odd integer, we conclude that it is at least $2q+1$. On the other hand, using $s\leq r_1 -1$ and $r_1q+s<\frac{r}{2r_1}$, we see that $$A<2q+\frac{1}{r_1r}\frac{r}{2r_1}+2\left(\frac{r_1-1}{r_1}-\frac{1}{r}\right)= 2q+2+\frac{r-4rr_1-4r_1^2}{2r_1^2r}<2q+2.$$   Therefore, $k=a_1d-2a=2q+1$  or $2q+3$.

 First, let $k= 2q+1$.  If $s\geq \frac{r_1}{2}$, then since $r_1$ is odd we have $s\geq\frac{r_1+1}{2}$ and $$2q+1\geq A=2q+\frac{r_1q+s}{rr_1}+2\left(\frac{s}{r_1}-\frac{1}{r}\right)\geq 2q+\frac{q}{r}+\frac{r_1+1}{2rr_1}+\frac{r_1+1}{r_1}-\frac{2}{r}\geq2q+1+\frac{2r-3r_1+1}{2rr_1}>2q+1,$$ a contradiction. The final inequality follows from $r_1s\leq r_1^2q+r_1s\leq n<\frac{r}{2}$ and $s\geq\frac{r_1}{2}$ which give $r>r_1^2$ and thus $2r-3r_1+1>2r_1^2-3r_1+1=(2r_1-1)(r_1-1)\geq 0$.  Hence, $s\leq \frac{r_1-1}{2}$. If $q=0$, then $a_1d-2a=1=r_1d-2r$, and it follows from $s\leq\frac{r_1-1}{2}$ and \cref{obs-v_0^2} that $$\v^2=r-ds\geq r-\frac{r_1d-d}{2}=r-\frac{2r+1-d}{2}=\frac{d-1}{2}\geq 0.$$  \cref{Lem:O_X[1] largest wall} eliminates this case.  Otherwise,  $q\geq 1$ and from $k=2q+1$, we get by \cref{obs-v_0^2} that
  $$\v'=(r',d'H,a')=R_{\v_1}(\v)=\left(\frac{r_1^2-1}{2}-r_1s,(r_1-2s)H,\frac{2n(r_1-2s)-(qr_1+s)}{r_1}\right).$$
    As $s\leq \frac{r_1-1}{2}$,  $r>r'\geq 0$ and $d'>0$.  By \cite[Theorem 2.5]{Yos99b} and $m=r_1d-2r=1$, the generic $E\in M_H(\v)$ sits in a short exact sequence $$0\to E_1^{\oplus -\langle\v,\v_1\rangle}\to E\to F\to 0,$$ with $E_1\in M_H(\v_1)$ and $F\in M_H(\v')$ generic.  From $q\geq 1$, we see that $r_1^2\leq r_1^2q+r_1s\leq n$, so $r_1,r'\leq n$.  Thus $H^1(X,F)=0=H^1(X,E_1)$ by \cref{thm:bound n r}.  Thus $H^1(X,E)=0$.
 
 Next, let $k=a_1d - 2a = 2q+3$.  Then 
$$\v'=(r',d'H,a')=R_{\v_1}(\v)=\left(\frac{3r_1^2-2r_1s-1}{2},(3r_1-2s)H,\frac{2n(3r_1-2s)-(r_1q+s)}{r_1}\right).$$ 
It is clear that $r>r',d'>0$.  By \cref{obs-v_0^2}  we get that 
 \begin{align*}\v'^2&=\v^2=6n+q\left(r_1(3r_1-2s)-1\right)+s(3r_1-2s)-s\left(\frac{4n+1}{r_1}\right)\\
 &\geq2n+q\left(r_1(3r_1-2s)-1\right)+s(3r_1-2s)+(a_1-1)>2n,
 \end{align*}
 where we use $s\leq r_1-1$.  As all the counterexamples in the theorem satisfy either $\v^2=-2$ or $\v^2=2n$, it follows as in the $d_1=5$ case that resolving the generic sheaf $E\in M_H(\v)$ using  \cite[Theorem 2.5]{Yos99b} and using induction gives the required cohomology vanishing.

 \subsection*{The case $d_1=1$} If $d_1 =1$ and $m \geq 3$, then \cref{prop:condensed inequality} \eqref{eqn:deriving finiteness, 1} and $r \leq 3n$ imply that 
 $$1 \leq \frac{1}{rr_1} (-mnd+rd - r_1) +  \left\{\frac{nd^2+1}{r}\right\} < 1,$$ which is a contradiction. We conclude that if $d_1=1$, then $m=1$ or $2$.

Suppose first that $r_1d-r=2$.  Since $r_1a_1=n+1$,  $r_1\mid r-2n$ and we can write $$r-2n=r_1^2q+r_1s,$$ where $0\leq s<r_1$.  Set
$$A=a_1d-\left(\frac{nd^2+1}{r}\right) = q+\frac{2 q-1}{r}+\frac{s}{r_1}+\frac{2 s}{r r_1}.$$ Since
$$ \frac{2 q-1}{r}+\frac{s}{r_1}+\frac{2 s}{r r_1} = - \frac{1}{r} + \frac{2qr_1 + 2s}{rr_1} + \frac{s}{r_1} \leq  -\frac{1}{r}+\frac{1}{r_1^2}+\frac{r_1-1}{r_1} = 1-\frac{1}{r}+\frac{1}{r_1^2}-\frac{1}{r_1}<1,$$
we have  that $q < A < q+1$. Since $A \leq a_1d - a < A+1$, we conclude that $a_1d-a= q+1$.

  If $q=0$, then $a_1d-a<r_1d-r$ and $H^1(X,E)=0$ for the generic $E\in M_H(\v)$ by \cref{Lem:O_X[1] largest wall}.  
  
  If $q=1$, then $a_1d-a=r_1d-r$ and by \cref{obs-v_0^2} we have $$\v^2=2((r_1-s)d-4).$$  If $\v^2= -2$, then either $d=1$ and $r_1-s=3$ or $d=3$ and $r_1-s=1$. If $d=1$, then $r > r_1 =r+2$, which is a contradiction. If $d=3$, then $r_1= s+1$, so using $3r_1-r=2$ and $r-2n=r_1^2+r_1s$, we get
$$n=-r_1^2+2r_1-1=-(r_1-1)^2<0,$$
 a contradiction. Thus $\v^2\geq 0$ in this case, and  \cref{Lem:O_X[1] largest wall} guarantees that $\v$ satisfies weak Brill-Noether this case.
  
 If $q\geq 2$, then by \cref{obs-v_0^2} we have 
 $$\v'=(r',d'H,a')=R_{\v_1}(\v)=\left(r_1^2-r_1s-2,(r_1-s)H,-q-1+\left(\frac{n+1}{r_1}\right)\left(r_1-s\right)\right),$$
 where $r'\geq -1$, $d'\geq 1$.  
 
 Suppose first that $r'=-1$.  Then $r_1=1$ and $\v'=(-1,H,n-q)=(-1,H,3n-r)$,  $\v=(r,(r+2)H,nr+4n+1)$.  Let $\tilde{\v}:=\v e^{-H}=(r,2H,1)$ for $r>3$. We have $D_{\tilde{\v}}=\varnothing$, so  \cref{prop:isom,prop:birational} imply that for a generic $E'\in M_H(\tilde{\v})$, we have  $\Phi_{X\to X}^{I_\Delta}(E')^\vee\in M_H(1,2H,r)$, i.e. $\Phi_{X\to X}^{I_\Delta}(E')^\vee=I_Z(2H)$, where $Z$ is a generic 0-dimensional subscheme of length $4n+1-r$.  Moreover, by \cref{Lem:FM-vanishing} \eqref{eqn:ShortToLongForEvalSequence}, a generic $E'\in M_H(\tilde{\v})$ sits in a distinguished triangle $$ I_Z(2H)^\vee\to \OO_X^{\oplus r+1}\to E'.$$  Thus a generic $E\in M_H(\v)$ sits in a distinguished triangle $$I_Z(H)^\vee\to \OO_X(H)^{\oplus r+1}\to E.$$  Dualizing and taking the long exact sequence of cohomology sheaves, we see that $E$ is locally free with dual sitting in a short exact sequence $$0\to E^\vee\to\OO_X(-H)^{\oplus r+1}\to I_Z(H)\to 0.$$  Taking cohomology gives $$h^1(X,E)=h^1(X,E^\vee)=h^0(X,I_Z(H))=\max\{0,n+2-(4n+1-r)\}=\max\{0,r+1-3n\}.$$ Hence, $h^1(X,E)=0$ unless $r=3n$, in which case $h^1(X,E)=1$.  Note that here $n >1$, otherwise we would have $n=1, r=3$ and $q=1$ contrary to assumptions. This gives case (3) in the theorem.

If instead $r' \geq 0$, or equivalently $r_1>1$, then the assumption $q\geq 2$ implies that $r_1\leq n$ since by \cref{lem:rewrite product}, $(q+1) r_1^2 \leq 4n+1$.
\cref{thm:bound n r} thus implies that   $H^1(X,E_1)=0$.  Moreover, as $$r'=r_1^2-r_1s-2<r_1^2+r_1s<r_1^2q+r_1s=r-2n\leq n$$ for the same reason we have $H^1(X,F)=0$ for the generic $F\in M_H(\v')$. Now, the generic $E\in M_H(\v)$ sits in a short exact sequence $$0\to E_1^{\oplus -\langle\v,\v_1\rangle}\to E\to R_{E_1}(E) \to 0,$$ where $E_1\in M_H(\v_1)$ and $R_{E_1}(E)\in M_\sigma(\v')$ is generic.  By \cref{lem:WallsFor n>=r}, $D_{\v'}=\varnothing$, so in fact $R_{E_1}(E)=F\in M_H(\v')$. Hence, $H^1(X,E)=0$.

Finally, we may assume $d_1=1$,  $m=r_1d-r=1$ and $r_1a_1=n+1$. By \cref{obs-v_0^2}, we can write $r-n=r_1^2q+r_1s$ where $0\leq s<r_1$. As usual, set
$$A = a_1d-\left(\frac{nd^2+1}{r}\right) = q+\frac{q-1}{r}+\frac{s}{r_1}+\frac{s}{r r_1}.$$ It is easy to check that $q \leq A \leq q+1$ and $A=q$ if and only if $q=1$ and $s=0$.  Since $A \leq a_1d - a < A+1$, we conclude that either $k=a_1d-a=A=q$ (so that $q=1$ and $s=0$) or $k=a_1d-a= q+1$.

If $q=1$, $s=0$, so that $a_1 d - a =1$, then  $$\v=\left(n+r_1^2,\left(\frac{r+1}{r_1}\right)H,\left(\frac{n+1}{r_1}\right)\left(\frac{r+1}{r_1}\right)-1\right)=\left(n+r_1^2,\left(\left(\frac{n+1}{r_1}\right)+r_1\right)H,\left(\frac{n+1}{r_1}\right)^2+n\right)$$ 
and by \cref{obs-v_0^2}, $$\v'=(r',d'H,a')=R_{\v_1}(\v)=(-1,0,-1)=\v(\OO_X[1])$$  and $\v^2=(\v')^2=-2$, so the unique $E\in M_H(\v)$ fits in the short exact sequence $$0\to \OO_X\to E_1^{\oplus d}\to E\to 0.$$ Since $H^1(X,E_1)=0$ by \cref{thm:bound n r}, we conclude  that $H^1(X,E)=H^2(X,\OO_X)=\C$. This gives case (2) of the theorem.

We are left to consider the case $k=a_1d-a = q+1$. By \cref{lem:rewrite product}, $(q+1)r_1^2 \leq 5n$. Hence, $r_1 \leq n$ unless $n=2$, $q=0$ and $r_1=3$ or $n=1$, $q=0$ and $r_1=2$. If $n=2,q=0, r_1=3$, then $\v_1= (3, H, 1)$ and $\v=(5, 2H,1)$. Since $\langle \v,\v_1 \rangle = 0$, we can eliminate that case. If $n=1, q=0, r_1=2$, then $\v_1=(2, H, 1)$ and $\v= (3, 2H, 1)$. In this case, $H^1(X, E_1)=0$ by \cref{prop:rankstwoandthree}. Otherwise $r_1 \leq n$ and \cref{thm:bound n r} implies that $H^1(X,E_1)=0$.  

By \cref{obs-v_0^2} we get $$\v'=(r',d'H,a')=R_{\v_1}(\v)=\left(r_1^2-r_1s-1,(r_1-s)H,\frac{(r_1^2-r_1s+1)n-r}{r_1^2}\right),$$
where clearly $r'\geq 0$ and $d'\geq 1$, and  $$\v^2=(\v')^2=2\left(\frac{n+(r_1^2-r_1s-1)r}{r_1^2}\right)>0.$$ If $q=0$, then $a_1d-ad_1=1=r_1d-rd_1$ so  \cref{Lem:O_X[1] largest wall} eliminates this case, and we may assume that $q\geq 1$.  It follows that $$r'=r_1^2-r_1s-1<r_1^2q+r_1s=r-n\leq 2n,$$ so $\v'$ satisfies weak Brill-Noether by \cref{prop:Rank0,prop:Rank1,thm:bound n r,Thm:n<r<=2n}.  Since $r_1^2-r_1s+1\geq r_1^2-r_1(r_1-1)+1=r_1+1$,  if $r_1\geq 2$, then $a'\geq 0$ and $r'>0$, so for generic $F\in M_H(\v')$, we have $H^1(X,F)=0$.  By \cite[Theorem 2.5]{Yos99b}, from $r_1d-rd_1=1$ we see that the generic $E\in M_H(\v)$ sits in a short exact sequence
$$0\to E_1^{\oplus-\langle\v,\v_1\rangle}\to E\to F\to 0,$$ where $F\in M_H(\v')$ is generic.  Then the vanishing of $H^1(X,E_1)$ and $H^1(X,F)$ gives the vanishing of $H^1(X,E)$.  

If instead, $r_1=1$, then $\v'=(0,H,2n-r)$ and $\v=(r,(r+1)H,nr+2n)$.  By \cite[Theorem 2.5]{Yos99b}, from $r_1d-rd_1=1$ we see that the generic $E\in M_H(\v)$ sits in a short exact sequence
$$0\to E_1^{\oplus-\langle\v,\v_1\rangle}\to E\to \OO_H(L)\to 0,$$ where $\OO_H(L)\in M_H(\v')$ is generic.  By \cref{prop:Rank1}, $H^1(X,E_1)=0$. Taking cohomology we conclude that $$h^1(X,E)=h^1(X,\OO_H(L))=r-2n>0.$$  This gives case (1) of the theorem.

To conclude the proof we show that the next Mukai vector in each of the series defined by \cref{enum:2n<r<=3n torsion quotient}-\ref{enum:2n<r<=3n weird quotient} does not define a counter-example.  In \cref{enum:2n<r<=3n torsion quotient}, the next Mukai vector is $\v=(r,(r+1)H,nr+2n-1)$.  Keeping $\v_1=(1,H,n+1)$, we have $\langle \v,\v_1\rangle=2n(r+1)-r(n+1)-nr-2n+1=1-r<0$, so $$\v'=R_{\v_1}(\v)=(1,2H,3n-r)=\v(I_Z(2H)),$$ where $Z$ is a 0-dimensional subscheme of length $r+n+1$.  As $r_1d-rd_1=1$, it follows from \cite[Theorem 2.5]{Yos99b} that the generic $E\in M_H(\v)$ fits in a short exact sequence
$$0\to \OO_X(H)^{\oplus(r-1)}\to E\to I_Z(2H)\to 0,$$ where $Z\in\Hilb^{r+n+1}(X)$ is generic.  For such $Z$, $H^1(X,I_Z(2H))=0$, so $H^1(X,E)=0$ as claimed. 

In \cref{enum:2n<r<=3n O_X[1] quotient}, the generic stable sheaf $E$ with Mukai vector
$\v= (n+r_1^2,\left(\frac{n+1}{r_1}+r_1\right)H, \left(\frac{n+1}{r_1}\right)^2 + n-1)$ fits in an exact sequence, $$0 \to E_1^{\frac{n+1}{r_1}} \to E \to F \to 0,$$ where $\v(E_1) = (r_1, H, \frac{n+1}{r_1})$ and $F$ is a generic stable sheaf with $\v(F)=(r_1^2-1, r_1H, n-1)$. As $r_1$ and $r_1^2-1$ are both less than or equal to $2n$, by induction $H^1(X, E_1)=H^1(X, F)=0$. Hence $H^1(X,E)=0$.

In \cref{enum:2n<r<=3n weird quotient}, the generic stable sheaf $E$ with Mukai vector $\v= (3n, (3n+2)H, 3n^2 + 4n)$ fits in an exact sequence
$$0 \to \OO_X(H)^{3n} \to E \to \OO_{2H}(D) \to 0,$$ where $\OO_{2H}(D)$ is a general line bundle on a curve with class $2H$ and Euler characteristic $n$. Since such a line bundle has no higher cohomology, it follows that $H^1(X,E)=0$.
\end{proof}

\begin{Cor}\label{Cor:n+rsquare}
Let $r_1$ be a positive integer which divides $n+1$. For $i \geq 0$, let $$\v(i)=\left(n+r_1^2,\left(\left(\frac{n+1}{r_1}\right)+r_1\right)H,\left(\frac{n+1}{r_1}\right)^2+n-i \right).$$ Then $M_H(\v(i))$ fails to satisfy weak Brill-Noether if and only if $i=0$. If $E \in M_H(\v(0))$, then $h^1(E)=1$.
\end{Cor}

\begin{proof}
As in the proof of Theorems \ref{Thm:n<r<=2n} and \ref{Thm:2n<r<=3n},  $E \in M_H(\v(0))$ fits in an exact sequence $$0 \to \OO_X \to E_1^{\oplus (\frac{n+1}{r_1} + r_1)} \to E \to 0,$$ where $E_1 \in M_H((r_1, H, \frac{n+1}{r_1}))$. Since $h^1(E_1)=0$, we conclude that $H^1(E)\cong H^2(\OO_X) \cong \mathbb{C}$.
On the other hand, if $E \in M_H(\v(1))$ is generic, then $E$ fits in an exact sequence 
$$0 \to E_1^{\oplus \frac{n+1}{r_1}} \to E \to F \to 0,$$ where $F$ is generic in $M_H(\w)$ with $\w=(r_1^2-1, r_1H, n-1)$. We claim that $D_{\w}^{BN} = \varnothing$.  Indeed, let $\w'=(n-1,r_1H,r_1^2-1)$ so that $D_{\w'}^{BN}=\varnothing$ by \cref{thm:bound n r}.  Then $D_{\w}^{BN}=\varnothing$ by \cref{Rem:symmetry}.  It follows that $h^1(F)=0$ and consequently $h^1(E)=0$. The corollary then follows from \cref{prop:minimal-a}. 
\end{proof}

\subsection{Small values of $a$} In this subsection, we study the failure of weak Brill-Noether for Mukai vectors $\v=(r, dH, a)$ with small values of $a$. Proposition \ref{Prop:a<=0} shows that if $a$ is negative then weak Brill-Noether holds for $M_H(r,dH,a)$.

\begin{Prop}\label{lem:NoTSSWalls a<=1}
Let $\v=(r,dH,1)$ be a Mukai vector with $\v^2\geq -2$ such that $r\geq 0$, $d>0$.  If $D_\v\ne\varnothing$, then $$n=1, \quad  \v=\left(r,\left(\frac{r+1}{2}\right)H,1\right) \quad \mbox{and} \quad D_\v=\{(2,H,1)\}.$$  Moreover, for $\v=(r,dH,a)$ with $r> 0$, $d>0$, and $a\leq 1$, $\v$ satisfies weak Brill-Noether.
\end{Prop}
\begin{proof}
%Let $\v_1\in D_\v$.  \cref{rem-numerics} \ref{enum:positive r_1} and \ref{enum:totally semistable spherical} imply that
%\begin{align*}
%0>\langle\v,\v_1\rangle&=2ndd_1-r_1a-ra_1=2ndd_1-r_1-r\left(\frac{nd_1^2+1}{r_1}\right) =\frac{1}{r_1}\left(2ndd_1r_1-rnd_1^2-r\right)-r_1\\
%&> \frac{r}{r_1}(nd_1^2-1)-r_1>nd_1^2-1-r_1.
%\end{align*}
%Hence, $nd_1^2\leq r_1$ and $$0<a_1=\frac{nd_1^2+1}{r_1}\leq\frac{r_1+1}{r_1}=1+\frac{1}{r_1}.$$ Therefore, either $r_1>1$ and  $a_1=1$ or $r_1=1$ and $a_1\leq 2$.  But, if $r_1=1$ then  $2\leq nd_1^2+1=a_1 \leq 2$, so we have $n=d_1=1$ and $a_1=2$. Since  $d-r \geq 1$, we have 
%$$0>2ndd_1-r_1a-ra_1 = 2d-a-2r=2(d-r)-1 \geq 1,$$ which is a contradiction. 
%Therefore, $r_1>1$ and $a_1=1$. In this case, \cref{rem-numerics} \ref{enum:positive a_1} implies that $d>d_1$.  Set $d=d_1+s$ with $s>0$.  Using $r_1 d - rd_1 \geq 1$, we get $\frac{r_1d -1}{d_1} \geq r$ and 
%\begin{align*}
%0&>2ndd_1-r_1-r \geq 2ndd_1-r_1 \frac{d_1 + d}{d_1} + \frac{1}{d_1} = 2n(d_1+s)d_1 - (nd_1^2 +1)\frac{(2d_1+s)}{d_1} + \frac{1}{d_1} \\
%& = s\left(nd_1 - \frac{1}{d_1}\right) -2 + \frac{1}{d_1}\geq nd_1-\frac{1}{d_1}-2+\frac{1}{d_1}=nd_1-2.   
%\end{align*}
%We conclude that  $n=d_1=1$ and thus that $r_1=2$ as well. The conditions \cref{rem-numerics} \ref{enum:positive r_1} and \ref{enum:totally semistable spherical} imply that $2d=r+1$, so $\v=(r,\left(\frac{r+1}{2}\right)H,1)$ with $\v_1=(2,H,1)$ and $n=1$, as required.  
The first claim follows from \cref{Rem:symmetry} and \cref{lem:WallsFor n>=r}.

Assume that $\v = (r, dH,1)$ and when $n=1$, $\v \not=(r,\left(\frac{r+1}{2}\right)H,1)$. Then $D^{BN}_{\v}\subset D_{\v}=\varnothing$.  \cref{thm:summing up inequalities} implies that the generic $E\in M_H(r,dH,1)$ satisfies $H^1(X,E)=0$.  By \cref{prop:minimal-a} and the fact that $r>0$, $(r,dH,a)$ satisfies weak Brill-Noether for all $a\leq 1$.  

If $n=1$ and $\v =(r,\left(\frac{r+1}{2}\right)H,1)$,  by  \cref{thm:summing up inequalities,prop:minimal-a}, it suffices to show that the generic $E'\in M_H(r,\left(\frac{r+1}{2}\right)H,1)$ satisfies $H^1(X,E')=0$.  As $D_{(r,\left(\frac{r+1}{2}\right)H,1)}$ consists of the unique Mukai vector $\v_1=(2,H,1)$, $\v_1$ is the unique effective spherical class in the primitive isotropic sublattice $\HHH$.   \cref{lem:max-HNF2} implies that the generic $E'\in M_H(r,\left(\frac{r+1}{2}\right)H,1)$ fits into a short exact sequence $$0\to T_1\to E'\to F\to 0,$$ where $T_1$ is the unique stable spherical sheaf with $\v(T_1)=\v_1$ and $F$ is a $\sigma$-stable object such that $\v(F)=(r-2,\left(\frac{r-1}{2}\right)H,0)$.  By \cref{Prop:a<=0}, $D_{\v(F)}=\varnothing$ and the generic $F\in M_\sigma(\v(F))$ is a stable sheaf with vanishing $H^1$.  As $H^1(X,T_1)=0$ by \cref{prop:rankstwoandthree}, we get $H^1(X,E')=0$.  
\end{proof}
\begin{Thm}\label{Lem:NoTSSWalls a=2}
Let $\v=(r,dH,2)$ be a Mukai vector with $\v^2\geq -2$ such that $r\geq 0$, $d>0$.  If $D_\v\ne\varnothing$, then either
\begin{enumerate}
    \item $n=1$ and one of the following holds
    \begin{enumerate}
        \item $\v=(11,5H,2)$ and $\v_1=(5,2H,1)$
        \item $\v=(23,7H,2)$ and $\v_1=(10,3H,1)$
        \item $\v = (r, dH, 2)$ with $d\geq 3$, $2d-3 \leq r \leq 2d-1$ and $\v_1=(2,H,1)$
         \item $\v=(12,5H,2)$ and $\v_1=(5,2H,1)$; or
    \end{enumerate}
    \item $n=2$ and one of the following holds
    \begin{enumerate}
        \item $\v=(11,4H,2)$ and $\v_1=(3,H,1)$
        \item $\v=(7,3H,2)$ and $\v_1=(3,H,1)$
        \item $\v=(8,3H,2)$ and $\v_1=(3,H,1)$; or
    \end{enumerate}
    \item $n=3$, $\v=(11,3H,2)$, and $\v=(4,H,1)$.
\end{enumerate}
Moreover, the generic $E\in M_H(\v)$ satisfies $H^1(X,E)=0$ except when $n=1$ and  $\v=(5,3H,2)$.
\end{Thm}
\begin{proof}
Let $\v_1\in D_\v$.  As in the proof of \cref{lem:NoTSSWalls a<=1}, conditions \cref{rem-numerics} \ref{enum:positive r_1} and \ref{enum:totally semistable spherical} imply
$$0>\langle\v,\v_1\rangle>nd_1^2-1-2r_1.$$
Thus  $r_1\geq\frac{nd_1^2}{2}$ and
$$0<a_1=\frac{nd_1^2+1}{r_1}\leq\frac{2r_1+1}{r_1}=2+\frac{1}{r_1}.$$
If $r_1=1$, since $a_1 = nd_1^2+1 \leq 3$, we conclude that either $$n=d_1=1, \ a_1=2, \  \langle \v, \v_1 \rangle = 2d-2r-2 \quad \mbox{or} \quad d_1=1,\  n=2, \ a_1=3, \ \langle \v, \v_1 \rangle= 4d-3r-2.$$ Since $d-r \geq 1$,  $\langle \v, \v_1 \rangle$ cannot be negative in either case. Hence, $r_1 > 1$ and $a_1=1$ or 2.
If $a_1=2$, then \cref{rem-numerics} \ref{enum:positive a_1} gives $2d-2d_1>0$, so that we can write $d = d_1 + s$ for a positive integer $s$. Since $2r_1 = nd_1^2 +1$ and $r_1 d - d_1 r \geq 1$, we obtain
$$ 0 > 2ndd_1 - 2r_1 - 2r \geq 2ndd_1 - 2r_1 \frac{d_1+d}{d_1} + \frac{2}{d_1} = s\left(nd_1 - \frac{1}{d_1}\right) -2 + \frac{2}{d_1} >-1,$$ which is a contradiction.

 We conclude that if $\v_1=(r_1,d_1H,a_1)\in D_\v$, then $a_1=1$, $r_1=nd_1^2+1$, $d=2d_1+s$ for a positive integer $s$ and $r\leq \frac{r_1d-1}{d_1}$.  Then $\langle\v,\v_1\rangle<0$ becomes
\begin{align*}
    0&>2ndd_1-2r_1-r\geq2ndd_1-r_1\frac{2d_1 + d}{d_1} + \frac{1}{d_1}=s\left(nd_1 - \frac{1}{d_1}\right)-4 + \frac{1}{d_1}.
\end{align*}
In particular, $1\leq nd_1\leq 3$.  If $n=d_1=1$, then $s$ can be an arbitrary positive integer and $r_1=2$, $d=2+s$. Hence, $d$ is an arbitrary integer greater than or equal to $3$. Furthermore, $$r_1d-rd_1 = 2d - r >0 \quad \mbox{and} \quad 0> 2ndd_1 - r_1a - r a_1 = 2d -4 -r.$$ Hence, $2d > r > 2d-4$. This gives case $(1)(c)$.

If $d_1=1$ and $n>1$, then from $0 > s(n-1) - 3$, we conclude that $n=2$, $1 \leq s  \leq 2$ or $n=3$, $s=1$. First assume $n=2$. Then $v_1=(3, H, 1)$. If further $s=2$, then $d=4$ and $$r_1d - rd_1 =12 - r >0>2ndd_1 - a r_1 -  a_1r = 10 - r.$$ Hence $\v = (11, 4H, 2)$, which is case 2(a). If $s=1$, then $d=3$ and $9-r > 0 >6 -r$, so $7 \leq r \leq 8$. Hence, $\v = (7, 3H, 2)$ or $\v= (8, 3H, 2)$ giving cases 2(b) or 2(c). Finally, if $n=3$, $s=1$, we have $\v_1 = (4, H, 1)$, $d=3$ and $12 - r >0 > 10 - r$. Hence, $\v=(11, 3H,2)$, which is case (3).

If $d_1 > 1$, then $n=s=1$ and $2 \leq d_1 \leq 3$. If $d_1 =2$, then $\v_1 = (5, 2H,1)$, $d=5$ and $25-2r > 0 > 10 -r$. Hence, $\v= (11, 5H, 2)$ or $\v= (12, 5H, 2)$, which are cases 1(a) or 1(d). If $d_1=3$, then $\v_1 = (10, 3H, 1)$, $d= 7$ and $70 - 3r > 0 > 22 -r$. Hence, $\v= (23, 7H, 2)$, which is case 1(b). This concludes the classification. 

\cref{Lem:O_X[1] largest wall} implies that
the generic $E\in M_H(\v)$ satisfies $H^1(X,E)=0$ in cases (1)(a),(b), and (d),  (2)(b) and (c), and (3).  In case (2)(a), we have $\langle\v,\v_1\rangle=-1$, so  \cref{lem:max-HNF1} implies that the generic $E\in M_H(\v)$ has HN-filtration
$$0\to T_1\to E\to F\to 0,$$
where $T_1$ is the unique stable spherical sheaf with $\v(T_1)=(3,H,1)$ and $F$ is a $\sigma$-stable object such that $\v(F)=(7,3H,1)$.  By \cref{lem:NoTSSWalls a<=1}, $D_{\v(F)}=\varnothing$, so the generic $F\in M_\sigma(7,3H,1)$ is a stable sheaf that is $\sigma$-stable and satisfies $H^1(X,F)=0$.  Moreover,  \cref{lem:NoTSSWalls a<=1} shows that $H^1(X,T_1)=0$, so $H^1(X,E)=0$ for generic $E\in M_H(\v)$.

In case (1)(c),  first suppose that $\v=(2d-1,dH,2)$.  Then the generic $E\in M_H(\v)$ fits in the short exact sequence 
$$0\to T_1^{\oplus 3}\to E\to F\to 0$$
with $M_H(2,H,1)=\{T_1\}$ and $F\in M_\sigma(2d-7,(d-3)H,-1)$.  If $d\geq 4$, then the generic $F\in M_\sigma(2d-7,(d-3)H,-1)$ is in $M_H(2d-7,(d-3)H,-1)$ and satisfies  $H^1(X,F)=0$ by \cref{Prop:a<=0}.  As $H^1(X,T_1)=0$ by \cref{lem:NoTSSWalls a<=1}, we get $H^1(X,E)=0$ for generic $E\in M_H(\v)$ when $d\geq 4$.  If $d=3$, then $\v=(5,3H,2)$, and the unique $E\in M_H(5,3H,2)$ sits in a short exact sequence
\begin{equation}\label{eqn:(5,3,2)}
0\to\OO_X\to T_1^{\oplus 3}\to E\to 0,
\end{equation}
from which we get $h^1(X,E)=1$.  

Next suppose that $\v=(2d-2,dH,2)$.  In this case, the generic $E\in M_H(\v)$ sits in a short exact sequence
$$0\to T_1^{\oplus 2}\to E\to F\to 0,$$
with $F\in M_H(2d-6,(d-2)H,0)$ for $d\geq 3$.  Thus $h^1(X,E)=h^1(X,F)=0$ by \cref{lem:NoTSSWalls a<=1}.  

Finally suppose that $\v=(2d-3,dH,2)$.  In this case, the generic $E\in M_H(\v)$ sits in a short exact sequence
$$0\to T_1\to E\to F\to 0,$$
with $F\in M_H(2d-5,(d-1)H,1)$ for $d\geq 3$.  Thus $h^1(X,E)=h^1(X,F)=0$ by \cref{lem:NoTSSWalls a<=1}.  
\end{proof}

\subsection*{Small values of $d$} Finally, the following corollary discusses the cases for small values of $d$.

\begin{Cor}\label{cor:smalld}
Let $X$ be a K3 surface such that $\Pic(X)= \mathbb{Z} H$ with $H^2 =2n$.
Let $\v=(r,dH,a)$ be a Mukai vector such that $ 0 \leq a, r$, $0 < d \leq 3$ and $\v^2 \geq -2$. Then $\v$ satisfies weak Brill-Noether unless $n=1$ and $\v=(2,3H,5)$ or $\v= (5, 3H, 2)$.
\end{Cor}

\begin{proof}
By \cref{prop:Rank0} and  \cref{prop:Rank1}, we may assume that $r\geq 2$. By \cref{prop:minimal-a}, we may consider the Mukai vector $\v_0$ with the smallest square. Let $\v_1 \in D_{\v_0}^{BN}$. Since $dr_1 - r d_1 > 0$ and $0 < r_1 < r$ by \cref{lem:useful consequences}, we must have $0 < d_1 < d$.

If $d=1$, then $D_{\v}^{BN} = \varnothing$, proving the corollary in this case.

If $d=2$, then  $d_1=1$, $2r_1 - r >0$ and $a_1 r_1 = n+1$. Hence, $r_1 \leq n+1$ and $r \leq 2n+1 \leq 3n$. The corollary in this case then follows from  \cref{Thm:n<r<=2n} and \cref{Thm:2n<r<=3n}.

If $d=3$, then $d_1 =1$ or $2$. If $d_1=1$, then $3r_1 - r >0$ and $a_1 r_1 = n+1$. If $a_1=1$, then $3a_1 - a \geq 1$, implies that $a \leq 2$ and the corollary follows from \cref{Lem:NoTSSWalls a=2}. In this case, we get the exception $n=1$ and $\v=(5, 3H, 2)$. Otherwise, $r_1 \leq n$ and $r < 3n$. By Theorems \ref{thm:bound n r}, \ref{Thm:n<r<=2n} and \ref{Thm:2n<r<=3n}, the cohomology vanishes unless $n=1$ and  $\v= (2, 3H, 5)$. If $d_1=2$, then $3r_1 - 2r >0$, $a_1 r_1 = 4n+1$ and $1 \leq 3a_1 - 2a$. If $a_1=1$ or $2$, then $1 \leq 3a_1 - 2a$ implies that $a\leq 2$. Hence, the corollary follows from \cref{Lem:NoTSSWalls a=2}. If $a_1 \geq 3$, then $r_1 \leq \frac{4n+1}{3}$ and $r \leq 2n$. The corollary follows from Theorems \ref{thm:bound n r} and \ref{Thm:n<r<=2n}.
\end{proof}

\begin{Rem}
The case $d=1$ of \cref{cor:smalld} was proved by Markman in \cite{Markman:BN}.
\end{Rem}

\begin{Cor}
Let $X$ be a K3 surface such that $\Pic(X)= \mathbb{Z} H$ with $H^2 =2n$.
Let $\v=(r,dH,a)$ be a Mukai vector such that $ 0< r$, $\v^2 \geq -2$ and $nd^2 < 3r-1$. Then $\v$ satisfies weak Brill-Noether unless $n=1$ and $\v= (5, 3H, 2)$.
\end{Cor}

\begin{proof}
By \cref{prop:Rank0} and  \cref{prop:Rank1}, we may assume that $r\geq 2$. As $\v^2\geq -2$, it follows from $nd^2 < 3r-1$ that $a < 3$. By \cref{Prop:a<=0,lem:NoTSSWalls a<=1,Lem:NoTSSWalls a=2}, we conclude that the only $\v$ that does not satisfy weak Brill-Noether occurs when $n=1$ and $\v= (5,3H,2)$.
\end{proof}

\section{General Theorems for Computations}\label{sec-computations}
In this section, we describe several techniques for computing the cohomology of the generic sheaf in $M_H(\v)$.

\subsection{Destabilizing Lines bundles and Torsion Quotients}
This family of examples accounts for about half of the vectors with $D_{\v} \neq \varnothing$. Although this theorem will be subsumed by \cref{Prop:CohomologyCalc a<=2}, we highlight it here for the prevalence of these examples. 

\begin{Thm}\label{Thm:LineBundleSubTorsionQuotient}
Let $\v=(r,dH,a)$ be a Mukai vector such that $r,a\geq0$ and $d>0$ and $\v^2\geq -2$.  Suppose that $\v_1=(1,d_1H,1+d_1^2 n)\in D_{\v}$ induces the largest totally semistable wall and satisfies $\rk R_{\v_1}(\v)=0$.
Then the generic $E\in M_H(\v)$ satisfies $$h^1(X,E)=\max\{ 0,r(nd_1^2+1)-a\} \quad \mbox{and} \quad  h^2(X,E)=0.$$
\end{Thm}
\begin{proof}
From the assumption that $$0=\rk R_{\v_1}(\v)=r+\langle\v,\v_1\rangle r_1=r+\langle\v,\v_1\rangle ,$$ we get $\langle\v,\v_1\rangle=-r$.   Thus $$\v_2=R_{\v_1}(\v)=\v+\langle\v,\v_1\rangle\v_1=(0,d-rd_1,a-r(nd_1^2+1)).$$
By \cref{rem-numerics} \ref{enum:positive r_1}, we have $d-rd_1=r_1d-rd_1>0$, so by \cref{prop:Rank0}, the generic $L\in M_H(\v_2)$ is $\sigma_{(s,t)}$-stable for all $(s,t)$ between $\GG$ and $\CC$.  By \cref{prop:Rank1}, the same holds for $\OO_X(d_1H)$, the unique element in $M_H(\v_1)$.  Moreover, we must have $\v^2=(R_{\v_1}(\v))^2=2n(d-rd_1)^2>0$.  It follows from \cref{lem:max-HNF1,lem:max-HNF2} that the generic $E\in M_H(\v)$ has Harder-Narasimhan filtration given by the short exact sequence in $\AA$:
$$0\to\OO_X(d_1H)^{\oplus r}\to E\to R_{\OO_X(d_1H)}(E)\to 0,$$ 
where $R_{\OO_X(d_1H)}(E)=L$ is generic in $M_H(\v_2)$.  As $H^i(X,\OO_X(d_1H))=0$ for $i>0$ and $\v_2$ satisfies weak Brill-Noether, we get $$h^1(X,E)=h^1(X,L)=\max\{0,-(a-r(nd_1^2+1))\},$$ as claimed.
\end{proof}
\subsection{The effect of tensoring}
Given a Mukai vector $\v=(r,dH,a)$ for which we know the generic cohomology, it is natural to study how $h^1$ changes as we tensor by $\OO(pH)$ for $p>0$.  The following proposition demonstrates the fruitfulness of this approach.
\begin{Prop}\label{prop:Tensor}
Let $X$ be a K3 surface with $\Pic(X)=\Z H$ and $H^2=2n$. Let $p>0$ be a positive integer. Let $E$ be a stable sheaf with $\v(E)=(r,dH,a)$ such that $d>0$.  Set $F_p:=\Phi_{X\to X}^{I_\Delta^\vee[2]}(\OO_X(-pH))$. Then
$$\hom(E,F_p)\leq h^1(X,E(pH))\leq\hom(E,F_p)+h^1(X,E)\chi(\OO_X(pH)).$$
In particular, if $H^1(X,E)=0$, then
$h^1(X,E(pH))=\hom(E,F_p).$
If, in addition, $\frac{d}{r}>\frac{p}{p^2 n+1}$, then $H^1(X,E(pH))=0$.
\end{Prop}
\begin{proof}
Since $H^i(X,\OO_X(pH))=0$ for $i\ne 0$, we have $$H^q(X,R\Gamma(X,E)\otimes\OO_X(pH))=H^q(X,E)\otimes H^0(X,\OO_X(pH)).$$  Tensor the distinguished triangle $$\Phi_{X\to X}^{I_\Delta}(E)\to R\Gamma(X,E)\otimes\OO_X\to E\to\Phi_{X\to X}^{I_\Delta}(E)[1]$$ from \cref{Lem:FM-vanishing} by $\OO_X(pH)$. Since $H^2(X,E)=0$, taking $R\Gamma(X,-)$ gives $$H^1(X,E)\otimes H^0(X,\OO_X(pH))\to H^1(X,E(pH))\to H^2(X,\Phi_{X\to X}^{I_\Delta}(E)\otimes\OO_X(pH))\to 0.$$  Thus \begin{align}
    \begin{split}
        h^2(X,\Phi_{X\to X}^{I_\Delta}(E)\otimes\OO_X(pH))&\leq h^1(X,E(pH))\\
        &\leq h^2(X,\Phi_{X\to X}^{I_\Delta}(E)\otimes\OO_X(pH))+h^1(X,E) h^0(X,\OO_X(pH))\\
        &=h^2(X,\Phi_{X\to X}^{I_\Delta}(E)\otimes\OO_X(pH))+h^1(X,E)\chi(\OO_X(pH)).
    \end{split}
\end{align}

We claim that $h^2(X,\Phi_{X\to X}^{I_\Delta}(E)\otimes\OO_X(pH))=\hom(E,F_p)$ and that $F_p$ is a stable sheaf of slope $\frac{p}{p^2n+1}$.  The  proposition easily follows from this claim.  By adjunction and Serre duality, we have 
\begin{align}
    \begin{split}
        h^2(X,\Phi_{X\to X}^{I_\Delta}(E)\otimes\OO_X(pH))&=\ext^2(\OO_X(-pH),\Phi_{X\to X}^{I_\Delta}(E))\\
        &=\hom(\Phi_{X\to X}^{I_\Delta}(E),\OO_X(-pH)) \\&=\hom(E,\Phi_{X\to X}^{I_\Delta^\vee[2]}(\OO_X(-pH)))\\
        &=\hom(E,F_p).
    \end{split}
\end{align}
 \cref{prop:Rank1} implies that $\Phi_{X\to X}^{I_\Delta}(\OO_X(pH))^\vee\in M_H(p^2n+1,pH,1)$, and by Grothendieck duality, $\Phi_{X\to X}^{I_\Delta}(\OO_X(pH))^\vee\cong\Phi_{X\to X}^{I_\Delta^\vee[2]}(\OO_X(-pH))=F_p$.  The claim follows.
\end{proof}

\begin{Rem}\label{rem:tensor}
Let $E \in M_H(\v)$ be generic. If $H^1(X,E)=0$, then $h^1(X,E(pH))=\hom(E,F_p)$ by \cref{prop:Tensor}. Since 
$$-\langle\v,\v(F_p)\rangle=-\langle\v,(np^2+1,pH,1)\rangle=a(np^2+1)+r-2npd,$$
as in \cref{lem:max-HNF1,lem:max-HNF2},  \cite[Lemma 6.8 and Proposition 8.3]{BM14b} imply  that 
$$\hom(E,F_p)= \max\{0,a(np^2+1)+r-2npd\}.$$
Hence, if $\v$ satisfies weak Brill-Noether, then  
$$h^1(X,E(pH))=\max\{0,a(np^2+1)+r-2npd\}.$$
\end{Rem}

By \cref{rem:tensor},  when $\v$ satisfies weak Brill-Noether, we can  calculate the cohomology of $E(pH)$ for all $p>0$. When $\v$ does not satisfy weak Brill-Noether, so that $H^1(X,E)\ne 0$ for the generic $E\in M_H(\v)$, it follows
that $D_\v\ne\varnothing$.
To determine the cohomology of $E(pH)$, we relate the totally semistable walls of $\v(E(pH))=e^{pH}\v$ to the vectors in $D_\v$. 

Recall from \eqref{eqn:CenterRadius} that for $\v_1=(r_1,d_1H,a_1)$, the wall $C^\v_{\v_1}$ is defined by
\begin{equation}\label{eqn:Circle to shift}
t^2+(s-\alpha)^2=\alpha^2-\frac{a_1d-ad_1}{n(rd_1-r_1d)}=\left(\tfrac{d}{r}-\alpha \right)^2-
\frac{ \v^2}{2nr^2},
\end{equation}
where $\alpha=\frac{ra_1-r_1a}{2n(rd_1-r_1d)}$ is necessarily less than $d/r$ (i.e. we are only considering the family of circles to the left of $s=d/r$). This expression for the equation of the circle makes it clear that we can study $C^\v_{\v_1}$ by referring only to $\alpha$, the $s$-coordinate of its center.  It is useful to consider the purely numerical wall given by \eqref{eqn:Circle to shift} for arbitrary $\alpha$ (in the range that makes sense) and denoted by $C^\v_\alpha$. We will refer to such circles as {\em pseudo-walls}.   Moreover, the wall $C^{e^{pH}\v}_{e^{pH}\v_1}$ for $e^{pH}\v$ is just $C^\v_{\v_1}$ shifted to the right by $p$.  In other words, $C^{e^{pH}\v}_{\alpha+p}=C^\v_\alpha+(p,0)$.
We observe that \eqref{eqn:Circle to shift} is equivalent to
\begin{equation}\label{eqn:otherformofcircle}
t^2+s^2-2\alpha s=-2 \alpha \tfrac{d}{r}+\tfrac{a}{rn}.
\end{equation}
\begin{Prop}\label{prop:relation-wall}
Let $\v=(r,dH,a)$ be a Mukai vector with $r,d>0$ and $\v^2\geq -2$.  Let $p$ be a positive integer, and assume that $a\geq-pdn$, with strict inequality if $n=1=p$.
Then $D_{e^{pH}\v} \subset e^{pH}(D_{\v} \cup \{(1,0,1)\})$.
\end{Prop}

\begin{proof}
Let $\v'_1=(r_1',d_1'H,a_1')\in D_{e^{pH}\v}\backslash\{(1,pH,p^2n+1)\}$, where $(1,pH,p^2n+1)=\v(\OO_X(pH))$.  
From \cref{Def:DestabilizingMukaiVectors}, it follows that $C^{e^{pH}\v}_{\v'_1}\cap\{s=0,t>0\}\ne\varnothing$.  Moreover, since $e^{-pH}$ is an isometry of $\Hal(X,\Z)$ and
$$e^{-pH}\v'_1=(r_1',(d_1'-r_1'p)H,a_1'+r_1'p^2n-2nd_1'p),$$
it follows that $e^{-pH}\v'_1\in D_\v$ if and only if $C^\v_{e^{-pH}\v'_1}\cap \{s=0,t>0\} \ne \varnothing$.   As we noted above, $C^\v_{e^{-pH}\v'_1}=C^{e^{pH}\v}_{\v'_1}-(p,0)$.  Thus for $\v'_1\ne \v(\OO_X(pH))\in D_{e^{pH}\v}$, we can summarize this equivalence geometrically by $e^{-pH}\v'_1\in D_{\v}$ if and only if the semicircle $C^{e^{pH}\v}_{\v'_1}$, which intersects $\{s=0,t>0\}$, also intersects the positive ray $\{s=p,t>0\}$.  So we want to show that for any $\v'_1\ne \v(\OO_X(pH))\in D_{e^{pH}\v}$, $C^{e^{pH}\v}_{\v'_1}$ intersects the positive ray $\{s=p,t>0\}$.  

Since we do not know anything a priori about the invariants of $\v'_1$, we prove a more general claim.  Namely, we claim that if $\alpha$ is such that $C^{e^{pH}\v}_\alpha\cap\{s=0,t>0\}\ne\varnothing$, then $C^{e^{pH}\v}_\alpha\cap\{s=p,t>0\}\ne\varnothing$ and this pseudo-wall is contained in $\H^0$.  As $C^{e^{pH}\v}_{\v'_1}\subset\H^0$ is one such semi-circle, this will guarantee that $e^{pH}\v'_1\in D_\v$ for every $\v(\OO_X(pH))\ne\v'_1\in D_{e^{pH}\v}$, as required.

To prove the claim, let $\lambda$ be the unique value of $\alpha$ such that $C^{e^{pH}\v}_\lambda$ contains the origin.  As of the $C^{e^{pH}\v}_\alpha$ that we are interested in lie above $C^{e^{pH}\v}_\lambda$ (guaranteed by the requirement that $C^{e^{pH}\v}_\alpha\cap\{s=0,t>0\}\ne\varnothing$), it suffices to show that $C^{e^{pH}\v}_\lambda\cap\{s=p,t\geq0\}\ne\varnothing$ and that if $(s,t) \in {\Bbb H}$ satisfies
\begin{equation}\label{eqn:verticalstrip}
t^2+s^2-2\lambda s > -2 \lambda \left(\tfrac{d}{r}+p\right)+\tfrac{a+rp^2n+2ndp}{rn},\;
0<s<p,
\end{equation}
then $(s,t) \in {\Bbb H}^0$ in \eqref{eq:geometric-def}.
Note that this is just the part of the vertical strip $\{0<s<p,t>0\}$ lying above the circle $C^{e^{pH}\v}_\lambda$.

As a first step, a simple calculation using \eqref{eqn:otherformofcircle} gives $\lambda=\frac{a+rp^2n+2ndp}{2n(d+rp)}=\frac{p}{2}+\frac{a+ndp}{2n(d+rp)}$, so \eqref{eqn:verticalstrip} becomes
\begin{equation}\label{eqn:verticalstrip2}t^2+s^2-2\lambda s>0,0<s<p.
\end{equation}
Moreover, we see that $C^{e^{pH}\v}_\lambda\cap\{s=p,t\geq0\}\ne\varnothing$ since setting $s=p$ and solving for $t^2$ in \eqref{eqn:otherformofcircle} for $e^{pH}\v$ gives 
\begin{equation}
t^2=\frac{p(pdn+a)}{n(d+rp)}\geq0, 
\end{equation}
so $C^{e^{pH}\v}_\lambda\cap\{s=p,t\geq0\}=\{(p,\sqrt{\frac{p(pdn+a)}{n(d+rp)}})\}$.

It remains to prove that the subset described in \eqref{eqn:verticalstrip2} is contained in $\H^0$.  We observe that the circle $C^{e^{pH}\v}_\lambda$ gets larger with $\lambda$, so it suffices to prove that the subset
\begin{equation}\label{eqn:verticalstrip3}
    t^2+s^2-ps>0,0<s<p 
\end{equation} is contained in $\H^0$, which is just \eqref{eqn:verticalstrip2} for the minimal possible value of $\lambda=\frac{p}{2}$, occurring when $a=-pdn$.   

To do this, we break the interval $(0,p)$ into three pieces.  We will refer to \cref{Fig:C_lambda} to help explain our argument.  

\begin{figure}[h]
   \begin{tikzpicture}[scale=5]
   %axes
   \draw [->] (0,0) node [below] {$s=0$} -- (2.2,0) node[above] {$s$};
   \draw[->] (0,0) -- (0,1.3) node[left] {$t$};
   \draw[-] (2,0) node [below] {$s=p$}--(2,1.3) ;
   \draw[-,dotted] (0.5,0) node [below] {$s=\frac{1}{2\sqrt{n}}$} -- (0.5,0.5);
   \draw[-,dotted] (1.5,0) node [below] {$s=p-\frac{1}{2\sqrt{n}}$} -- (1.5,0.5);
   \draw[red,thick,domain=0:2] plot (\x,{sqrt(2*\x-pow(\x,2))});
    \draw[green,thick,domain=0:2.1] plot (\x,{sqrt(2.1*\x-pow(\x,2))});
   \draw[blue,thick,domain=0:0.5] plot (\x,{0.5-sqrt(0.25-pow(\x,2))});
 \draw[blue,thick,domain=1.5:2] plot (\x,{0.5-sqrt(0.25-pow(\x-2,2))});
 \draw[blue,thick,domain=0.5:1.5] plot (\x,{0.5});
 \draw[gray,domain=0:0.5] plot (\x,{0.5+sqrt(0.25-pow(\x,2))});
  \draw[gray,domain=1.5:2] plot (\x,{0.5+sqrt(0.25-pow(\x-2,2))});
\node[below] at (0.4,0.5) {$D_+$};
  \node[below] at (1.6,0.5) {$D_-$};
  \node[above] at (1,0.5) {$t=\frac{1}{2\sqrt{n}}$};
  \node[below] at (1,1) {$C^{e^{pH}\v}_{p/2}$};
  \node[above] at (1,1.1) {$C^{e^{pH}\v}_{\lambda},\lambda>\frac{p}{2}$};
  \node[right] at (0.4,0.78) {$(s_0,t_0)$};
   \end{tikzpicture}
   \caption{Bounding $C^{e^{pH}\v}_\lambda$ from below}\label{Fig:C_lambda}
\end{figure}

When $0<s<\frac{1}{2\sqrt{n}}$, we claim that the region describe in \eqref{eqn:verticalstrip3} is contained in $U_+$.  Indeed, let $D_+$ be the right semi-circle given by
\begin{equation}
s^2+(t-\frac{1}{2\sqrt{n}})^2=\frac{1}{4n},0<s<\frac{1}{2\sqrt{n}},
\end{equation} 
whose lower half gives the boundary curve of $U_+$ (which is blue in \cref{Fig:C_lambda}).  Then it is easy to see that
$$
D_+ \cap C^{e^{pH}\v}_{p/2}=\{(0,0),(s_0,t_0)\},\; s_0=\frac{p}{np^2+1},\;
t_0=\sqrt{n}p s_0.
$$
Thus $t_0/s_0=p\sqrt{n}> 1$ (from the restriction on $p$ and $n$ in case $a=-ndp$). It follows that $(s_0,t_0)$ is on the upper half of $D_+$.  As $C^{e^{pH}\v}_{p/2}$ lies above the boundary curve of $U_+$ (i.e. the lower half of $D_+$) near the origin, and the entire lower half of $D_+$ lies below the line $t=s$, it follows that $C^{e^{pH}\v}_{p/2}$ lies above the boundary curve of $U_+$ when $0<s<\frac{1}{2\sqrt{n}}$, giving the claim.  

We similarly define $D_-$ to be the left semi-circle given by
\begin{equation}
(s-p)^2+(t-\frac{1}{2\sqrt{n}})^2=\frac{1}{4n},p-\frac{1}{2\sqrt{n}}<s<p,
\end{equation} 
whose lower half gives the boundary curve of $U_-+(p,0)$.   As the maxima of the lower halves of $D_+$ and $D_-$ are identical (equal to $t=\frac{1}{2\sqrt{n}}$), we can connect the these lower halves by the line segment $\{\frac{1}{2\sqrt{n}}<s<p-\frac{1}{2\sqrt{n}},t=\frac{1}{2\sqrt{n}}\}$ to give the blue curve in \cref{Fig:C_lambda}.  As the center of a circle $C^{e^{pH}\v}_\lambda$ is $\lambda$, the center of $C^{e^{pH}\v}_{p/2}$ is $\frac{p}{2}$, so  the symmetry of a circle about its center guarantees that $C^{e^{pH}\v}_{p/2}$, which is the red curve in \cref{Fig:C_lambda}, remains above the blue curve.

As $U_+\subset\H^0$ and $U_-+(p,0)\subset\H^0$ by \cite[Prop. 2.6]{Yos17},
we see that when $0<s<\frac{1}{2\sqrt{n}}$ or $p-\frac{1}{2\sqrt{n}}<s<p$, the region above $C^{e^{pH}\v}_{p/2}$ is in $\H^0$.  

To conclude the proof, suppose that $(s,t)$ in \eqref{eqn:verticalstrip3} with $\frac{1}{2\sqrt{n}}\leq s\leq p-\frac{1}{2\sqrt{n}}$ were not in $\H^0$.  
Unwinding the definition of $\H^0$ in \eqref{eq:geometric-def}, we see that
$$
\H^0=\H\backslash\bigcup_{\v_1\in\Delta_+(X)}\Set{(s,t)\in\H \ | \ s=\frac{d_1}{r_1},t\leq\frac{1}{r_1\sqrt{n}}}.$$
Thus it follows that there is some $\v_1=(r_1,d_1H,a_1)\in\Delta_+(X)$ such that $(s,t)$ satisfies
$$s=\frac{d_1}{r_1},t\leq\frac{1}{r_1\sqrt{n}}.$$
As $C^{e^{pH}\v}_{p/2}$ is above blue horizontal line in \cref{Fig:C_lambda} for these values of $s$, so we have $t>\frac{1}{2\sqrt{n}}$, 
which forces $r_1=1$.  Thus $s=d_1$ is an integer satisfying $1\leq s\leq p-1$, so $p\geq 2$ and we have
$$\frac{1}{\sqrt{n}}\geq t>\sqrt{-s^2+ps}=\sqrt{\frac{p^2}{4}-\left(s-\frac{p}{2}\right)^2}\geq 1,$$
a contradiction.  
\end{proof}
\begin{Rem}
Even when $a<-pdn$, the same argument allows one to determine precisely how $D_\v$ and $D_{e^{pH}\v}$ are related.
\end{Rem}
The significance of this result is that it allows us to study $\v$ with small $d$ relative to $r$.  Then, from the result we can determine the largest totally semistable wall for $\v e^{pH}$ by comparing the totally semistable walls determined by $\v_1=\v(\OO(pH))=(1,pH,1+p^2n)$ and $\v_1=\v_1' e^{pH}$, where $\v_1'\in D_{\v}$ gives the largest totally semistable wall for $\v$.  Whichever of these two walls has larger radius is the largest totally semistable wall for $\v e^{pH}$.  In particular, when $D_\v=\varnothing$, we know the only possible totally semistable wall for $\v e^{pH}$ is given by $\v_1=(1,pH,1+p^2n)$, from which it is easy to determine the cohomology of the generic sheaf in $M_H(\v e^{pH})$.  This is the motivating rationale for the following results.
\begin{Thm}\label{Prop:CohomologyCalc a<=2}
Let $r,d,i$ be integers such that $r\geq 0$, $d>0$, and $i\geq 0$.  Let $E \in M_X(\v)$ be a generic sheaf. 
\begin{enumerate}
    \item If $\v=(r,dH,-i)$ with $ i\leq r$, then  $h^1(X,E)=0$ and
    $$h^1(X,E(pH))=\max\{0,r-2pdn-(p^2n+1)i\}$$
     for $p> 0$.
    \item Let $\v=(r,dH,1-i)$ with $i\leq r+1$ and if $n=1$, assume that $\v \neq (2d-1,dH,1)$. Then $h^1(X,E)=0$ 
    and $$h^1(X,E(pH))=\max\{0,r-2pdn-(p^2n+1)(i-1)\} \quad \mbox{for} \quad 0<p<d.$$
      Moreover, if $\v^2\geq 0$, then $h^1(X,E(dH))=0$.  If $\v^2=-2$, then $h^1(X,E(dH))=1$ and $h^1(X,E((d+1)H))=0$.  Finally,
    if $n=1$ and $\v=(2d-1,dH,1)$, then $H^1(X,E(pH))=0$ for all $p\geq 0$.
    \item If $\v=(r,dH,2-i)$ with $i\leq r+2$, then $h^1(X,E)=0$ except if $n=1$ and $\v=(5,3H,2)$.  Moreover, if $\v$ is not one of the vectors listed in \cref{Lem:NoTSSWalls a=2}, we have 
    $$h^1(X,E(pH))=\max\{0,r-2pdn-(p^2n+1)(i-2)\} \quad \mbox{for} \quad 0 < p < \frac{d}{2},$$
    except when $n=1$, $\v=(\frac{d^2+1}{2},dH,2)$ and  $p=\frac{d-3}{2}$.  In this case, $h^1(X,E(\left(\frac{d-3}{2}\right)H))=8$.  
    \begin{enumerate}
    \item If $n=1$ and $\v=(5,3H,2)$, then  $h^1(X,E)=1$, $h^1(X,E(H))=3$, and $h^1(X,E(pH))=0$ for $p\geq 2$. 
    \item If $\v$ is one of the vectors in cases (1)(c), (2)(a),(b) and (c), and (3) of \cref{Lem:NoTSSWalls a=2} and  $\v \neq (5,3H,2)$ when $n=1$, then $H^1(X,E(pH))=0$ for $p\geq 0$.  
    \item In case (1)(a), $h^1(X,E(H))=5$ and $h^1(X,E(pH))=0$ for $p\geq 2$. 
    \item In case (1)(b), $h^1(X,E(H))=13$, $h^1(X,E(2H))=5$, and $h^1(X,E(pH))=0$ for $p\geq 3$. 
    \item In case (1)(d), $h^1(X,E(H))=6$ and $h^1(X,E(pH))=0$ for $p\geq 2$.
    \end{enumerate}
\end{enumerate}
\end{Thm}
\begin{proof}
Let us prove (1) first.  By \cite[Theorem 0.1]{KY11}, $\Phi_{X\to X}^{I_\Delta[1]}$ induces a birational map $$M_H(\v)\dashrightarrow M_H(i,dH,-r),$$ so by the proof of \cref{Lem:FM-vanishing}, the generic $E\in M_H(\v)$ fits into a short exact sequence of sheaves \begin{equation}\label{eqn:a<=0}
    0\to \OO_X^{\oplus(r-i)}\to E\to\Phi_{X\to X}^{I_\Delta}(E)[1]\to 0.
\end{equation}
Setting $F=\Phi_{X\to X}^{I_\Delta}(E)[1]$ and taking cohomology, there is a short exact sequence $$0\to H^1(X,E)\to H^1(X,F)\to \C^{\oplus(r-i)}\to 0.$$  Since $h^1(X,F)=-\chi(F)=r-i$ by \cref{Prop:a<=0},  $H^1(X,E)=0$ as claimed.  

Tensoring \eqref{eqn:a<=0} by $\OO_X(pH)$ ($p\geq 1$) and taking cohomology, we see that $h^1(E(pH))=h^1(F(pH))$.  Note that $$\chi(F(pH))=i-r+2pdn+p^2ni=-r+2pdn+i(p^2n+1),$$ so if $i<i_0:=\max\{\lceil\frac{r-2pdn}{p^2n+1}\rceil,0\}$, then $\chi(F(pH))<0$ and thus $$h^1(X,E(pH))=h^1(X,F(pH))=-\chi(F(pH))=r-2pdn-i(p^2n+1)>0$$ by \cref{Prop:a<=0}.

Next we show that if $i\geq i_0$, then $h^1(X,E)=0=\max\{0,r-2pdn-(p^2n+1)i\}$ by induction on $r$.  Observe first that $i_0<r$ since $\frac{r-2pdn}{p^2n+1}\leq\frac{r}{2}\leq r-1$.   By \cref{prop:minimal-a}, it suffices to consider the case $i=i_0$.  Tensoring \eqref{eqn:a<=0} by $\OO_X(pH)$ ($p\geq 1$) and taking cohomology, we see that $h^1(X,E(pH))=h^1(X,F(pH))$.  By induction, the generic $F'\in M_H(i_0,dH,-i_0')$ satisfies $H^1(X,F'(pH))=0$, where $i_0':=\max\{\lceil\frac{i_0-2pdn}{p^2n+1}\rceil,0\}<i_0$.  From the definition of $i_0$, $\chi(F(pH))\geq 0$, so applying \cref{prop:minimal-a} to $F(pH))$ which is generic in its moduli space, we get that $H^1(X,F(pH))=0$ as well.  This gives $H^1(X,E(pH))=0$ as required.

Now we prove (2).  By part (1), it suffices to consider  the case $i=0$.  
By \cref{lem:NoTSSWalls a<=1},  $D_\v=\varnothing$ unless $n=1$ and $\v=(r,\left(\frac{r+1}{2}\right)H,1)$, in which case $D_\v=\{(2,H,1)\}$. In this case,   $\v_1=(2,H,1)$ is the unique effective spherical class $\v_1$ in the primitive isotropic sublattice $\HHH$ defined by the wall $C^\v_{\v_1}$.  
By \cref{lem:max-HNF2}, the generic $E\in M_H(\v)$ fits into a short exact sequence 
\begin{equation}\label{eqn:exceptionalcase}
0\to T_1\to E\to F\to 0,
\end{equation}
 where $T_1$ is the unique stable spherical sheaf with $\v(T_1)=\v_1$ and $F$ is a $\sigma$-stable object such that $\v(F)=(r-2,\left(\frac{r-1}{2}\right)H,0)$.  
 By \cref{Prop:a<=0}, $D_{\v(F)}=\varnothing$.  Thus the generic $F\in M_\sigma(\v(F))$ is a stable sheaf with vanishing $H^1$ by \cref{thm:summing up inequalities}.  
 As $H^1(X,T_1)=0$ by \cref{prop:rankstwoandthree}, we get $H^1(X,E)=0$.  As $\frac{d}{r}=\left(\frac{r+1}{r}\right)\left(\frac{1}{2}\right)>\frac{1}{2}\geq\frac{p}{p^2+1}$ for $p\geq 1$, it follows from \cref{prop:Tensor} that $h^1(X,E(pH))=0$ for all $p\geq 1$.
 
 Now we study the case where $D_\v=\varnothing$.  It follows that $\Phi_{X\to X}^{I_\Delta}(E)^{\vee}\in M_H(1,dH,r)$.  Applying $\Hom(\OO(-pH),\blank)$ to the distinguished triangle 
 $$\Phi_{X\to X}^{I_\Delta}(E)\to R\Gamma(X,E)\otimes\OO_X\to E,$$
 and using Serre duality, we get 
 \begin{align*}
 h^1(X,E(pH))&=h^0(X,\Phi_{X\to X}^{I_\Delta}(E)^\vee(-pH))=\max\{\chi(\Phi_{X\to X}^{I_\Delta}(E)^\vee(-pH)),0\}\\
 &=\max\{0,r+(p^2n+1)-2pdn\}
 \end{align*}
 if $0<p<d$, as required.  Suppose that $p=d$.  Then if $\v^2\geq 0$, then $\Phi_{X\to X}^{I_\Delta}(E)^\vee\ne\OO_X(dH)$, so $h^1(X,E(dH))=0$.  If, on the other hand, $\v^2=-2$ so that $\Phi_{X\to X}^{I_\Delta}(E)^\vee=\OO_X(dH)$, then $h^1(X,E(dH))=1$ while $h^1(X,E((d+1)H))=0$.
 
Finally, we prove (3). By parts (1) and (2), it suffices to consider $i=0$.  If $\v$ is not one of the exceptional cases in \cref{Lem:NoTSSWalls a=2}, then $D_\v=\varnothing$, so for generic $E\in M_H(\v)$ we have  $F:=\Phi_{X\to X}^{I_\Delta}(E)^\vee\in M_H(2,dH,r)$.  Hence as long as $F(-pH)\ne(2,3H,5)$ when $n=1$, it follows that
\begin{align*}
    h^1(X,E(pH))&=h^0(F(-pH))=\max\{0,\chi(F(-pH))\}\\
    &=\max\{0,r-2pdn+2(p^2n+1)\}
\end{align*}
for $0<p<\frac{d}{2}$ by \cref{prop:rankstwoandthree}.  

It only remains to consider one of the exceptional cases in \cref{Lem:NoTSSWalls a=2} when $D_\v\ne\varnothing$.

In cases (1)(c), (2)(a),(b) and (c), and (3), the generic $E\in M_H(\v)$ satisfies $H^1(X,E(pH))=0$ for $p\geq 0$ by \cref{Lem:NoTSSWalls a=2} and \cref{prop:Tensor} except when $n=1$ and $\v=(5,3H,2)$.  In this case, the unique $E\in M_H(5,3H,2)$  sits in the short exact sequence \eqref{eqn:(5,3,2)}
$$0\to\OO_X\to T_1^{\oplus 3}\to E\to 0,$$
where $T_1$ is the unique element of $M_H(2,H,1)$.  From this we see that $$h^1(X,E(pH))=3h^1(X,T_1(pH))$$ for every $p\geq 1$.  So $h^1(X,E)=1$, $h^1(X,E(H))=3$, and $h^1(X,E(pH))=0$ for $p\geq 2$.

It remains to consider cases (1)(a),(b), and (d).  It follows from \cref{prop:Tensor} that the generic $E\in M_H(\v)$ satisfies $H^1(X,E(pH))=0$ for $p\geq 2$ in cases (1)(a) and (d), and $p\geq 3$ in case (1)(b).  Consider (1)(a) first.  Then $\langle\v,\v_1\rangle=-1$, so the generic $E\in M_H(11,5H,2)$ sits in a short exact sequence
$$0\to T_1\to E\to F\to 0,$$
where $F\in M_\sigma(6,3H,1)$ is generic and $M_H(5,2H,1)=\{T_1\}$.  By \cref{lem:NoTSSWalls a<=1}, as $F$ is generic, we must have $F\in M_H(6,3H,1)$.  Moreover, it follows that there is a short exact sequence
$$0\to\Phi_{X\to X}^{I_\Delta}(F)^\vee\to\Phi_{X\to X}^{I_\Delta}(E)^\vee\to\Phi_{X\to X}^{I_\Delta}(T_1)^\vee\to 0,$$
where $\Phi_{X\to X}^{I_\Delta}(T_1)^\vee\cong\OO_X(2H)$ and $\Phi_{X\to X}^{I_\Delta}(F)^\vee\in M_H(1,3H,6)$.  Thus
$$h^1(X,E(H))=h^0(X,\Phi_{X\to X}^{I_\Delta}(E)^\vee(-H))=h^0(X,\OO_X(H))+h^0(X,\Phi_{X\to X}^{I_\Delta}(F)^\vee(-H))=3+2=5.$$

Now consider (1)(d).  Then $\langle\v,\v_1\rangle=-2$, so the generic $E\in M_H(12,5H,2)$ sits in a short exact sequence
$$0\to T_1^{\oplus 2}\to E\to F\to 0,$$
where $F\in M_\sigma(2,H,0)$ is generic and $M_H(5,2H,1)=\{T_1\}$.  By \cref{Prop:a<=0}, we must have $F\in M_H(2,H,0)$.  Moreover, it follows that there is a short exact sequence
$$0\to\Phi_{X\to X}^{I_\Delta}(F)^\vee\to\Phi_{X\to X}^{I_\Delta}(E)^\vee\to(\Phi_{X\to X}^{I_\Delta}(T_1)^\vee)^{\oplus 2}\to 0,$$
where $\Phi_{X\to X}^{I_\Delta}(T_1)^\vee\cong\OO_X(2H)$ and $\Phi_{X\to X}^{I_\Delta}(F)^\vee\in M_H(0,H,2)$.  Thus
$$h^1(X,E(H))=h^0(X,\Phi_{X\to X}^{I_\Delta}(E)^\vee(-H))=2h^0(X,\OO_X(H))+h^0(X,\Phi_{X\to X}^{I_\Delta}(F)^\vee(-H))=6+0=6.$$ 

Finally, we consider case (1)(b).  Then $\langle\v,\v_1\rangle=-1$, so the generic $E\in M_H(23,7H,2)$ sits in a short exact sequence
$$0\to T_1\to E\to F\to 0,$$
where $F\in M_\sigma(13,4H,1)$ is generic and $M_H(10,3H,1)=\{T_1\}$.  By \cref{lem:NoTSSWalls a<=1}, we must have $F\in M_H(13,4H,1)$.  Moreover, it follows that there is a short exact sequence
$$0\to\Phi_{X\to X}^{I_\Delta}(F)^\vee\to\Phi_{X\to X}^{I_\Delta}(E)^\vee\to\Phi_{X\to X}^{I_\Delta}(T_1)^\vee\to 0,$$
where $\Phi_{X\to X}^{I_\Delta}(T_1)^\vee\cong\OO_X(3H)$ and $\Phi_{X\to X}^{I_\Delta}(F)^\vee\in M_H(1,4H,13)$.  Thus
$$h^1(X,E(pH))=h^0(X,\Phi_{X\to X}^{I_\Delta}(E)^\vee(-pH))=h^0(X,\OO_X((3-p)H))+h^0(X,\Phi_{X\to X}^{I_\Delta}(F)^\vee(-pH)).$$
This gives $h^1(X,E(H))=13$ and $h^1(X,E(2H))=5$ for generic $E\in M_H(23,7H,2)$.
\end{proof}
\begin{Prop}\label{Prop:d=r/n+1}
Let $\v=(r,dH,a)$ with $r,a\geq 0$, $d>0$, and $\v^2\geq -2$.  Suppose that $d=\frac{r}{n+1}$ and assume that there is a semistable sheaf $E$ with $\v(E)=\v$.
\begin{enumerate}
    \item If $a>\left(\frac{n}{n+1}\right)d$, then $a=d=1$, and $M_H(\v)=\{E_0\}$, in which case $H^1(X,E_0(pH))=0$ for $p=0$ and $p\geq 2$, while $h^1(X,E_0(H))=1$.
    \item If $0< a\leq\left(\frac{n}{n+1}\right)d$, then for generic $E\in M_H(\v)$, $h^1(X,E(pH))=0$ for $p=0$ and $p\geq 2$, while $h^1(X,E(H))=\max\{(n+1)a-(n-1)d,0\}$.
\end{enumerate}
\end{Prop}
\begin{proof}
If $\v^2=-2$, then it is easy to see that $a=d=1$.  Otherwise, $0\leq\v^2=2d(nd-(n+1)a)$ which is equivalent to $a\leq\left(\frac{n}{n+1}\right)d$.

Thus if $a>\left(\frac{n}{n+1}\right)d$, then $a=d=1$, as claimed, and in this case $\v=\v_0=(n+1,H,1)$ and $\v^2=-2$, so $M_H(\v)=\{E_0\}$.  \cref{lem:NoTSSWalls a<=1} implies that $h^1(X,E_0)=0$, and \cref{prop:Tensor} implies that $h^1(X,E_0(pH))=0$ for $p\geq 2$.  The unique sheaf $E_0(H)\in M_H(n+1,(n+2)H,n^2+3n+1)$ satisfies $h^1(X,E_0(H))=1$ as it sits in a short exact sequence 
$$0\to\OO_X\to\OO_X(H)^{\oplus n+2}\to E_0(H)\to 0.$$

Otherwise, we may assume that $0< a\leq \left(\frac{n}{n+1}\right)d$.  Observe that we can write 
$$\v=d\v_0-b(0,0,1),$$ where $\v_0=(n+1,H,1)$ and $b\in\Z$ satisfies $0<b=d-a<d$.  Now the unique $E_0\in M_H(\v_0)$ satisfies $h^1(X,E_0)=0$ from the previous paragraph.  Moreover, we see that $a\leq\left(\frac{n}{n+1}\right)d$ is equivalent to $(n+1)b\geq d$, so we may apply \cref{Lem:E_0} and \cref{prop:E_0} to see that the generic $E\in M_H(\v)$ is on the same irreducible component as the kernel of a generic quotient 
$$f\colon E_0^d\onto\bigoplus_{i=1}^b k_{x_i},$$ which satisfies $\Hom(E_0,\Ker f)=0$ and $H^1(X,\Ker f)=0$.  Thus $H^1(X,E)=0$ by semicontinuity.  \cref{prop:Tensor} and its proof then imply that 
$$h^1(X,E(H))=\dim\Hom(E,E_0)=\max\{0,(n+1)a-(n-1)d\},\qquad h^1(X,E(pH))=0,p\geq 2.$$
\end{proof}

\section{The Mukai vectors of rank at most 20 violating weak Brill-Noether}\label{sec-rank20}
Let $X$ be a K3 surface such that $\Pic(X) = \mathbb{Z}H$ with $H^2 =2n$. 
In this section, we list the Mukai vectors $\v= (r, dH, a)$ with $d > 0$, $\v^2 \geq -2$ and $0 \leq r \leq 20$ such that the moduli space $M_X(\v)$ does not satisfy weak Brill-Noether. By \cref{lem:NoTSSWalls a<=1}, we may also assume that $a\geq 2$. Throughout the section, let $E \in M_H(\v)$ be a generic sheaf. We also compute $h^1(E)$ and record the Mukai vector $\v_1$ that defines the largest totally semistable (TSS) wall.

First, we have the following five families:

\begin{enumerate}
\item    Let $r_1$ be an integer that divides $n+1$. Let $$\v = \left(n+r_1^2, \left(\frac{n+1}{r_1} + r_1 \right)H, \left(\frac{n+1}{r_1}\right)^2 + n \right).$$ Then the largest TSS wall is given by $\v_1 = \left(r_1, H, \frac{n+1}{r_1}\right)$ and $h^1(E)=1$ (see \cref{Cor:n+rsquare}).
\item Let $0<p, j$ and $0 \leq i \leq r$ be three integers. Let
 $$\v = \left(r, (rp+j)H, np^2 r +2njp -i\right).$$ Then by \cref{Prop:CohomologyCalc a<=2}, the largest TSS wall is given by $$\v_1=\left(1, pH, np^2+1\right) \quad \mbox{and} \quad h^1(E)= \max\left(0, r-2npj-(np^2+1)i\right).$$ 
 
 \item Let $0< p \leq j $ and $0 \leq i \leq r+1$ be three integers. If $n=1$, assume that $(r,j,i) \neq (2j-1,j,1)$. Let
 $$\v = \left(r, (rp+j)H, np^2 r +2njp+1 -i\right).$$ If $p<j$, then by \cref{Prop:CohomologyCalc a<=2}, the largest TSS wall is given by $$\v_1=\left(1, pH, np^2+1\right) \quad \mbox{and} \quad h^1(E)= \max\left(0, r-2npj-(np^2+1)(i-1)\right).$$  If $p=j$ and  $\v^2 = -2$, then $h^1(E)=1$.
 
 \item Let $0<p < j/2$ and $0 \leq i \leq r+2$ be three integers. 
Let
 $$\v = \left(r, (rp+j)H, np^2 r +2njp+2 -i\right).$$ Other than the exceptional cases when $n=1$ that are described in \cref{Exceptional counter examples}, by \cref{Prop:CohomologyCalc a<=2}, the largest TSS wall is given by  $$\v_1=\left(1, pH, np^2+1\right) \quad \mbox{and} \quad h^1(E)= \max\left(0, r-2npj-(np^2+1)(i-2)\right).$$ 
  \begin{table}[ht]
    \centering
    \begin{tabular}{|c|c|c|}
    \hline
      $\v$ & $\v_1$ &  $h^1(E)$  \\
      \hline 
 $\left(\frac{d^2+1}{2},\left(d+\left(\frac{d^2+1}{2}\right)\left(\frac{d-3}{2}\right)\right)H,2+d(d-3)+\left(\frac{d^2+1}{2}\right)\left(\frac{d-3}{2}\right)^2\right)$&$\left(1,\left(\frac{d-3}{2}\right)H,\left(\frac{d-3}{2}\right)^2+1\right)$ & 8\\
\hline 
$(5,8H,13)$ &$(2,3H,5)$ & 3\\
\hline
 $(11,16H,23)$ & $(1,H,2)$ & 5\\
\hline
 $(23,30H,39)$ &$(10,13H,17)$ & 13\\
\hline
 $(23,53H,122)$ &$(10,23H,53)$ & 5\\
\hline
$(12,17H,24)$ &$(5,7H,10)$ & 6\\
\hline
\end{tabular}
    \caption{Exceptional counterexamples when $n=1$}
    \label{Exceptional counter examples}
    \end{table}
 \item Suppose that $n+1$ divides $r$, and let $a$ be a positive integer such that $a\leq \frac{rn}{(n+1)^2}$.  Let $$\v=\left(r,\left(r+\frac{r}{n+1}\right)H,a+rn+\frac{2rn}{n+1}\right).$$
 Then by \cref{Prop:d=r/n+1}, the largest TSS wall is given by $$\v_1=(1,H,n+1)\quad \mbox{and}\quad h^1(E)=\max\left(0,(n+1)a-\frac{(n-1)r}{n+1}\right).$$ 
 \end{enumerate} 
 
Outside of these five families, the Mukai vectors with $r \leq 20$ that are counterexamples to weak Brill-Noether are listed in \cref{remainingcounterexamples}. To compute the list, we fixed the rank and let a computer list the finitely many potential counterexamples guaranteed by \cref{thm:finiteness}.  We made the computations faster by using \cref{prop:minimal-a} to assume that $a=\left\lfloor\frac{nd^2+1}{r}\right\rfloor$, and then we had a computer find the finitely many solutions to the inequalities in \cref{thm:summing up inequalities}.  The full list of potential counterexamples for $2\leq r\leq20$ with maximal $a$ is available on the second author's website.\footnote{\url{https://drive.google.com/file/d/1_LE3IjdF1lX8ce0c4b-ls_0K636KI6Jp/view?usp=sharing}}  We then used the Harder-Narasimhan filtration along the largest TSS wall to calculate the cohomology of the generic sheaf.   

In \cref{remainingcounterexamples}, we include $D_{\v}$, the cohomology of the generic sheaf and an explanation of how we calculate this cohomology. Given the space constraint in the table, let us elaborate what we mean in the ``Reason'' column by means of some representative examples.  They come in three flavors.

   \begin{table}[!htbp]
    \centering
    \begin{tabular}{|c|c|c|c|c|}
    \hline
      $n$ & $\v$ & $D_\v$ &  $h^1(E)$ & Reason \\
\hline
1 & $(11,6,3)$ & $(2,1,1)$ &1&$E_{2,1,1}^5\into E\onto I_{4pt}(H)$\\ \hline 
1 & $(11,17,26)$ & $\{(1,1,2),(2,3,5)\}$ &3&\cref{prop:Tensor}+$E_{2,3,5}^5\into E\onto I_{4pt}(2H)$\\ \hline
2 & $(11,15,41)$ & $\{(1,1,3),(3,4,11)\}$ &4&$\OO(H)\into E_{3,4,11}^4\onto E$\\ 
\hline 1&$(12,7,4)$&$(2,1,1)$&2&$E_{2,1,1}^6\into E\onto\OO_H(L)$\\
\hline 1&$(12,19,30)$&$\{(1,1,2),(2,3,5)\}$&6& $E_{2,3,5}^6\into E\onto \OO_H(L)$ with $\chi(L)=0$\\ 
\hline 1&$(13,8,5)$&$\{(2,1,1),(5,3,2)\}$&3&$E_{2,1,1}^7\into E\onto \OO(-H)[1]$\\ 
\hline 1&$(13,21,34)$&$\{(1,1,2),(2,3,5),(5,8,13)\}$&8&$E_{2,3,5}^7\into E\onto \OO[1]$\\
\hline 1&$(15,8,4)$&$(2,1,1)$&2&$E_{2,1,1}^7\into E\onto I_{5pt}(H)$\\ 
\hline 1&$(15,9,5)$&$(2,1,1)$&1&$E_{2,1,1}^7\into E\onto I_{7pt}(2H)$\\ 
\hline 1&$(15,22,32)$&$(1,1,2)$&7&\cref{prop:Tensor}\\ 
\hline  1&$(15,23,35)$&$\{(1,1,2),(2,3,5)\}$&6&\cref{prop:Tensor} implies $h^1(E)\leq 6$\\ 
 &&&&$E_{2,3,5}^7\into E\onto I_{5pt}(2H)\Rightarrow h^1(E)\geq 6$\\ 
\hline 1&$(15,24,38)$&$\{(1,1,2),(2,3,5)\}$&3&\cref{prop:Tensor} implies $h^1(E)\leq 3$\\
&&&&$E_{2,3,5}^7\into E\onto I_{7pt}(3H)\Rightarrow h^1(E)\geq 3$\\  
\hline 1&$(16,9,5)$&$(2,1,1)$&3&$E_{2,1,1}^8\into E\onto \OO_H(L)$ with $\chi(L)=-3$\\ 
\hline 3&$(16,9,15)$&$(2,1,2)$&1&$E_{2,1,2}^8\into E\onto\OO_H(L)$ with $\chi(L)=-1$\\ 
\hline 1&$(16,23,33)$&$(1,1,2)$&8&\cref{prop:Tensor}\\
\hline 1&$(16,25,39)$&$\{(1,1,2),(2,3,5)\}$&9&$E_{2,3,5}^{8}\into E\onto \OO_H(L)$, $\chi(L)=-1$\\ 
\hline 2&$(16,21,55)$&$(1,1,3)$&5&\cref{prop:Tensor}\\
\hline 3&$(17,9,14)$&$(2,1,2)$&1&$E_{2,1,2}^8\into E\onto I_{6pt}(H)$\\ 
\hline 2&$(17,12,17)$&$(3,2,3)$&1&$E_{3,2,3}^6\into E\onto \OO[1]$\\ 
\hline 1&$(17,25,36)$&$(1,1,2)$&7&\cref{prop:Tensor}\\ 
\hline 2&$(17,22,57)$&$(1,1,3)$&6&\cref{prop:Tensor}\\ 
\hline 1&$(18,10,5)$&$(2,1,1)$&1&$E_{2,1,1}^8\into E\onto E_{2,2,-3}$\\ 
\hline 1&$(18,26,37)$&$\{(1,1,2),(5,7,10)\}$&8&\cref{prop:Tensor}\\
\hline 1&$(18,28,43)$&$\{(1,1,2),(2,3,5)\}$&3&\cref{prop:Tensor} implies $h^1(E)\leq 3$\\ 
&&&&$E_{2,3,5}^8\into E\onto E_{2,4,3}\Rightarrow h^1(E)\geq 3$\\
\hline 1&$(19,10,5)$&$(2,1,1)$&3&$E_{2,1,1}^9\into E\onto I_{6pt}(H)$\\ 
\hline  1&$(19,11,6)$&$(2,1,1)$&2&$E_{2,1,1}^9\into E\onto I_{8pt}(2H)$\\ 
\hline 
 1&$(19,27,38)$&$\{(1,1,2),(5,7,10)\}$&9&\cref{prop:Tensor}\\ 
 \hline 1&$(19,28,41)$&$(1,1,2)$&9&\cref{prop:Tensor}\\ 
 \hline 1 &$(19,28,40)$&$(1,1,2)$ &7&\cref{prop:Tensor}\\
\hline  1&$(19,29,44)$&$\{(1,1,2),(2,3,5)\}$&9&\cref{prop:Tensor} implies $h^1(E)\leq 9$\\ &&&&$E_{2,3,5}^9\into E\onto I_{6pt}(2H)\Rightarrow h^1(E)\geq 9$\\
\hline 1&$(19,30,47)$&$\{(1,1,2),(2,3,5)\}$&6&\cref{prop:Tensor} implies $h^1(E)\leq 6$\\
&&&&$E_{2,3,5}^9\into E\onto I_{8pt}(3H)\Rightarrow h^1(E)\geq 6$\\
\hline 2&$(19,25,65)$&$(1,1,3)$&4&\cref{prop:Tensor}  \\ 
\hline 3&$(19,24,91)$&$\{(1,1,4),(4,5,19)\}$&5&$E_{4,5,19}^5\into E\onto \OO(H)[1]$\\
\hline 1&(20,11,6)& (2,1,1)&4 &$E_{2,1,1}^{10}\into E\onto \OO_H(L)$, $\chi(L)=-4$\\
\hline 1&(20,12,7)&(2,1,1)& 3&$E_{2,1,1}^{10}\into E\onto\OO_{2H}(L)$, $\chi(L)=-3$\\
\hline 3&(20,11,18)& (2,1,2)&2&$E_{2,1,2}^{10}\into E\onto \OO_H(L)$, $\chi(L)=-2$\\ 
\hline 1&(20,28,39)& (1,1,2)&10&\cref{prop:Tensor}\\ 
\hline 1&(20,29,42)&(1,1,2)&10&\cref{prop:Tensor}\\
\hline 1&(20,29,41)&(1,1,2) &8& \cref{prop:Tensor}\\
\hline 1&(20,31,48)& $\{(1,1,2),(2,3,5)\}$&12&$E_{2,3,5}^{10}\into E\onto \OO_H(L)$, $\chi(L)=-2$\\
\hline 1&(20,32,51)&$\{(1,1,2),(2,3,5)\}$&9&\cref{prop:Tensor} implies $h^1(E)\leq 9$\\
&&&&$E_{2,3,5}^{10}\into E\onto \OO_{2H}(L)$ and $\chi(L)=1$ imply $9\leq h^1(E)$\\
\hline 2&(20,26,67)&(1,1,3)&5&\cref{prop:Tensor}\\ 
\hline 1&(20,48,115)&(1,2,5)&3&\cref{prop:Tensor}\\
\hline 
 \end{tabular}
    \caption{The table of exceptional counterexamples}
    \label{remainingcounterexamples}
    \end{table}

    First, when we only list a short exact sequence as the reason, we are indicating which $\v_1\in D_\v$ defines the largest TSS wall, the corresponding short exact sequence in $\AA_\sigma$, and that taking the long exact sequence of cohomology is enough to calculate the cohomology.  For example, when $n=1$ and $\v=(11,6H,3)$, $D_\v$ consists of the single Mukai vector $\v_1=(2,H,1)$.  The spherical twist induced by $\v_1$ gives $R_{\v_1}(\v)=(1,H,-2)$.  Since $\frac{a_1 d-ad_1}{r_1d-rd_1}>1$, the wall induced by $\v_1$ is above the wall induced by $\OO_X[1]$, which is the only TSS wall for $(1,H,-2)$ by \cref{Prop:a<=0}.  Thus for the generic ideal sheaf $I_{4pt}$ of a $0$-dimensional subscheme of length $4$, the twist $I_{4pt}(H)$ is still stable for $\sigma$ along the wall determined by $\v_1$.  Thus for generic $E\in M_H(\v)$, its Harder-Narasimhan filtration across the wall is given as in \cref{remainingcounterexamples}, where $E_{2,1,1}$ denotes the unique stable vector bundle in $M_H(\v_1)$.  As $\v_1=(2,H,1)$ and $(1,H,-2)$ satisfy weak Brill-Noether and the latter Mukai vector has $\chi<0$, when we take the long exact sequence of cohomology we get 
$$0=H^1(X,E_{2,1,1})^5\to H^1(X,E)\to H^1(X,I_{4pt}(H))\to H^2(X,E_{2,1,1})^5=0,$$
giving $h^1(X,E)=h^1(X,I_{4pt}(H))=1$, as claimed.  When all vectors in $D_\v$ give the same wall, we indicate which resolution we choose to use to calculate cohomology.

    In the next flavor of exceptional counterexample, we similarly give the largest TSS wall, which has the given form by the same reasoning.  However, taking the long exact sequence on cohomology only gives a lower bound for $h^1(X,E)$:
$$h^1(X,E)\geq h^1(X,E(-pH))\chi(\OO_X(pH)).$$  Writing $d=pr+d_0$ with $0<d_0<r$, we have $d_0>\frac{r}{n+1}$ in these examples, so $\hom(E(-pH),F_p)=0$ by stability, where $F_p\cong\Phi_{X\to X}^{I_\Delta}(\OO_X(pH))^\vee$ is the unique stable sheaf of Mukai vector $(p^2n+1,pH,1)$.  It then follows from \cref{prop:Tensor} that $h^1(X,E)\leq h^1(X,E(-pH))\chi(\OO_X(pH))$, so we get equality and thus the values given in \cref{remainingcounterexamples}. 

The final flavor of examples use \cref{prop:Tensor} in a complementary way.  Here we indicate \cref{prop:Tensor} alone as the reason.  This is because we have already shown in our check that for generic $E\in M_H(\v)$, $H^1(X,E(-pH))=0$ for the same $p$ as above.  Thus \cref{prop:Tensor} tells us that $h^1(X,E)=\hom(E(-pH),F_p)$, and by \cite{BM14b}, for generic $E\in M_H(\v)$, 
$$\hom(E(-pH),F_p)=-\langle\v e^{-pH},(p^2n+1,pH,1)\rangle,$$
giving the results in the table.

Let us summarize how to use our table and classification.  Given a Mukai vector $\v=(r,dH,a)$ with $1\leq r\leq 20$ and $\v^2\geq -2$, if $(n,\v)$ appear together in \cref{remainingcounterexamples}, then the cohomology of the generic $E\in M_H(\v)$ is given there.  If not, one checks if $\v$ falls into one the five families enumerated above, in which case we give a formula for the cohomology.  In all other cases, $\v$ satisfies weak Brill-Noether.

\bibliographystyle{plain}
\bibliography{NSF_Research_Proposal}
\end{document}